\documentclass {article}

\usepackage {amsmath, amssymb, amsthm, amscd, thmtools}
\usepackage{geometry}
\usepackage{cancel}
\usepackage{graphicx}
\usepackage{tikz-cd}
\usepackage{soul} %strike-out, under-line
\usepackage{enumitem}
\usepackage{mathrsfs}

\usepackage{thm-restate}

\newcommand{\ulindex}[1]{\ul{#1}\index{#1}}
\makeindex

\newcommand {\fsum} {\displaystyle \sum} % Fancy sum
\newcommand {\fprod}{\displaystyle \prod} % Fancy product
\newcommand{\ra} {\rightarrow}

\newcommand{\bea}{\begin{eqnarray*}}
\newcommand{\eea}{\end{eqnarray*}}
\newcommand{\inv}{^{-1}}

\newcommand{\ds}{\displaystyle}

\newcommand{\N}{\mathbb{N}}
\newcommand{\Z}{\mathbb{Z}}

\newcommand{\C}{\mathbb{C}}
\newcommand{\R}{\mathbb{R}}

\newcommand{\T}{\mathbb{T}}

\newcommand{\closedspan}{\overline{\spanset}}
\newcommand{\norm}[1]{\left\|{#1}\right\|}
\newcommand{\vertiii}[1]{{\left\vert\kern-0.25ex\left\vert\kern-0.25ex\left\vert #1 
    \right\vert\kern-0.25ex\right\vert\kern-0.25ex\right\vert}}

\usepackage{imakeidx}
\makeindex

\usepackage{hyperref} %If hyperref is loaded after imakeidx, the index will be clickable. To enable this, comment out the first \usepackage{hyperref}
\usepackage{lineno}
\makeatletter
\def\showlineno{line \the\inputlineno}
\renewcommand\LineNumber{\the\inputlineno}
\definecolor{linenogrey}{RGB}{150,150,150}

\DeclareMathOperator{\id}{id}

\DeclareMathOperator{\spanset}{span}
\DeclareMathOperator{\supp}{supp}

\newtheorem{theorem}{Theorem}[section]
\newtheorem{lemma}[theorem]{Lemma}
\newtheorem{definition}[theorem]{Definition}
\newtheorem{question}[theorem]{Question}

\newtheorem{corollary}[theorem]{Corollary}
\newtheorem{remark}[theorem]{Remark}
\newtheorem{notation}[theorem]{Notation}
\newtheorem{example}[theorem]{Example}
\newtheorem{proposition}[theorem]{Proposition}

\newtheorem{conjecture}[theorem]{Conjecture}

\author {Robert Huben}
\date{March 2021}
\title{Gauge-Invariant Uniqueness and \\ Reductions of Ordered Groups}
\begin{document}

\maketitle

\begin{abstract}
A reduction $\varphi$ of an ordered group $(G,P)$ to another ordered group is an order homomorphism which maps each interval $[1,p]$ bijectively onto $[1, \varphi(p)]$. We show that if $(G,P)$ is weakly quasi-lattice ordered and reduces to an amenable ordered group, then there is a gauge-invariant uniqueness theorem for $P$-graph algebras. We also consider the class of ordered groups which reduce to an amenable ordered group, and show this class contains all amenable ordered groups and is closed under direct products, free products, and hereditary subgroups.

\end{abstract}

\tableofcontents

\section{Introduction}

Cuntz-Krieger algebras and their generalizations (Exel-Laca algebras, graph algebras, higher-rank graph algebras, etc) are all, broadly speaking, the universal algebras generated by partial isometries whose range and source projections satisfy certain combinatorial relations. But defining these algebras by their universality comes at a cost, since it becomes difficult to check if any particular collection of partial isometries (a ``representation'') is universal. Mathematicians responded with uniqueness theorems, conditions on the representation which guarantee that it is universal. A classic example is the \emph{gauge-invariant uniqueness theorem} for graph algebras, which states that so long as the canonical generators of the algebra are nonzero and there is a gauge action (meaning an action $\alpha$ of the circle satisfying $\alpha_z(T_\gamma)= z^{\text{length}(\gamma)} T_\gamma$ for any $z \in \T$ and path $\gamma$ in the graph), then any other representation of the graph is a quotient of this representation. This construction and uniqueness theorem has been generalized to higher rank graphs by Kumjian and Pask \cite{kumjian2000higher}, where paths are given lengths in $\N^k$ instead of $\N$, and an action of $\T^k = \widehat{\Z^k}$ replaces the gauge action of $\T = \widehat{\Z}$.

One might then ask that paths be allowed to have ``lengths'' in any positive cone $P$ of a group $G$. The first attempt at such a generalization was by Brownlowe, Sims, and Vittadello \cite{brownlowe2013co}, who studied $P$-graphs when $(G,P)$ is quasi-lattice ordered, but surprisingly the authors construct a $P$-graph algebra which is not universal but \emph{co-}universal, meaning it is a quotient of any sufficiently large representation.

In Section 4 of this paper, we build on their work to show that if the grading group $(G,P)$ is weakly quasi-lattice ordered and has a certain kind of quotient map called a reduction such that the quotient is amenable, this algebra is both universal and co-universal. This allows us to generalize the gauge-invariant uniqueness theorems for graphs and $k$-graphs (such as \cite[Theorem 2.2]{raeburn2005graph} or \cite[Theorem 3.4]{kumjian2000higher}) to make a gauge-invariant uniqueness theorem for $P$-graphs:

\begin{restatable*}{theorem}{uniquenesstheorem}
\label{uniquenesstheorem}
Let $(G,P)$ be a weakly quasi-lattice ordered group which reduces to an amenable ordered group, and let $\Lambda$ be a finitely-aligned $P$-graph. Then there is exactly one representation (up to canonical isomorphism) of $\Lambda$ which is $\Lambda$-faithful, tight, and has a gauge coaction of $G$. This representation is universal for tight representations and co-universal for representations which are $\Lambda$-faithful and have a gauge coaction by $G$. 
\end{restatable*}

This diagram provides a visual summary of the result:

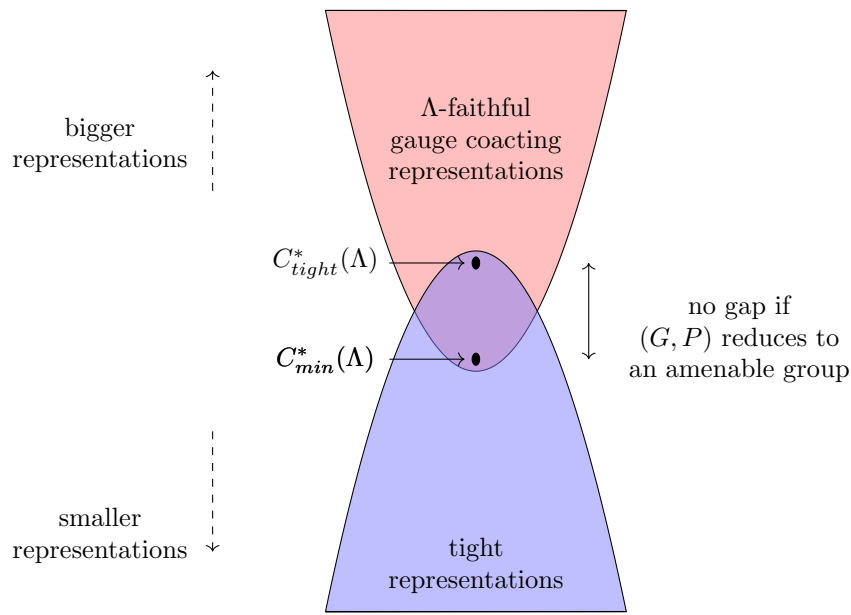
\begin{center}
\begin{tikzpicture}[yscale=.8, xscale=.5]
\draw[fill=red!50!white, fill opacity=.5] plot[color=red,smooth,samples=100,domain=-4:4] (\x,{3/8*\x*\x-1}) -- plot[color=red,smooth,samples=100,domain=-4:4] (\x,{5});
\draw[fill=blue!50!white, fill opacity=.5] plot[color=blue,smooth,samples=100,domain=-4:4] (\x,{-3/8*\x*\x+1}) -- plot[color=blue,smooth,samples=100,domain=-4:4] (\x,{-5});

\node [label={[align=center]$\Lambda$-faithful \\ gauge coacting\\ representations}] at (0,1.8) {};

\node [label={[align=center]tight \\ representations}] at (0,-5) {};

    \draw node at (-4,-0.8) {$C^*_{min}(\Lambda)$};
    
    \draw [fill] (0,0.8) circle(1 mm);
    \draw [->] (-2.3, 0.8) -- (-0.3, 0.8);
    \draw node at (-4,0.8) {$C^*_{tight}(\Lambda)$};

    \draw [fill] (0,-0.8) circle(1 mm);
    \draw [->] (-2.3, -0.8) -- (-0.3, -0.8);
    \draw node at (-4,-0.8) {$C^*_{min}(\Lambda)$};

    \draw [<->, ] (3, -0.8) -- (3, 0.8);
    \node [label={[align=center]no gap if \\ $(G,P)$ reduces to \\ an amenable group }] at (7,-1.5) {};

    \draw [->, dashed] (-7, 2) -- (-7, 4);
    \node [label={[align=center]bigger \\ representations}] at (-10,2) {};

    \draw [->, dashed] (-7, -2) -- (-7, -4);
    \node [label={[align=center]smaller \\ representations}] at (-10,-4.5) {};

\end{tikzpicture}

\end{center}

We believe that the algebra generated by this simultaneously universal and co-universal representation deserves the title of \emph{the} Cuntz-Krieger algebra of the $P$-graph. 

The notion of a reduction of ordered groups is introduced in Section 3, culminating in this theorem showing that several natural operations on ordered groups preserve the existence of a reduction onto an amenable group:

\begin{restatable*}{theorem}{AmenableReductionTheorem}
\label{AmenableReductionTheorem}
The class of ordered groups which have (strong) reductions onto amenable groups contains all amenable ordered groups and is closed under hereditary subgroups, finite direct products, and finite free products.
\end{restatable*}

In Theorem \ref{uniquenesstheorem}, being reducible to an amenable group plays the role of amenability in guaranteeing a unique representation, but by Theorem \ref{AmenableReductionTheorem}, this condition is, in at least one sense, more robust than amenability. In particular, \emph{being reducible to} an amenable group is preserved under free products, whereas \emph{being} amenable is almost always destroyed by a free product (for instance if $G$ and $H$ contain copies of $\Z$, then $G*H$ is not amenable). 

The notion of being reducible to an amenable group is sufficiently robust that, as corollary to Theorem \ref{AmenableReductionTheorem}, the important example $(\Z^2 *\Z, \N^2 *\N)$ from the literature (\cite{spielberg2007graph, brownlowe2013co}) reduces to an amenable group. In Proposition \ref{Kirchberg UCT Algebras Result}, we use this fact in combination with the results of \cite{brownlowe2013co} to show that every Kirchberg algebra in the UCT class is stably isomorphic to the simultaneously universal and co-universal algebra of a $(\N^2*\N)$-graph. 

Besides the results listed in Theorem \ref{AmenableReductionTheorem}, a great deal is not yet known about how reductions interact with other group constructions. The limits of our understanding are discussed in Section \ref{sectionSummary}.

We hope that this new approach will provide a useful tool for the analysis of $P$-graphs and their algebras, and possibly have implications for product systems and Fell bundles.

\section{Background}

\subsection{Ordered Groups}

Ordered groups have been studied in both a context of group theory and operator theory, and we largely follow the conventions of \cite{brownlowe2013co} and \cite[Chapter 32]{exel2017partial}. The reader should note that other authors, such as \cite{clay2016ordered}, use the term ``ordered group'' to refer to a group with a total order, but we follow the convention that the group has a partial order.

\begin{definition}
A \ulindex{positive cone} in a group $G$ is a submonoid $P$ such that $P \cap P \inv = \{1\}$. 

A \ulindex{left-invariant partial order} on a group $G$ is a partial order $\leq$ such that for all $a,b,c \in G$, $a \leq b$ if and only if $ca \leq cb$. 

Every left-invariant partial order on a group arises naturally from a positive cone $P$ by saying that $a \leq b$ if and only if $a \inv b \in P$. We will say that an \ulindex{ordered group} $(G,P)$ is a group along with a positive cone, with the implication that the group takes on a left-invariant partial order in this way.

A group is \ulindex{totally ordered} if its partial order is a total order. This is equivalent to the condition that $P \cup P \inv=G$. 
\end{definition}

Although our theory of reductions does not require any additional structure on our ordered groups, the relators for a graph algebra will ask for an additional property on our orderings from \cite[Chapter 32]{exel2017partial}:

\begin{definition}
We say $(G,P)$ is \ulindex{weakly quasi-lattice ordered} (WQLO)\index{WQLO} if whenever $x,y \in P$ have a common upper bound in $P$, they have a supremum (least common upper bound), denoted $x \vee y$. Note that any upper bound of a positive element is itself positive by transitivity.
\end{definition}

While many mathematicians (including\cite{brownlowe2013co, sehnem2019c, spielberg2014groupoids}) worked in the context of quasi-lattice ordered groups, we have found that the slightly more general condition of weak quasi-lattice order suffices. The condition of weak quasi-lattice order is also friendlier than quasi-lattice order since it is entirely a property of the submonoid $P$, and therefore one can ``forget'' the ambient group. The same is not true for quasi-lattice order. For a more thorough introduction to quasi-lattice order, weak quasi-lattice order, and their connections, the reader is directed to \cite[Chapter 32]{exel2017partial}.

\begin{notation}
In this work, we will make use of the following notation:

\begin{itemize}
\item  $\N =\{0,1,2,3...\}$ includes $0$.
\item Given a group $G$ and subset $S \subseteq G$, $\langle S\rangle$ will denote the group generated by the elements of $S$, and $S^*$ will denote the monoid generated by the elements of $S$.
\item The identity element in a multiplicative group $G$ will be denoted by $1_G$, or $1$ if the group is clear from context. Less commonly we will denote the identity element by $e$.
\end{itemize}\end{notation}

\begin{example}
Let $G=\Z^k$, and let $P=\N^k$. Then $(G,P)$ is WQLO. Every pair of elements $p=(p_1,...,p_k)$ and $q=(q_1,...,q_k)$ have a supremum $p \vee q = (\max(p_1,q_1), ..., \max(p_k,q_k))$. (In fact, this ordering is a ``lattice order'', meaning any two elements have both a supremum and an infimum.)

Let $G=F_k=\langle a_1,...,a_k\rangle$ be the free group on $k$ generators, and let $P=\{a_1,...,a_k\}^*$ be the free monoid generated by the generators of $G$. Then $(G,P)$ is WQLO, since any $p,q \in P$ have a common upper bound if and only if $p \leq q$ or $q \leq p$, in which case $p \vee q =\max\{p,q\}$.
 
\end{example}

\begin{notation}
Let $(G,P)$ be an ordered group. \index{$[a,b]$} \index{$[1,p]$}We use the following standard notation for intervals: for $a,b \in G$, write 
\[ [a,b] := \{g \in G | a \leq g \leq b \}.\]
\end{notation}

It is immediate that such an interval is nonempty if and only if $a \leq b$. We will most often be interested in intervals of the form $[1,p]$ for some $p \in P$.

\begin{definition}
Let $(G,P)$ and $(H,Q)$ be ordered groups. We say that $\varphi:G \ra H$ is an \ulindex{order homomorphism} if it is a group homomorphism with $\varphi(P) \subseteq Q$. We will write $\varphi:(G,P) \ra (H,Q)$ to denote an order homomorphism.
\end{definition}

Order homomorphisms preserve both the group structure and the order structure, as the following lemma shows:

\begin{lemma} \label{Order Homomorphism}
Let $(G,P)$ and $(H,Q)$ be ordered groups, and $\varphi:G \ra H$ a group homomorphism. Then $\varphi$ is an order homomorphism if and only if $x \leq y$ implies $\varphi(x) \leq \varphi(y)$ for all $x,y \in G$.

In particular, if $\varphi$ is an order homomorphism and $p \in P$, then $\varphi([1,p]) \subseteq [1, \varphi(p)]$.
\end{lemma}

\begin{proof}
If $\varphi$ is an order homomorphism and $x \leq y$, then $x\inv y \in P$, so $\varphi(x\inv y)=\varphi(x) \inv \varphi(y) \in Q$. Thus $\varphi(x) \leq \varphi(y)$.

Conversely, if $\varphi$ is not an order homomorphism, then there is $p \in P$ such that $\varphi(p) \not \in Q$. Then $1_G \leq p$, and $1_H=\varphi(1_G) \not \leq \varphi(p)$, as desired.

For the ``in particular'', if $s \in [1,p]$, then $1 \leq s \leq p$, so $1= \varphi(1) \leq \varphi(s) \leq \varphi(p)$, so $\varphi(s) \in [1, \varphi(p)]$ as desired.
\end{proof}

\subsection{$C^*$-algebras}

\subsubsection{The Factors Through Theorem}
The following result is an elementary notion from algebra, which we state here for completeness since we will make frequent use of it:

\begin{lemma}[Factors Through Theorem]\label{Factors Through Theorem}
Let $A,B$, and $C$ be sets, let $f:A \ra B$ and $g:A \ra C$ be functions, and suppose that $f$ is surjective.

\begin{center}
\begin{tikzcd}
A \arrow[r, two heads, "f"] \arrow [d, "g"]& 
B \arrow[ld, dashed, "\exists h"]\\
C
\end{tikzcd}
\end{center}

\begin{enumerate}

\item Suppose that whenever $f(a)=f(a')$, then $g(a)=g(a')$. Then there is a ``bonding map'' $h:B \ra C$ such that $h \circ f = g$. In particular, $h$ is (well-)defined by $h(f(a))=g(a)$ for all $f(a) \in B$.

\item If $A, B,$ and $C$ are groups, $f$ and $g$ are group homomorphisms, and $\ker f \subseteq \ker g$, then hypothesis of (1) that $f(a)=f(a')$ implies $g(a)=g(a')$ is satisfied. Therefore, the conclusion of (1) is true.

\item If $A,B,C$ have a binary operation $\cdot$ (respectively a unary operation $^*$)  which is preserved by $f$ and $g$, then $\cdot$ (respectively $^*$) is also preserved by $h$.
\end{enumerate}

\end{lemma}

%The proof of this result is straightforward and left to the reader.

\begin{proof}
For (1), for all $b\in B$, since $f$ is surjective, there is some $a \in A$ such that $f(a)=b$. Now, define $h(b)=g(a)$, which is well-defined since if there were some other $a' \in A$ with $f(a')=b$, then $g(a)=g(a')$ by hypothesis. That is, we have defined $h$ by $h(f(a))=g(a)$, so $h \circ f =g$.

For (2), if $f(a)=f(a')$, then $a\inv a' \in \ker f \subseteq \ker g$, so $g(a)=g(a')$, as desired.

For (3), let $b_1, b_2 \in B$. Since $f$ is surjective, there is some $a_1, a_2 \in A$ such that $f(a_1)=b_1, f(a_2)=b_2$. Then 

\bea
 h(b_1 \cdot b_2) &=& h(f(a_1) \cdot f(a_2))= h(f(a_1 \cdot a_2))= g(a_1 \cdot a_2)\\
&=& g(a_1) \cdot g(a_2)=h(f(a_1)) \cdot h(f(a_2))=h(b_1) \cdot h(b_2)
\eea

\noindent and

\[ h(b_1^*)= h(f(a_1)^*)= h(f(a_1^*))= g(a_1^*)= g(a_1)^* = h(f(a_1))^*=h(b_1)^*\]

\noindent  as desired.

\end{proof}

\subsubsection{Universal Algebras}

The following construction of a ``universal algebra'' is a slight simplification of the one outlined by Blackadar in \cite{blackadar1985shape}. We state it here for completeness.

\begin{lemma} [Construction of a Universal Algebra, \cite{blackadar1985shape}] 
\label{Universal Algebra}
Let $\mathcal G=\{x_\alpha\}_{\alpha \in I}$ denote a set of formal symbols, $\mathcal G^*=\{x_\alpha^*\}_{\alpha \in I}$ another set of formal symbols over the same indexing set, and $\mathcal R$ a set of relators of the form $p(x_{\alpha_1},...,x_{\alpha_n}, x_{\alpha_1}^*,...,x_{\alpha_n}^*)=0$ where $p$ is some polynomial in $2n$ non-commuting variables and complex coefficients. 

\sloppy We define a \ul{representation}\index{representation!of arbitrary generators and relators} of $(\mathcal G, \mathcal R)$ to be a collection of elements $\{y_\alpha\}_{\alpha \in I}$ in a $C^*$-algebra that satisfy $p(y_{\alpha_1},...,y_{\alpha_n}, y_{\alpha_1}^*,...,y_{\alpha_n}^*)=0$ for each relator in $\mathcal R$.

Suppose that the relators on $(\mathcal G, \mathcal R)$ imply that in any representation, the images of the generators are partial isometries. Then there is a representation $\{z_\alpha\}_{\alpha \in I}$ such that for any other representation $\{y_\alpha\}_{\alpha \in I}$ there is a surjective $*$-homomorphism $\pi: C^*(\{z_\alpha\}_{\alpha \in I}) \ra C^*(\{y_\alpha\}_{\alpha \in I})$ satisfying $z_\alpha \mapsto y_\alpha$. 

The existence of this surjective $*$-homomorphism is called the \ulindex{universal property} of $C^*(\{z_\alpha\}_{\alpha \in I})$. We call this representation the \ulindex{universal representation} and say that $C^*(\{z_\alpha\}_{\alpha \in I})$ is the \ulindex{universal algebra} for $(\mathcal G, \mathcal R)$.

\end{lemma}

\begin{proof}
Let $\mathcal F(\mathcal G)$ denote the free $*$-algebra over $\C$ generated by $\C$ and $\mathcal G \cup \mathcal G^*$ where the $*$-map is given by $(x_\alpha)^*=x_\alpha^*$ and $(x_\alpha^*)^*= x_\alpha$ and extended to the entire algebra via conjugate-linearity and $(ab)^*=b^*a^*$. Note that any representation $\varphi: x_\alpha \mapsto y_\alpha$ of $(\mathcal G, \mathcal R)$ extends uniquely to a $*$-algebra homomorphism $\bar \varphi: \mathcal F(\mathcal G)\ra C^*(\{y_\alpha\}_{\alpha \in I})$.

For $X \in \mathcal F(\mathcal G)$, let $\vertiii{X}= \sup \{  \norm{ \bar \varphi(X)} : \varphi \text{ a rep. of } (\mathcal G, \mathcal R) \}$. Note that since $(\mathcal G, \mathcal R)$ can be represented by the zero map, this supremum is bounded below by 0, and since each generator must be represented by a partial isometry, the image of each generator has norm 1, and thus the supremum is bounded above by the sum of the absolute values of its $\C$-coefficients. Thus $\vertiii{X}$ is a well-defined, non-negative finite number. 

Let $J=\{X \in \mathcal F(\mathcal G) : \vertiii{X}=0\}$, and note that $J= \ds \bigcap_{\varphi \text{ a rep.}} \ker \bar \varphi$ where the intersection is taken over all representations. Thus $J$ is the intersection of $*$-ideals, so it is a $*$-ideal. Now let $\psi: \mathcal F(\mathcal G) \ra  \mathcal F(\mathcal G)/J$ denote the quotient map, and for all $\alpha$, let $z_\alpha= \psi(x_\alpha)$. For any relator $p(x_{\alpha_1},...,x_{\alpha_n}, x_{\alpha_1}^*,...,x_{\alpha_n}^*)=0 \in \mathcal R$, since this relator is satisfied for every representation, then $\vertiii{p(x_{\alpha_1},...,x_{\alpha_n}, x_{\alpha_1}^*,...,x_{\alpha_n}^*)}=0$, and thus $p(x_{\alpha_1},...,x_{\alpha_n}, x_{\alpha_1}^*,...,x_{\alpha_n}^*) \in J$, so $p(z_{\alpha_1},...,z_{\alpha_n}, z_{\alpha_1}^*,...,z_{\alpha_n}^*)=0$. That is, the $\{z_\alpha\}_{\alpha \in I}$ satisfy the relators.

Now note that $\vertiii{\cdot}$ is a $C^*$-seminorm, so the norm $\norm{X+J} := \vertiii{X}$ is a well-defined $C^*$-norm on $\mathcal F(\mathcal G)/J$, and thus the completion of $\mathcal F(\mathcal G)/J$ with respect to this norm is a $C^*$-algebra. Thus $\psi: x_\alpha \ra z_\alpha$ is a representation of $(\mathcal G, \mathcal R)$.

Now, given any other presentation $\varphi: x_\alpha \mapsto y_\alpha$, note that $\ker \bar \varphi = \{X \in \mathcal F(\mathcal G) : \norm{\bar \varphi(X)}=0\} \subseteq \{X \in \mathcal F(\mathcal G) : \vertiii{X}=0\}= \ker \bar \psi$, so by Lemma \ref{Factors Through Theorem}(2), there is a function on the $*$-algebras generated by $y$ and $z$ given by $y_\alpha \mapsto z_\alpha$, and by Lemma \ref{Factors Through Theorem}(3) applied to $+, \times, $ and $^*$, this function is a $*$-homomorphism, so it extends to a $*$-homomorphism $\pi: C^*(\{z_\alpha\}_{\alpha \in I}) \ra C^*(\{y_\alpha\}_{\alpha \in I})$ as desired.
\end{proof}

\subsubsection{Conditional Expectations}

We will also make frequent use of conditional expectations, whose definition we include here. For a more thorough treatment of conditional expectations, the reader is directed to \cite[Chapter III.3]{TakesakiThOpAlI}, where they are called projections of norm one.

\begin{definition}
Let $C$ be a $C^*$-algebra, and $D \subseteq C$ a $C^*$-subalgebra. We say a \ulindex{conditional expectation} or \ulindex{projection of norm one} is a linear map $\Phi: C \ra D$ such that:

\begin{itemize}
%\item $\Phi$ is \ul{positive}\index{positive linear transformation}, meaning $\Phi(c) \geq 0$ if $c \geq 0$.
\item $\Phi$ is \ulindex{contractive}, meaning $\norm{\Phi(c)} \leq \norm{c}$ for all $c \in C$ (and is thus continuous).
\item $\Phi$ is \ulindex{idempotent}, meaning $\Phi \circ \Phi= \Phi$.
\item $\Phi(d)=d$ for all $d \in D$.
%\item $\Phi$ is a $D$-bimodule map, meaning $\Phi(d_1cd_2)=d_1 \Phi(c) d_2$ for all $d_1,d_2 \in D$ and $c \in C$.
\end{itemize}

If additionally $c>0$ implies $\Phi(c)>0$, we say $\Phi$ is \ul{faithful}\index{conditional expectation!faithful}\index{faithful!conditional expectation}.
\end{definition}

\begin{example}
In a matrix algebra $M_n$, there is a faithful conditional expectation $\Phi$ given by ``restricting to the diagonal'', meaning $\Phi([a_{ij}])=[b_{ij}]$ where 

\[ b_{ij}= \begin{cases} a_{ij} & \text{ if } i=j \\ 0 & \text{ if } i \neq j \end{cases}. \]
\end{example}

The following result is \cite[Theorem 1]{tomiyama1957projection} and is known as Tomiyama's Theorem. It says that a conditional expectation has some additional properties ``for free''. Some authors will define conditional expectations as requiring these properties, and some will define them to only require the shorter list of conditions (as we have done here).

\begin{theorem}[Tomiyama's Theorem]
If $\Phi:C \ra D$ is a conditional expectation, then $\Phi$ is positive, meaning that $x \geq 0$ implies $\Phi(x) \geq 0$\index{positive linear transformation}. Furthermore, $\Phi$ is a $D$-bimodule map, meaning $\Phi(d_1cd_2)=d_1 \Phi(c) d_2$ for all $d_1,d_2 \in D$ and $c \in C$.
\end{theorem}

\subsubsection{Tensor Products}

Here we will give a very brief introduction to the minimal tensor product. A more thorough introduction can be found in \cite[Chapter 4]{TakesakiThOpAlI}. Throughout this paper, as in much of the coaction literature, we will use unadorned $\otimes$ for the minimal tensor product of $C^*$-algebras.

\begin{definition}
Given two vector spaces $V$ and $W$, we let $V \odot W$ denote the \ul{algebraic tensor product}\index{tensor product!algebraic} of $V$ and $W$, meaning the $\C$-vector space spanned by formal symbols $\{v \otimes w : v \in V, w \in W\}$ with the relations of linearity over either coordinate. 

Let $A$ and $B$ be $C^*$-algebras. Then one can give $A \odot B$ a $*$-algebra structure by 

\[ (a \otimes b)(a' \otimes b')=(aa') \otimes (bb') \text{ and }(a \otimes b)^*=a^* \otimes b^*. \] 

If $\pi:A \ra \mathcal B(H)$ and $\rho: B \ra \mathcal B(\mathcal K)$ are representations of $A$ and $B$ on Hilbert spaces $\mathcal H$ and $\mathcal K$ respectively, then there exists a $*$-algebra representation $\pi \otimes \rho:A \odot B \ra \mathcal B(\mathcal H \otimes \mathcal K)$ satisfying $\left[(\pi \otimes \rho)(a \otimes b)\right](h \otimes k)= \pi(a)h \otimes \rho(b)k$ for all $a \in A, b \in B, h \in \mathcal H, k \in \mathcal K$. The \ulindex{minimal norm} on $A \odot B$ is defined by

\[ \left| \left| \fsum_{i=1}^n a_i \otimes b_i \right| \right| _{min}= \sup \left\{ \left |\left | \fsum_{i=1}^n \pi(a_i) \otimes \rho(b_i) \right| \right | \right\} \]

\noindent where the supremum is taken over all representations $\pi$ and $\rho$ of $A$ and $B$, respectively (and this is indeed a norm which is minimal in an appropriate sense). Note that $\norm{a \otimes b}_{min}=\norm{a}\cdot \norm{b}$ (that is, $\norm{\cdot}_{min}$ is a ``$C^*$-cross norm''). The \ul{minimal tensor product}\index{tensor product!minimal} $A \otimes_{min}B$ is the $C^*$-algebra created by completing $A \odot B$ with respect to this norm. Since all of our tensor products will be minimal, we will henceforth write $\norm{\cdot}$ for $\norm{\cdot}_{min}$ and $A \otimes B$ for $A \otimes_{min} B$.

\end{definition}

The following result will be used occasionally, and it is a minor variant of \cite[Lemma A.1]{echterhoff2006categorical}:

\begin{lemma}\label{Tensor of Homomorphisms}
Let $\varphi:A \ra C$ and $\psi: B \ra D$ be $*$-homomorphisms of $C^*$-algebras. Then there is a homomorphism $\varphi \otimes \psi: A \otimes B \ra C \otimes D$ satisfying $(\varphi \otimes \psi)(a \otimes b) =\varphi(a) \otimes \psi(b)$. If $\varphi$ and $\psi$ are nondegenerate (respectively, faithful), then so is $\varphi \otimes \psi$.
\end{lemma}

\begin{proof}
Consider $C$ and $D$ as subalgebras (respectively) of their multiplier algebras $M(C)$ and $M(D)$, and therefore consider $\varphi$ and $\psi$ as mapping into $M(C)$ and $M(D)$, respectively. By \cite[Lemma A.1]{echterhoff2006categorical}, there is a $*$-homomorphism $\varphi \otimes \psi: A \otimes B \ra M(C \otimes D)$ satisfying $(\varphi \otimes \psi)(a \otimes b) =\varphi(a) \otimes \psi(b)$, and if $\varphi$ and $\psi$ are nondegenerate (respectively, faithful), then so is $\varphi \otimes \psi$. 

It then suffices to show that $\varphi \otimes \psi$ maps into $C \otimes D$ properly, instead of mapping into $M(C \otimes D)$. But for all $a \in A, b \in B$, we have $\varphi(a) \in C$ and $\varphi(b) \in D$, so $(\varphi \otimes \psi)(a \otimes b) \in C \otimes D$ for all $a \in A, b \in B$. Taking closed spans, we get that for any $x \in A \otimes B$, we have $(\varphi \otimes \psi)(x) \in C \times D$, as desired.
\end{proof}

\subsection{Connections between Groups and $C^*$-Algebras}

In this section we will remind the reader of some of the many connections between groups and $C^*$-algebras, including group $C^*$-algebras, amenability, Fell bundles, gradings, coactions, and actions. Our context is relatively simple since we are limiting our attention to discrete groups.

\subsubsection{Group $C^*$-algebras and Amenable Groups}
The following discussion of group $C^*$-algebras and amenability is adapted from \cite[Chapter VII]{davidson1996c}, to which the reader is directed for a more thorough treatment.

\begin{definition}
Let $G$ be a discrete group. Then a \ul{representation}\index{representation!of a group} of $G$ in a $C^*$-algebra is a collection of unitary operators $\{u_g\}_{g\in G}$ such that $u_gu_h=u_{gh}$ and $u_g^* =u_{g\inv}$ for all $g, h \in G$. 

Since these relators are all polynomials, then by Lemma \ref{Universal Algebra} there is a representation of $G$ which is universal for all representations. We denote it by $\{U_g\}_{g \in G}$, and call $C^*(G) := C^*(\{U_g\}_{g \in G})$ the \ul{full group $C^*$-algebra}\index{group $C^*$-algebra!full}\index{C@$C^*(G)$}.

A particularly nice representation of a group $G$, called the \ul{left-regular representation of a group}\index{left-regular representation!of a group}, comes from the natural action of $G$ multiplying on itself. In particular, let $\{e_g\}_{g \in G}$ denote the standard orthonormal basis for $\ell^2(G)$, and define an operator $L_g$ by $L_g e_{g'}=e_{gg'}$ for $g, g' \in G$. Then $\{L_g\}_{g\in G}$ is a representation of $G$ in $\mathcal B( \ell^2(G))$. We call $C_r^*(G) := C^*(\{L_g\}_{g\in G})$ the \ul{reduced group $C^*$-algebra}\index{group $C^*$-algebra!reduced}\index{C@$C^*_r(G)$}. 

\end{definition}

The left-regular representation is more commonly denoted by $\lambda$ instead of $L$, but our choice of notation will help avoid ambiguity with paths $\lambda \in \Lambda$ later in the text.

Note that $C^*(G)$ and $C^*_r(G)$ are the closed spans of $\{U_g\}_{g\in G}$ and $\{L_g\}_{g \in G}$, respectively.

Famously, there are many equivalent definitions of amenability, but for our purposes, this one is the most convenient:

\begin{definition} \label{Amenable Group Definition}
By the universal property of $C^*(G)$, there is a \sloppy surjective $*$-homomorphism $\pi^U_L: C^*(G) \ra C^*_r(G)$ given by $U_g \mapsto L_g$. We say that $G$ is \ul{amenable}\index{amenable!group} if and only if $\pi^U_L$ is injective (and hence an isomorphism).
\end{definition}

Readers more familiar with another definition of amenability may wish to read \cite[Theorem VII.2.5]{davidson1996c} which proves the equivalence of this definition with a more common one.

The following remark summarizes some of the well-known results about amenability of groups.

\begin{remark}
By \cite[Proposition VII.2.3]{davidson1996c}, for discrete groups amenability is preserved under:

\begin{enumerate}
\item Subgroups
\item Quotients
\item Direct Limits
\item Extensions (meaning that if $1 \ra N \ra G \ra H \ra 1$ is a short exact sequence of groups, with $N$ and $H$ amenable, then $G$ is amenable).
\item Finite direct products (which is immediate from being closed under extensions).
\end{enumerate}

Furthermore, by \cite[Proposition VII.2.2]{davidson1996c}, every abelian group is amenable. Every finite group is amenable since $C^*(G)$ and $C^*_{r}(G)$ have the same finite dimension $|G|$, so by the rank-nullity theorem the surjective map $\pi^U_L$ is injective.

Notably, the free group on two generators is not amenable by \cite[Example VII.2.4]{davidson1996c}. By the subgroup property, any group containing a free group on two (or more) generators is also not amenable.
\end{remark}

\subsubsection{Fell Bundles, Topologically Graded $C^*$-algebras, and Coactions}

In this section, we will give a short introduction to Fell bundles, topologically graded $C^*$-algebras, and coactions, and show some of the ways these closely-related structures overlap. For a more detailed treatment, the reader is directed to \cite{exel2017partial,exel1996amenability} for Fell bundles, \cite{exel2017partial} for topological gradings, and \cite[Appendix A]{echterhoff2006categorical} for coactions.

Throughout this section, all of our groups will be discrete, which simplifies some definitions. The following is from \cite[Definition 16.1]{exel2017partial}:

\begin{definition}
Let $G$ be a discrete group. Let $\mathcal B= \{B_g\}_{g \in G}$ be a collection of Banach spaces, and write $\mathscr B$ for the disjoint union of the $\{B_g\}_{g\in G}$, called the \ulindex{total space}. Suppose $\mathscr B$ has a binary operation $\cdot$ called multiplication, and an involution $*$ which satisfy the following properties for all $g, h \in G$ and $b,c \in \mathscr B$:

\begin{enumerate}[label=\alph*.]
\item $B_g B_h \subseteq B_{gh}$,
\item Multiplication is bilinear from $B_g \times B_h$ to $B_{gh}$,
\item Multiplication on $\mathscr B$ is associative,
\item $\norm{bc} \leq \norm{b}\cdot \norm{c}$,
\item $(B_g)^* \subseteq B_{g\inv}$,
\item Involution is conjugate-linear from $B_g$ to $B_{g\inv}$,
\item $(bc)^* =c^*b^*$,
\item $b^{**}=b$,
\item $\norm{b^*}=\norm{b}$,
\item $\norm{b^*b} =\norm{b}^2$,
\item $b^*b \geq 0$ in $B_1$.
\end{enumerate}

Then we say that $\mathcal B$ is a \ulindex{Fell bundle} over $G$. We call each $B_g$ a \ulindex{fiber}.
\end{definition}

The following is from \cite[Definitions 16.2 and 19.2]{exel2017partial}:

\begin{definition}
Let $A$ be a $C^*$-algebra, and let $G$ be a (discrete) group. We say that a  \ul{($C^*$-) grading}\index{grading} for $A$ is a collection $\{A_g\}_{g \in G}$ of linearly independent closed subspaces such that $\bigoplus_{g \in G} A_g$ is dense in $A$, $A_g A_h \subseteq A_{gh}$, and $A_g^* \subseteq A_{g\inv}$ for $g,h \in G$. Each $A_g$ is called a \ulindex{graded subspace} or \ulindex{graded component}. 

If there is also a conditional expectation $\Phi: A \ra A_e$ satisfying 

\[\Phi(a)=\begin{cases} a &\text{ if }a \in A_e \\ 0 & \text{ if } a\in A_g \text{ for }g\neq e \end{cases},\]

\noindent then we say that $\{A_g\}_{g\in G}$ is a \ul{topological grading}\index{grading!topological}.
\end{definition}

\begin{remark}\label{Graded algebra is a Fell Bundle}
It is straightforward to verify that if $A$ is a topologically graded $C^*$-algebra, then its graded components form a Fell bundle. Most of the Fell bundles we will use arise in this way.

\end{remark}

\begin{remark} \label{Fell Bundle Conditional Expectation}
Given a Fell bundle $\mathcal B$, one can construct a ``reduced'' and ``full'' cross sectional algebra representing this bundle, respectively denoted $C^*_r (\mathcal B)$ and $C^*(\mathcal B)$ (see Definition 2.3 of \cite{exel1996amenability} and the comment following it). 

By \cite[Theorem 3.3]{exel1996amenability}, if $B$ is a topologically graded $C^*$-algebra, and $\mathcal B$ denotes its associated Fell bundle, then there is a $*$-homomorphism $L: B \ra C^*_r(\mathcal B)$ called the \ul{left-regular representation} of the Fell bundle\index{left-regular representation!of a Fell bundle}. Combining this result with their Proposition 3.7, $L$ is an isomorphism if and only if the conditional expectation from the topological grading is faithful.

In a Fell bundle $\mathcal B$, the full cross sectional algebra $C^*(\mathcal B)$ has a topological grading, and the bundle is called \ul{amenable}\index{amenable!Fell bundle} if its left-regular representation $L:C^*(\mathcal B) \ra C^*_r(\mathcal B)$ is injective (and hence an isomorphism). The reader may notice the parallel with our definition of amenability for a group (Definition \ref{Amenable Group Definition}). 

\end{remark}

The following two results are respectively \cite[Theorem 4.7]{exel1996amenability} and \cite[Proposition 4.2]{exel1996amenability} :

\begin{lemma} \label{Fell bundle over amenable group}
Let $G$ be a discrete amenable group. Then every Fell bundle over $G$ is amenable.
\end{lemma}

\begin{lemma} \label{Amenable Fell Bundle Uniqueness}
If $\mathcal B$ is an amenable Fell bundle, then all topologically graded $C^*$-algebras whose associated Fell bundles coincide with $\mathcal B$ are isomorphic to each other.
\end{lemma}

Finally, we will give a brief introduction to coactions. The following is both the simplest example of a coaction, and a necessary component of its definition.

\begin{example} \label{Delta_G}
If $G$ is a discrete group, then the operators $\{U_g \otimes U_g\}_{g\in G} \subset C^*(G) \otimes C^*(G)$ are a representation of $G$. Therefore, by the universal property of $C^*(G)$, there is a $*$-homomorphism $\delta_G: C^*(G) \ra C^*(G) \otimes C^*(G)$ given by $U_g \mapsto U_g \otimes U_g$.

Note that letting $\id_G:=\id_{C^*(G)}$, then $(\delta_G \otimes \id_{G}) \circ \delta_G = (\id_{G}\otimes \delta_G) \circ \delta_G$, which can be easily verified by checking that both send $U_g \mapsto U_g \otimes U_g \otimes U_g \in C^*(G) \otimes C^*(G) \otimes C^*(G)$.
\end{example}

The following definition of coactions for discrete groups is taken from \cite[Section 2]{echterhoff1999induced}:

\begin{definition}\label{Coaction Definition}
Let $G$ be a discrete group. A \ulindex{coaction} of $G$ on a $C^*$-algebra $A$ is an injective, nondegenerate homomorphism $\delta:A \ra A \otimes C^*(G)$ satisfying the \ulindex{coaction identity} that $(\delta \otimes \id_G) \circ \delta= (\id_A \otimes \delta_G) \circ \delta$ as maps from $A$ into $A \otimes C^*(G) \otimes C^*(G)$, summarized in this diagram:

\begin{center}
\begin{tikzcd}
A \arrow[r, "\delta"] \arrow [d, "\delta"]& 
A \otimes C^*(G) \arrow[d, "\delta \otimes \id_G"]\\
A \otimes C^*(G) \arrow [r, "\id_A \otimes \delta_G"]&
A \otimes C^*(G) \otimes C^*(G) \\
\end{tikzcd}
\end{center}

In this context, nondegeneracy means that $\closedspan \left[ \delta(A) (A\otimes C^*(G)) \right]= A \otimes C^*(G)$.

We call the triple $(A, G, \delta)$ a \ulindex{cosystem}.

\end{definition}

If $G$ is discrete, a cosystem has a nice topological grading as described in \cite[Proposition A.3]{sehnem2019c}:

\begin{lemma} \label{Graded Conditional Expectation}
Let $(A, G, \delta)$ be a discrete cosystem, and for $g \in G$ let 

\[ A_g= \{a \in A | \delta(a) = a \otimes U_g\}.\]

Then the collection $\{A_g\}_{g \in G}$ is a topological grading of $A$. We will write $\Phi_A:A \ra A_e$ for the conditional expectation.
\end{lemma}

Some coactions are particularly nice, having a condition called normality. There are many definitions of normality, but we will find a definition depending on this conditional expectation to be the most convenient:

\begin{definition} \label{Normal Coaction Definition}
We say a discrete coaction $(A, G, \delta)$ is \ul{normal}\index{normal coaction} if the conditional expectation $\Phi_A$ is faithful.
\end{definition}

A more common definition of normality of a coaction (such as \cite[Definition A.50]{echterhoff2006categorical} or the comments preceding Definition 1.1 of \cite{quigg1996discrete}) is that the coaction is normal if and only if the map $j_A:= (\id_A \otimes \pi^U_L) \circ \delta$ is injective. Our definition is proved equivalent to the more common one in \cite[Lemma 1.4]{quigg1996discrete}.

For a coaction over a discrete amenable group, this property is automatic:

\begin{lemma}\label{Amenable Coactions are Normal}
Let $(A,G, \delta)$ be a cosystem. If $G$ is amenable and discrete, then $\delta$ is normal.
\end{lemma}

\begin{proof}
For a cosystem $(A,G, \delta)$ over an amenable group, by Lemma \ref{Graded Conditional Expectation} there is a topological grading $\{A_g\}_{g\in G}$ of $A$. By Remark \ref{Graded algebra is a Fell Bundle}, these graded components form a Fell bundle $\mathcal B$, and this Fell bundle is amenable by Lemma \ref{Fell bundle over amenable group}.

Since $\mathcal B$ is amenable, by Lemma \ref{Amenable Fell Bundle Uniqueness}, $A \cong C^*_{r}(\mathcal B)$. But for any Fell bundle, $C^*_r(\mathcal B)$ has a conditional expectation by \cite[Proposition 2.9]{exel1996amenability}, which is faithful by \cite[Proposition 2.12]{exel1996amenability}.
\end{proof}

Finally, we will show that coactions, topological gradings, and Fell bundles are equivalent over amenable groups.

\begin{lemma}\label{Coaction Grading Fell Equivalence}
Let $G$ be an amenable group, and let $A$ be a $C^*$-algebra. Then the following are equivalent:

\begin{enumerate}
\item $A$ has a coaction by $G$.
\item $A$ has a topological grading $\{A_g\}_{g \in G}$.
\item There is a Fell bundle $\mathcal B= \{B_g\}_{g \in G}$ whose fibers are linearly independent closed subspaces of $A$ such that $A \cong C^*(\mathcal B)$.
\end{enumerate}

And under these conditions, the structures coincide in the sense that the fibers $\{B_g\}_{g\in G}$ of the Fell bundle are equal to the graded components $\{A_g\}_{g \in G}$, and the graded components arise from the coaction via $A_g=\{a \in A: \delta(a)= a \otimes U_g\}$.
\end{lemma}

\begin{proof}
By Lemma \ref{Graded Conditional Expectation}, (1) implies (2).

To show that (2) implies (3), suppose that $A$ has a topological grading $\{A_g\}_{g \in G}$. By Remark \ref{Graded algebra is a Fell Bundle}, these fibers form a Fell bundle. To see that the fibers are linearly independent, fix a finite sum $\fsum_{i=1}^n a_{g_i}$ where each $a_{g_i} \in A_{g_i}$ for some distinct $g_i \in G$, and suppose this sum equals 0. Then by \cite[Corollary 19.6]{exel2017partial}, there are contractive linear maps $F_g: A \ra A_g$ satisfying $F_{g_j}(\fsum_{i=1}^n a_{g_i}) =a_{g_j}$. But since $0=\fsum_{i=1}^n a_{g_i}$, then for each $g_j$, we have $0=F_{g_j}(\fsum_{i=1}^n a_{g_i}) =a_{g_j}$, so each summand is 0, so the graded components are linearly independent. Finally, we must show that $A \cong C^*(\mathcal B)$. By Lemma \ref{Fell bundle over amenable group}, $\mathcal B$ is amenable, and by \cite[Proposition 19.3]{exel2017partial}, $C^*(\mathcal B)$ is a topologically graded $C^*$-algebra. Since this amenable bundle coincides with the bundle in $A$, then by Lemma \ref{Amenable Fell Bundle Uniqueness}, $A \cong C^*(\mathcal B)$, as desired.

To see that (3) implies (1), suppose that $\mathcal B= \{B_g\}_{g \in G}$ is a Fell bundle of linearly independent closed subspaces of $A$ such that $A \cong C^*(\mathcal B)$. Since $G$ is amenable, $C^*(G) \cong C^*_r(G)$, and by Lemma \ref{Fell bundle over amenable group}, $\mathcal B$ is an amenable Fell bundle, so $C^*(\mathcal B) \cong C^*_r(\mathcal B)$. 
Combining these simplifications with \cite[Proposition 18.7]{exel2017partial}, there is an injective $*$-homomorphism $\delta: C^*(\mathcal B) \ra C^*(\mathcal B) \otimes C^*(G)$ satisfying $\delta(b_g)= b_g \otimes U_g$ for all $g \in G$ and $b_g \in B_g$. It is routine to verify that this is a coaction, and that the graded parts of this coaction coincide with the original Fell bundle $\mathcal B$.

\end{proof}

\subsubsection{Coactions as Actions of the Dual Group}

In this section we will show that if $G$ is discrete and abelian, then a coaction by $G$ is equivalent to an action by its dual group $\widehat G$. For a non-abelian group, this correspondence breaks down because there is no dual group.

First, we will remind the reader of actions by groups and dual groups. The following definition is well-known, and can be found for instance in the remarks preceding Proposition 2.1 in \cite{raeburn2005graph}:

\begin{definition}
Given a locally compact group $G$, an \ulindex{action} of $G$ on a $C^*$-algebra $A$ is a strongly continuous homomorphism $\alpha: G \ra \operatorname{Aut}(A)$, the group of automorphisms of $A$. (Here, strong continuity means that for all $a \in A$, the map $g \mapsto \alpha_g(a)$ is continuous as a function from $G$ to $A$.) The trio $(A,G,\alpha)$ is called a \ul{($C^*$-dynamical) system}\index{C$^*$@$C^*$-dynamical system}.
\end{definition}

We will now define the dual group. A more detailed treatment of dual groups can be found in \cite[Chapter 1.7]{Kaniuth2013Induced}.

\begin{definition}
Given a abelian locally compact group $G$, we say that a \ulindex{character} of $G$ is a continuous homomorphism $\chi: G \ra \T$, where $\T= \{z \in \C : |z|=1\}$ denotes the unit circle in the complex numbers. The set of characters $\widehat G$ forms an abelian locally compact group called the \ulindex{dual group} under the operation $(\chi_1* \chi_2) (g):= \chi_1(g) \chi_2(g)$ and the topology of uniform convergence on compacta.
\end{definition}

Some examples of dual groups are that $\widehat \Z^k \cong \T^k$, $\widehat{\R}^k \cong \R^k$, and that if $G$ is finite and abelian then $\widehat G \cong G$.

The following two theorems are well-known results about the duals of abelian locally compact groups. They can be found in \cite[Theorem 1.85]{Kaniuth2013Induced} and \cite[Theorem 1.88]{Kaniuth2013Induced}  respectively.

\begin{theorem}[Pontryagin duality theorem]
An abelian locally compact group is naturally isomorphic to its double dual by the map $g \mapsto [\chi \mapsto \chi(g)]$. That is, $G \cong \widehat{\widehat G}$. 

\end{theorem}

\begin{theorem}
If $G$ is a locally compact abelian group, then $G$ is discrete if and only if $\widehat G$ is compact.
\end{theorem}

We are now ready to prove the equivalence between coactions and actions by a dual group.

\begin{lemma} \label{Action Coaction Duality}
Let $G$ be a discrete abelian group and let $A$ be a $C^*$-algebra. Then for every $C^*$-dynamical system $(A, \widehat G, \alpha)$, there is a unique cosystem $(A, G, \delta)$ which satisfies 

\[ A_g = \{ a \in A: \alpha_\chi(a)= \chi(g) \cdot a \text{ for all } \chi \in \widehat G\} \]

\noindent for all $g \in G$. All cosystems by discrete abelian groups arise in this way.

\end{lemma}

\begin{proof}
Since $G$ is abelian, it is amenable. Thus by Lemma \ref{Coaction Grading Fell Equivalence}, a cosystem is equivalent to a topological grading whose graded components are the $\{A_g\}_{g \in G}$. By \cite[Theorem 3]{Raeburn2018On}, such a topological grading is equivalent to the desired group action.
\end{proof}

\subsection{Graphs and $P$-graphs}

\begin{definition}
Let $(G,P)$ be an ordered group, and consider $P$ as a category with one object where the morphisms are the elements of $P$ under their multiplication structure. A \ul{$P$-graph}\index{P-graph@$P$-graph} is a countable category $\Lambda$ along with a functor $d: \Lambda \ra P$ with the \ulindex{unique factorization property}: if $\lambda \in \Lambda$ and $p_1, p_2 \in P$ with $p_1p_2= d(\lambda)$, then there exists unique $\lambda_1, \lambda_2 \in \Lambda$ with $\lambda= \lambda_1 \lambda_2$ and $d(\lambda_1)=p_1$ (and hence necessarily $d(\lambda_2)=p_2$).

We refer to the morphisms of $ \Lambda$ as \ulindex{paths}, and identity morphisms in $\Lambda$ as \ulindex{vertices}. We let $\Lambda^0$ denote the set of vertices in $\Lambda$, and for $p \in P$, we write $\Lambda^p=\{\mu \in \Lambda : d (\mu)=p\}$. As is common in the literature, we identify the objects in $\Lambda$ with the identity morphism at those objects, so we will refer to an object $v$ with its identity morphism $\id_v$ interchangeably. Given a path $\lambda \in \Lambda$, let $r(\lambda)$ denote its range vertex and $s(\lambda)$ its source vertex. We write the composition of paths ``working backwards'' so that given $\alpha, \beta \in \Lambda$, the product $\alpha \beta$ is defined if and only if $s(\alpha)=r(\beta)$, in which case $s(\alpha \beta)=s(\beta)$ and $r(\alpha \beta) =r (\alpha)$. This can be summarized in the following diagram:

\begin{center}
\begin{tikzcd}
r(\alpha) \\
s(\alpha)=r(\beta) \arrow[u, "\alpha"] &
s(\beta) \arrow[l, "\beta"] \arrow[ul, dashed, "\alpha \beta"]

\end{tikzcd}
\end{center}

\end{definition}

\begin{remark}
The existence of the degree functor and the unique factorization property gives several nice properties to the graph.

First, $\Lambda^0=\{\lambda \in \Lambda: d(\lambda)=1_G\}$, since for a vertex $v$, we have $v^2=v$, so $d(v)^2=d(v)$, so $d(v)=1_G$, and conversely if $d(v)=1_G$, then $r(v)v=v=vs(v)$ are two factorizations with the same degrees, so $r(v)=v=s(v)$ and hence $v$ is a vertex.

Second, having $\alpha \beta= s(\beta)$ implies $\alpha=\beta=s(\beta)$, since then $d(\alpha) d(\beta) = d(s(\beta))=1$, so $d(\alpha), d(\beta) \in P \cap P\inv = \{1\}$.

Finally, the category has both left- and right-cancellation. For left cancellation, if $\alpha\beta= \alpha \gamma$, then this provides two factorizations, so by the uniqueness of the factorizations, we have $\beta=\gamma$, and a nearly identical argument shows right-cancellation.

These properties together imply that a $P$-graph is a \ulindex{category of paths} in the sense of \cite{spielberg2014groupoids}. Our representation of a $P$-graph will be a special case of their representation of a category of paths, and the reader is invited to compare our Definitions \ref{$P$-Graph Representation} and \ref{Representation Terminology}c to Theorems 6.3 and 8.2 of \cite{spielberg2014groupoids}, respectively.

\end{remark}

\begin{definition}
Let $(G,P)$ be an ordered group, let $\Lambda$ be a $P$-graph. For $\alpha \in \Lambda$, we write $\alpha \Lambda = \{\alpha \mu : \mu \in \Lambda\}$.

We can give $\Lambda$ a partial order $\leq$ by saying $\alpha \leq \beta$ if there exists some $\alpha_1 \in \Lambda$ such that $\alpha \alpha_1 =\beta$. That is, $\alpha \leq \beta$ if and only if $\beta \in \alpha \Lambda$.

We say that $\alpha$ and $\beta$ have a \ulindex{common extension} if there is a $\mu \in \Lambda$ such that $\alpha \leq \mu$ and $\beta \leq \mu$.

Given $\alpha, \beta \in \Lambda$, we say that $\mu \in \Lambda$ is a \ul{minimal common extension}\index{common extension!minimal} of $\alpha$ and $\beta$ if $\mu$ is a common extension of $\alpha$ and $\beta$, and for all other common extensions $\nu$, $\nu \leq \mu$ implies $\nu=\mu$. We let $MCE(\alpha, \beta)$ denote the set of minimal common extensions of $\alpha$ and $\beta$.

\end{definition}

For general ordered groups, the ordering on $\Lambda$ may be sufficiently poorly behaved that there are no \emph{minimal} common extensions, even if there are common extensions. However, the humble hypothesis of weak quasi-lattice ordering on $(G,P)$ prevents this catastrophe, and therefore will recur as a hypothesis in almost all results relating to $P$-graphs. The following lemma is the main ``entry-point'' for the hypothesis of weak quasi-lattice ordering into $P$-graphs:

\begin{lemma} \label{Equivalent Minimality}
Let $(G,P)$ be an ordered group, let $\Lambda$ be a $P$-graph, and $\alpha, \beta \in \Lambda$. If $(G,P)$ is weakly quasi-lattice ordered (WQLO), then:

\begin{enumerate}
\item Let $MDCE(\alpha, \beta)= \{\mu \in  \alpha \Lambda \cap \beta \Lambda : d(\mu)= d(\alpha) \vee d(\beta)\}$ denote the set of minimal degree common extensions. For every common extension $\lambda$ of $\alpha$ and $\beta$, there is a $\mu \in MDCE(\alpha, \beta)$ with $\mu \leq \lambda$.

\item $MCE(\alpha, \beta)= MDCE(\alpha, \beta)$.
\item For every common extension $\lambda$ of $\alpha$ and $\beta$, there is a $\mu \in MCE(\alpha, \beta)$ with $\mu \leq \lambda$.

\item $ \displaystyle \alpha \Lambda \cap \beta \Lambda = \bigsqcup_{\mu \in MCE(\alpha, \beta)} \mu \Lambda$, where $\bigsqcup$ denotes a disjoint union.
\end{enumerate}
\end{lemma}

\begin{proof}

(1) Suppose that $\lambda$ is a common extension of $\alpha$ and $\beta$, so $\lambda= \alpha \alpha_1=\beta \beta_1$. 

Note that $d(\lambda) \geq d(\alpha)$ and $d(\lambda) \geq d(\beta)$, so $d(\lambda) \geq d(\alpha) \vee d(\beta)$. Thus by factorization, we may write $\lambda= \mu \mu_1$ where $d(\mu)=d(\alpha) \vee d(\beta)$. Now, $d(\mu) \geq d(\alpha)$, so $\mu$ may be additionally factored as $\mu=\alpha' \alpha_2$, where $d(\alpha')=d(\alpha)$. Then we have $\lambda= \alpha \alpha_1 =\mu \mu_1= \alpha' \alpha_2 \mu_1$, meaning we have two factorizations of $\lambda$ whose initial segments are of equal length $d(\alpha')=d(\alpha)$. Then by uniqueness of factorizations we have $\alpha'=\alpha$, so $\mu$ is an extension of $\alpha$. Similarly, $\mu$ will be an extension of $\beta$, so $\mu$ is a common extension of $\alpha$ and $\beta$ of length $d(\alpha) \vee d(\beta)$, so $\mu \in MDCE(\alpha, \beta)$ as desired.

(2) If $\nu \in MCE(\alpha, \beta)$, then by (1) there is some $\mu\in MDCE(\alpha, \beta)$ with $\mu \leq \nu$. Then by minimality of $\nu$, we have $\mu=\nu$, so $\nu \in MDCE(\alpha, \beta)$ and thus $MCE(\alpha, \beta) \subseteq MDCE(\alpha, \beta)$.

Conversely, if $\mu \in MDCE(\alpha, \beta)$, suppose there were some $\nu \in \Lambda$ with $\alpha \leq \nu, \beta \leq \nu, \nu \leq \mu$. Then by (1), there would be some $\mu' \in MDCE(\alpha, \beta)$ with $\mu' \leq \nu \leq \mu$. Then $\mu' \leq \mu$, so we may write $\mu s(\mu)=\mu= \mu' \mu_1$ for some $\mu_1 \in \Lambda$, and since $d(\mu)=d(\alpha) \vee d(\beta) = d(\mu')$, then by the uniqueness of the factorization we have $\mu=\mu'$. Then $\mu=\nu$, so $\mu$ is ``minimal'' in the appropriate sense such that $\mu \in MCE(\alpha, \beta)$. Thus $MDCE(\alpha, \beta)=MCE(\alpha, \beta)$.

(3) Immediately follows from (1) and (2).

(4) It is immediate that $\bigcup_{\mu \in MCE(\alpha, \beta)} \mu \Lambda \subseteq \alpha \Lambda \cap \beta \Lambda$, and by (3) we have $ \alpha \Lambda \cap \beta \Lambda \subseteq \bigcup_{\mu \in MCE(\alpha, \beta)} \mu \Lambda$, so it suffices to show the union is disjoint. To this end, if $\mu, \nu \in MCE( \alpha, \beta)$ and $\lambda \in \mu \Lambda \cap \nu \Lambda$, then we would have $\lambda = \mu \mu_1 = \nu \nu_1$, but this presents two factorizations with equal degrees $d(\mu) =d (\alpha) \vee d(\beta) = d(\nu)$, so we must have $\mu=\nu$. Thus the $\{\mu \Lambda : \mu \in MCE(\alpha, \beta)\}$ are pairwise disjoint, as desired.
\end{proof}

\begin{remark}
It should be noted that much of the literature (e.g. \cite[Definition 2.3]{brownlowe2013co}) defines $MCE(\alpha,\beta)$ as we have defined $MDCE(\alpha, \beta)$, which the previous lemma shows are equivalent. We have chosen to emphasize the definition arising purely from the category structure because later we will be considering a category as a $P$-graph and as a $Q$-graph, and this definition clarifies that the category structure does not depend on the choice of the grading (as long as one exists).

\end{remark}

\subsection{Representations of $P$-graphs and $P$-graph Algebras}

Two additional notions will be needed in our study of $P$-graph algebras.

\begin{definition}
Let $(G,P)$ be a WQLO group, and let $\Lambda$ be a $P$-graph. We say that $\Lambda$ is \ulindex{finitely-aligned} if for all $\mu, \nu \in \Lambda$, $MCE(\mu, \nu)$ is finite.

We say that a set $A \subseteq \Lambda$ is \ulindex{exhaustive} for a set $B \subseteq \Lambda$ if for all $\beta \in B$ there is an $\alpha \in A$ such that $\alpha$ and $\beta$ have a common extension.

\end{definition}

We are at last ready to define our main object of study, the representations of $P$-graphs in a $C^*$-algebra. The following definition is due to \cite{brownlowe2013co}, with the slight modification to apply to WQLO groups:

\begin{definition} \label{$P$-Graph Representation}
Let $(G,P)$ be a WQLO group and $\Lambda$ a finitely-aligned $P$-graph. A \ul{representation}\index{representation!of a $P$-graph} of $\Lambda$ in a $C^*$-algebra $B$ is a function $t: \Lambda \ra B, \lambda \mapsto t_\lambda$ such that:

\begin{enumerate}
\item [(T1)] $t_v=t_v^*$ and $t_v t_w=\delta_{v,w} t_v$ for all $v,w \in \Lambda^0$ (here $\delta$ denotes the Kronecker delta). That is, $\{t_v: v \in \Lambda^0\}$ is a collection of pairwise orthogonal projections.
\item [(T2)] $t_\mu t_\nu = t_{\mu \nu}$ whenever $s(\mu)=r(\nu)$.
\item [(T3)] $t_\mu ^*t_\mu = t_{s(\mu)}$ for all $\mu \in \Lambda$.
\item [(T4)] $t_\mu t_\mu^*t_\nu t_\nu^*= \fsum_{\lambda \in MCE(\mu, \nu)} t_\lambda t_\lambda^*$ for all $\mu, \nu \in \Lambda$.
\end{enumerate}

We denote by $C^*(t)$ the $C^*$-algebra generated by the $\{t_\lambda\}_{\lambda \in \Lambda}$, and $C^*(t)$ is called a \ul{$P$-graph $C^*$-algebra}\index{P-graph $C^*$-algebra@$P$-graph $C^*$-algebra} or just $P$-graph algebra.

\end{definition}

We require that our graphs be finitely-aligned so that the expression in the (T4) relator is a finite sum. The fact that an infinite sum of projections does not (usually) converge in the norm topology is a perennial problem in the field of generalizing Cuntz-Krieger algebras, and our solution is to limit our attention to the finitely-aligned case where the issue does not arise. 

We will often want to check the properties of $C^*(t)$ on a dense spanning set, so the following lemma is useful. This result is widely known (see for instance \cite[Remark 5.2]{brownlowe2013co}).

\begin{lemma} \label{Zigzag Resolution}
Let $(G,P)$ be a WQLO group, let $\Lambda$ a finitely aligned $P$-graph, and let $t$ be a representation of a $P$-graph. Then $C^*(t) = \closedspan\{ t_\alpha t_\beta^* : \alpha, \beta \in \Lambda\}$.
\end{lemma}

\begin{proof}
Fix some $\alpha, \beta, \mu, \nu \in \Lambda$. If $s(\alpha) \neq s(\beta)$ or $s(\mu) \neq s(\nu)$, then by the T1 relator we have $t_\alpha t_\beta^* =0$ or $t_\mu t_\nu^*=0$ respectively, and in either case $(t_\alpha t_\beta^*)(t_\mu t_\nu^*)=0$. If instead $s(\alpha) = s(\beta)$ and $s(\mu) = s(\nu)$, then

\bea
(t_\alpha t_\beta^*)(t_\mu t_\nu^*) &=&(t_\alpha t_\beta^*)(t_\beta t_\beta^*)(t_\mu t_\mu^*)(t_\mu t_\nu^*) \text{ by T3 and T2} \\
&=&(t_\alpha t_\beta^*)\left( \fsum_{\lambda \in MCE(\beta, \mu)} t_\lambda t_\lambda^* \right)(t_\mu t_\nu^*) \text{ by T4} \\
&=& \fsum_{\lambda \in MCE(\beta, \mu)} t_{\alpha (\beta \inv \lambda)} t_{\nu (\mu\inv \lambda)} \text{ by T2 and T3} \\
\eea

\noindent where $\beta\inv \lambda$ denotes the unique path for which $\beta (\beta\inv \lambda)=\lambda$, and similarly $\mu\inv \lambda$ denotes the unique path for which $\mu (\mu\inv \lambda)=\lambda$.

This calculation shows that $\spanset\{ t_\alpha t_\beta^* : \alpha, \beta \in \Lambda\}$ is closed under multiplication, so it is a $*$-subalgebra of $C^*(t)$, and it contains each $t_\alpha$ since $t_\alpha = t_\alpha t_{s(\alpha)}^*$. Thus $\closedspan\{ t_\alpha t_\beta^* : \alpha, \beta \in \Lambda\}$ is a $C^*$-subalgebra of $C^*(t)$ containing all of the generators, so $C^*(t)=\closedspan\{ t_\alpha t_\beta^* : \alpha, \beta \in \Lambda\}$.

\end{proof}

\begin{example} \label{Left Regular Representation}
As we have already seen with groups and Fell bundles, a simple way of representing an object often comes from its own action on itself, and such an example arises for representations of $P$-graphs as well. To define the \ul{left-regular representation} of a $P$-graph $\Lambda$\index{left-regular representation!of a $P$-graph}\index{representation!of a $P$-graph!left-regular}, let $\mathcal H= \ell^2(\Lambda)$, with the typical orthonormal basis $\{e_\alpha\}_{\alpha \in \Lambda}$. (Kribs and Power call this the \ulindex{Fock space} of $\Lambda$ in their papers \cite{kribs2003free, kribs2004h}.) For $\mu \in \Lambda$ define a ``forward shift by $\mu$'' operator $L_\mu$ by 

\[ L_\mu e_\alpha= \begin{cases}
e_{\mu \alpha} & \text{ if } s(\mu)=r(\alpha) \\
0 & \text{ if } s(\mu) \neq r(\alpha)
\end{cases} \]

\noindent and extending this linearly. A short computation on the inner product shows that $L_\mu^*$ is the ``backwards shift by $\mu$'' operator, which is given by

\[ L_\mu^* e_\beta = \begin{cases}
e_{\alpha} & \text{ if there exists }\alpha \text{ with } \mu \alpha= \beta \\
0 & \text{ if no such } \alpha \text{ exists}
\end{cases} \]

\noindent (noting that if such an $\alpha$ does exist, it is unique by left cancellation).

We will now show that $\{L_\mu\}_{\mu \in \Lambda}$ is a representation of $\Lambda$. The T1 and T2 relations are immediate.

For the T3 relation, observe that

\[ L_\mu^* L_\mu e_\alpha = L_\mu^*\left( \begin{cases}
e_{\mu \alpha} & \text{ if } s(\mu)=r(\alpha) \\
0 & \text{ if } s(\mu) \neq r(\alpha) \end{cases} \right)
= \begin{cases}
e_{\alpha} & \text{ if } s(\mu)=r(\alpha) \\
0 & \text{ if } s(\mu) \neq r(\alpha) \end{cases} = L_{s(\mu)} e_\alpha \]

\noindent as desired.

For the T4 relation, first note that 

\[ L_\mu L_\mu^* e_\alpha = \begin{cases}
e_{\alpha} & \text{ if } \alpha \in \mu \Lambda \\
0 & \text{ if } \alpha \not \in \mu \Lambda \end{cases}. \]

Thus 

\[ L_\mu L_\mu^* L_\nu L_\nu^* e_\alpha = \begin{cases}
e_{\alpha} & \text{ if } \alpha \in \mu \Lambda \cap \nu \Lambda \\
0 & \text{ if } \alpha \not \in \mu \Lambda \cap \nu \Lambda \end{cases}.\] 

But since $\mu \Lambda \cap \nu \Lambda = \bigsqcup_{\lambda \in MCE(\mu, \nu)} \lambda \Lambda$ by Lemma \ref{Equivalent Minimality}(4), then if $\alpha \in \mu \Lambda \cap \nu \Lambda$ there exists a unique $\lambda_0 \in MCE(\mu, \nu)$ with $\alpha \in \lambda \Lambda$, so $L_{\lambda_0} L_{\lambda_0}^* e_\alpha = e_\alpha$, and $L_\lambda L_\lambda^* e_\alpha =0$ for $\lambda \in MCE(\mu, \nu) \setminus \lambda_0$. Thus 

\[ \fsum_{\lambda \in MCE(\mu, \nu)} L_\lambda L_\lambda^* e_\alpha=e_\alpha \]

\noindent when $\alpha \in \mu \Lambda \cap \nu \Lambda$, so $L_\mu L_\mu^* L_\nu L_\nu^*=\fsum_{\lambda \in MCE(\mu, \nu)} L_\lambda L_\lambda^*$ as desired.

Thus $\{L_\mu\}_{\mu \in \Lambda}$ is a representation, as desired.
\end{example}

There are several additional properties that we will sometimes ask of our representation. In order, we will ask for our generators to be nonzero, that our representation preserve its knowledge of the $P$-graded structure through a certain coaction called a gauge coaction, and that our representation be ``tight'' in the sense of \cite{exel2009tight}. More precisely:

\begin{definition}\label{Representation Terminology}
Let $(G,P)$ be a WQLO group and $\Lambda$ a finitely aligned $P$-graph. We say that a representation $t$ of $\Lambda$...

\begin{enumerate}[label=\alph*.]
\item is \ulindex{$\Lambda$-faithful}\index{faithful!$\Lambda$-faithful} if each $t_\lambda$ is nonzero.

\item has a \ulindex{gauge coaction}\index{coaction!gauge} if there is a $G$-coaction $\delta$ on $C^*(t)$ such that

\[ \delta(t_\lambda ) = t_\lambda \otimes U_{d(\lambda)} \]

\noindent for all $\lambda \in \Lambda$.

\item is \ulindex{tight} if whenever $\mu \in \Lambda$ and $E \subset \mu \Lambda$ is finite and exhaustive for $\mu \Lambda$, we have

 \[ \fprod_{\alpha \in E} (t_\mu t_\mu^* -t_{\alpha}t_{\alpha}^* )=0. \] 
 
(Recall that $E$ is exhaustive for $\mu \Lambda$ if for all $\nu \in \mu \Lambda$, there is an $\alpha \in E$ such that $\alpha$ and $\nu$ have a common extension.) This terminology is motivated by \cite{exel2009tight}, which defines a tight representation of a semilattice. We show in Appendix 1 that this notion of tight is equivalent to the one in \cite{exel2009tight}.

\item \ul{canonically covers}\index{cover! of a representation}\index{cover! canonical} another representation $s$ if there is a (necessarily surjective) $*$-homomorphism from $C^*(t)$ to $C^*(s)$ given by $t_\lambda \mapsto s_\lambda$. In such a case, we will write $\pi_s^t: C^*(t)\ra C^*(s)$ to denote the $*$-homomorphism, which we call the \ul{canonical covering}. The notation $\pi_s^t$ should hopefully be suggestive of $t$ ``covering'' $s$.

\item is \ulindex{canonically isomorphic} to another representation $s$ if there is an isomorphism between $C^*(s)$ and $C^*(t)$ given by $t_\lambda \mapsto s_\lambda$ (i.e. if $s$ and $t$ canonically cover each other).

\end{enumerate} 
\end{definition}

It should be noted that the conditions (T1)-(T4) and tightness are all polynomial relations of the form suitable for Lemma \ref{Universal Algebra} (finite-alignment being necessary in the case of (T4)). However, $\Lambda$-faithfulness and the existence of a gauge coaction are not polynomial relators in an obvious way.

\begin{remark} \label{Regular Representation Properties}
The left-regular representation $L$ from Example \ref{Left Regular Representation} is $\Lambda$-faithful. 

However, this representation is as far from tight as possible: given $\mu \in \Lambda$ and $E \subset \mu \Lambda$ which is finite and exhaustive for $\mu \Lambda$, $\fprod_{\alpha \in E} (L_\mu L_\mu^* -L_{\alpha}L_{\alpha}^* )=0$ if and only if $\mu \in E$ (i.e. when one of the terms of the product is 0). To see this, note that if $\alpha \in \mu \Lambda \setminus \mu$, then $L_\alpha L_\alpha^* e_\mu=0$, so $(L_\mu L_\mu^* -L_{\alpha}L_{\alpha}^*) e_\mu=e_\mu$, and thus taking the product over $\alpha \in E \subset \mu \Lambda \setminus \mu$, we get that $\fprod_{\alpha \in E} (L_\mu L_\mu^* -L_{\alpha}L_{\alpha}^* ) e_\mu = e_\mu \neq 0$.

We will show in Corollary \ref{Left Regular is Toeplitz} that $L$ has a gauge coaction if $(G,P)$ has a strong reduction to an amenable group.
\end{remark}

\begin{lemma} \label{Toeplitz Construction}
Let $(G,P)$ be a WQLO group, and let $\Lambda$ be a finitely aligned $P$-graph. One may apply Lemma \ref{Universal Algebra} to the relators (T1)-(T4), so there is a representation of $\Lambda$ which is universal.
\end{lemma}

\begin{proof}
Observe that (T1) and (T3) together imply that the generators are partial isometries. Also, all of the relators are polynomials in the generators (in the case of (T4), we must note that the expression is finite since the graph is finitely aligned). Thus there is a universal representation of $\Lambda$.
\end{proof}

\begin{definition} \label{Toeplitz Definition}
We call the universal representation of a $P$-graph $\Lambda$ the \ulindex{Toeplitz(-Cuntz-Krieger) representation}\index{representation!of a $P$-graph!Toeplitz} of $\Lambda$. We will use $\mathcal T$\index{T@$\mathcal T$}\index{T@$\mathcal TC^*(\Lambda)$} to denote this representation, and write $\mathcal T C^*(\Lambda)$ or $C^*(\mathcal T)$ for the algebra generated by it which we call the \ulindex{Toeplitz(-Cuntz-Krieger) algebra} of the graph.
\end{definition}

\begin{lemma}\label{Representation Order}

Let $(G,P)$ be a WQLO group, and let $\Lambda$ be a finitely aligned $P$-graph. 
\begin{enumerate}

\item Let $Rep(\Lambda)$ denote the set of all representations of $\Lambda$ up to canonical isomorphism. Then $Rep(\Lambda)$ can be given a partial ordering $\geq_{rep}$ by saying that $s \geq_{rep} t$\index{$\geq_{rep}$} if and only if $s$ canonically covers $t$ (and $t \cong s$ if and only if they are canonically isomorphic). 

\item There is a bijection $\Psi$ from $Rep(\Lambda)$ to ideals in $\mathcal TC^*(\Lambda)$ given by $t \mapsto \ker \pi_t^{\mathcal T}$, which is order-reversing in the sense that $t \leq_{rep} s \iff \Psi(t) \supseteq \Psi(s)$.

\end{enumerate}
\end{lemma}

\begin{proof}
For (1), we must check the partial order properties of reflexivity, transitivity, and antisymmetry. 

Certainly each representation canonically covers itself by the identity map, so $\leq_{rep}$ is reflexive.

If $t \leq_{rep} s$ and $s \leq_{rep} r$, then there are canonical coverings $\pi_s^r:C^*(r) \ra C^*(s)$ and $\pi_t^s:C^*(s) \ra C^*(t)$. Then the composition $\pi_t^r:= \pi_t^s \circ \pi_s^r$ is a canonical covering of $t$ by $r$, so $t \leq_{rep} r$, giving transitivity.

If $t \leq_{rep} s$ and $s \leq_{rep} t$, then there are canonical coverings $\pi_t^s:C^*(s) \ra C^*(t)$ and $\pi_s^t:C^*(t) \ra C^*(s)$. Composing these maps in either direction shows that they fix each generator, and therefore are the identity maps. Therefore $\pi_t^s$ and $\pi_s^t$ are inverses and hence canonical isomorphisms.

Thus $\leq_{rep}$ is indeed a partial order, as desired.

For (3), it is certainly true that $\Psi$ is a well-defined function from $Rep(\Lambda)$ to $\{\text{ideals in }\mathcal TC^*(\Lambda)\}$. 

We will first check that if $s,t \in Rep(\Lambda)$, then $t \leq_{rep} s \iff \Psi(t) \supseteq \Psi(s)$. If $t \leq_{rep} s$, then there is a canonical covering $\pi^s_t: C^*(s) \ra C^*(t)$, and since $\pi^s_t  \circ \pi^{\mathcal T}_s = \pi^{\mathcal T}_t$, then $\Psi(t)=\ker \pi^{\mathcal T}_t  \supseteq \pi^{\mathcal T}_s= \Psi(s)$. Conversely, if $\Psi(t) \supseteq \Psi(s)$, then by Lemma \ref{Factors Through Theorem} applied to $\pi^{\mathcal T}_t$ and $\pi^{\mathcal T}_s$, there is a  $*$-homomorphism $\pi_t^s:C^*(s) \ra C^*(t)$ with $\pi_t^s \circ \pi^{\mathcal T}_s =\pi^{\mathcal T}_t$so $\pi_t^s$ is a canonical covering and $t \leq_{rep} s$. This completes the proof that $\Psi$ is order-reversing.

Now to show that $\Psi$ is a bijection, if $\Psi(s)=\Psi(t)$, then by the above $s\leq_{rep} t$ and $t \leq_{rep} s$, so $s \cong t$ and hence $\Psi$ is injective. For surjectivity, if $J \lhd \mathcal T C^*(\Lambda)$ is an ideal, represent $\Lambda$ by $\lambda \mapsto \mathcal T_\lambda +J$. Then this representation has kernel $J$, so $\Psi$ is surjective, as desired.
\end{proof}

As a consequence of this, representations have the order structure of the family of ideals of an algebra, and in particular have a lattice order (where meet and join are represented by intersection and addition of ideals).

\begin{remark}
$\Lambda$-faithfulness is a ``largeness'' condition in that if $s \leq t$ and $s$ is $\Lambda$-faithful, then so is $t$.

Tightness is a ``smallness'' condition in that if $s \leq t$ and $t$ is tight, then so is $s$.
\end{remark}

\subsubsection{Gauge coactions and Coactionization}

It may not be obvious what role ``has a gauge coaction'' plays: is it saying a representation is ``large'' or ``small''? 

The following proposition is a tangled web of results, but it shows that the existence of a gauge coaction is properly considered a ``largeness'' condition: any representation can be ``lifted'' (1) to have a gauge coaction (4), the lifted version covers the original (2) and is the smallest gauge coacting representation to do so (6), and they are isomorphic if and only if the original had a gauge coaction (5). We also simplify the notion of a gauge coaction by showing that any homomorphism satisfying $t_\lambda \mapsto t_\lambda \otimes U_{d(\lambda)}$ is automatically a gauge coaction (3), and show that a representation is $\Lambda$-faithful or tight if and only if its coactionization is (7).

\begin{proposition} \label{Coactionization}
Let $(G,P)$ be a WQLO group, and let $\Lambda$ be a finitely aligned $P$-graph. Let $t$ be a representation of $\Lambda$. 

\begin{enumerate}

\item There is a representation $t'$ of $\Lambda$ given by $t'_\lambda= t_\lambda \otimes U_{d(\lambda)}$. 
\item There is a canonical covering $t' \mapsto t$.
\item If there is a $*$-homomorphism $\delta: C^*(t) \ra C^*(t) \otimes C^*(G)$ satisfying $\delta(t_\lambda) = t'_\lambda$, then $\delta$ is a coaction, and hence a gauge coaction.
\item $t'$ has a gauge coaction.
\item $t$ is canonically isomorphic to $t'$ if and only if $t$ has a gauge coaction.
\item If $s \leq t$, then $s' \leq t'$, and in particular for any representation $s$, $s'$ is the smallest gauge coacting representation that covers $s$.
\item $t'$ is $\Lambda$-faithful (respectively, tight) if and only if $t$ is.
\end{enumerate}

Due to these properties, we believe that $t'$ deserves to be called the \ulindex{coactionization} of $t$.
\end{proposition}

\begin{proof}
The proof of (1) is a routine confirmation of the T1-T4 relators, which we include here. We have:

%\bea
%T1: \hspace{1 cm}& t_v' t_w'=& (t_v \otimes U_e)(t_w \otimes U_e)= t_v t_w \otimes U_e = \delta_{v,w} t_v \otimes U_e= \delta_{v,w} t_v' \\
% T2: \hspace{1 cm}&t_\mu' t_\nu'=& (t_\mu \otimes U_{d(\mu)}) (t_\nu \otimes U_{d(\nu)})= (t_{\mu} t_\nu) \otimes U_{d(\mu)} U_{d(\nu)} = t_{\mu \nu} \otimes U_{d(\mu \nu)} = t_{\mu \nu}' \\
% T3: \hspace{1 cm}&{t_\mu'} ^* t_\mu' =& (t_\mu \otimes U_{d(\mu)})^* (t_\mu \otimes U_{d(\mu)}) = t_\mu^* t_\mu \otimes U_{d(\mu)}^* U_{d(\mu)}= t_{s(\mu)} \otimes U_e=  t_{s(\mu)}'\\
% T4: \hspace{1 cm}&t_\nu' {t_\nu'} ^* t_\mu' t_\mu'^* =& (t_\nu \otimes U_{d(\nu)}) (t_\nu \otimes U_{d(\nu)})^* (t_\mu \otimes U_{d(\mu)})(t_\mu \otimes U_{d(\mu)})^* = (t_\nu t_\nu^* t_\mu t_\mu^*) \otimes U_e 
%  \\ &&= \fsum_{\lambda \in MCE(\mu, \nu)} t_\lambda t_\lambda^* \otimes U_e =  \fsum_{\lambda \in MCE(\mu, \nu)} t_\lambda' {t_\lambda'}^* 
% \eea

\begin{tabular}{cccl}
T1: &  $t_v' t_w'$&=& $(t_v \otimes U_e)(t_w \otimes U_e)= t_v t_w \otimes U_e = \delta_{v,w} t_v \otimes U_e= \delta_{v,w} t_v'$ \\
 T2: & $t_\mu' t_\nu'$ &=& $(t_\mu \otimes U_{d(\mu)}) (t_\nu \otimes U_{d(\nu)})= (t_{\mu} t_\nu) \otimes U_{d(\mu)} U_{d(\nu)}$ \\
 &&=& $t_{\mu \nu} \otimes U_{d(\mu \nu)} = t_{\mu \nu}'$ \\
 T3: &${t_\mu'} ^* t_\mu'$ &=& $(t_\mu \otimes U_{d(\mu)})^* (t_\mu \otimes U_{d(\mu)}) = t_\mu^* t_\mu \otimes U_{d(\mu)}^* U_{d(\mu)}$ \\
 &&=& $t_{s(\mu)} \otimes U_e=  t_{s(\mu)}'$\\
 T4: & $t_\nu' {t_\nu'} ^* t_\mu' t_\mu'^*$ &=& $(t_\nu \otimes U_{d(\nu)}) (t_\nu \otimes U_{d(\nu)})^* (t_\mu \otimes U_{d(\mu)})(t_\mu \otimes U_{d(\mu)})^*$ \\
 &&=& $(t_\nu t_\nu^* t_\mu t_\mu^*) \otimes U_e = \fsum_{\lambda \in MCE(\mu, \nu)} t_\lambda t_\lambda^* \otimes U_e$ \\
 &&=& $ \fsum_{\lambda \in MCE(\mu, \nu)} t_\lambda' {t_\lambda'}^* $
 \end{tabular}

\noindent all as desired.

For (2), let $V$ denote the trivial representation of $G$ (given by $V_g=1 \in \C$ for all $g \in G$). Then by the universal property of $C^*(G)$, there is a covering map $\pi^U_V: C^*(G) \ra \C$ given by $\pi^U_V(U_g)=V_g=1$.

Now by Lemma \ref{Tensor of Homomorphisms}, there is a map $\id_{C^*(t)} \otimes \pi^U_V: C^*(t) \otimes C^*(G) \ra C^*(t) \otimes \C$ satisfying 

\bea
(\id_{C^*(t)} \otimes \pi^U_V) (t'_\lambda) &=& (\id_{C^*(t)} \otimes \pi^U_V) (t_\lambda \otimes U_{d(\lambda)}) \\
&=& t_\lambda \otimes 1.
\eea

By identifying $C^*(t)$ with $C^*(t) \otimes \C$ in the natural way, this gives the desired canonical covering.

For (3), we must check that $\delta$ is injective, nondegenerate, and satisfies the coaction identity of Definition \ref{Coaction Definition}. To this end, note that if $\pi_t^{t'}$ is the canonical covering from the previous part, then $\pi_t^{t'} \circ \delta = \id_{C^*(t)}$, so $\delta$ must be injective. To check nondegeneracy, we must check that 

\[ \closedspan \left[ \delta(C^*(t)) (C^*(t) \otimes C^*(G)) \right]= C^*(t) \otimes C^*(G). \]

To this end, note that for all $g \in G$ and $\mu, \nu \in \Lambda$,  

\[ \delta(t_{s(\mu)}) (t_\mu t_\nu^* \otimes U_g) = (t_{s(\mu)} \otimes U_1) (t_\mu t_\nu^* \otimes U_g)= t_\mu t_\nu^* \otimes U_g. \] 

That is, $\{t_\mu t_\nu^* \otimes U_g: \mu, \nu \in \Lambda, g \in G\} \subseteq \delta(C^*(t))(C^*(t) \otimes C^*(G))$. But since $C^*(t) = \closedspan\{ t_\mu t_\nu^* : \mu, \nu \in \Lambda\}$ and $C^*(G)= \closedspan \{U_g \}_{g \in G}$, then these simple tensors have dense span $C^*(t) \otimes C^*(G)$. Thus

\[ C^*(t) \otimes C^*(G)= \closedspan \{t_\mu t_\nu^* \otimes U_g\}_{\mu, \nu \in \Lambda, g \in G} = \closedspan \left[ \delta(C^*(t)) (C^*(t) \otimes C^*(G)) \right] \]

\noindent as desired.

Finally, it suffices to check the coaction identity on each generator $t_\mu$. To this end,

\bea 
((\delta \otimes \id_G) \circ \delta)(t_\mu) &=& (\delta \otimes \id_G) (t_\mu \otimes U_{d(\mu)})\\
&=& t_\mu \otimes U_{d(\mu)}\otimes U_{d(\mu)}\\
&=& t_\mu \otimes \delta_G(U_{d(\mu)})\\
&=& ((\id_{C^*(t)} \otimes \delta_G) ( t_\mu \otimes U_{d(\mu)})\\
&=& ((\id_{C^*(t)} \otimes \delta_G) \circ \delta)(t_\mu)\\
\eea

\noindent as desired.

For (4), recall that there is a coaction $\delta_G: C^*(G) \ra C^*(G) \otimes C^*(G)$ given by $U_g \mapsto U_g \otimes U_g$. By Lemma \ref{Tensor of Homomorphisms} we may define $\bar \delta: C^*(t) \otimes C^*(G) \ra C^*(t) \otimes C^*(G) \otimes C^*(G)$ by $\bar \delta=\id_{C^*(t)} \otimes \delta_G$. Let $\delta = \bar \delta \mid_{C^*(t')}$. Then for $\lambda \in \Lambda$, 

\[ \delta(t')= \delta(t_\lambda \otimes U_{d(\lambda)})= t_\lambda \otimes \delta_G(U_{d(\lambda)})= t_\lambda \otimes U_{d(\lambda)} \otimes U_{d(\lambda)} = t_{\lambda}' \otimes U_{d(\lambda)}.\] 

Thus by (3), $\delta$ is a gauge coaction of $C^*(t')$.

For (5), if $t$ is canonically isomorphic to $t'$, then there is a homomorphism between their $C^*$-algebras given by $t_\lambda \mapsto t'_\lambda =t_\lambda \otimes U_{d(\lambda)}$. By (3) this is precisely a gauge coaction.

For (6), if $s \leq t$, recall that $\pi_s^t$ denotes the canonical covering. Then by Lemma \ref{Tensor of Homomorphisms} there is a $*$-homomorphism $\pi_{s'}^{t'}: C^*(t') \ra C^*(s')$ given by $\pi_{s'}^{t'}= \pi_s^t \otimes \id_{C^*(G)}$, and it is immediate that this is a canonical covering. For the ``in particular'', if $s \leq t$ and $t$ has a gauge coaction, then $t \cong t'$ by (5), so $s' \leq t' \cong t$, and thus $s' \leq t$, as desired.

For (7), observe that for a fixed $\lambda \in \Lambda$, 

\[ \norm{t_\lambda'}^2=  \norm{t_\lambda \otimes U_{d(\lambda)}}^2 = \norm{t_\lambda^*t_\lambda \otimes 1 }= \norm{t_\lambda^* t_\lambda} = \norm{t_\lambda}^2\]

\noindent so any $t_\lambda'$ is 0 if and only if $t_\lambda$ is 0. Thus $t$ is $\Lambda$-faithful if and only if $t'$ is $\Lambda$-faithful.

For tightness, fix a $\mu \in \Lambda$ and finite $E \subset \mu\Lambda$ which is exhaustive for $\mu \Lambda$. Then

\bea
\norm{ \fprod_{\alpha \in E} t_\mu' {t_\mu'}^* - t_\alpha' {t_\alpha'}^*}&=& \norm{ \left( \fprod_{\alpha \in E} t_\mu {t_\mu}^* - t_\alpha {t_\alpha}^* \right) \otimes 1} \\
&=&  \norm{\fprod_{\alpha \in E} t_\mu {t_\mu}^* - t_\alpha {t_\alpha}^*}
\eea

\noindent and in particular $\fprod_{\alpha \in E} t_\mu' {t_\mu'}^* - t_\alpha' {t_\alpha'}^*=0$ if and only if $\fprod_{\alpha \in E} t_\mu {t_\mu}^* - t_\alpha {t_\alpha}^*=0$, as desired.
\end{proof}

Finally, note that property (3) says that the notion of coaction is dramatically simplified when considering just gauge coactions. Indeed, the necessary and sufficient condition to having a gauge coaction is that the map $t_\lambda \mapsto t_\lambda \otimes U_{d(\lambda)}$ extends to a $*$-homomorphism on $C^*(t)$.

\subsection{The Classical Gauge-Invariant Uniqueness Theorem}

In this section we will introduce the gauge-invariant uniqueness theorem for directed graphs (what we would call $\N$-graphs), and explain how we hope to generalize it.

The following definitions and theorems are adapted from Chapters 1 and 2 of \cite{raeburn2005graph}.

\begin{definition}
A \ulindex{directed graph} is a tuple $E=(E^0, E^1, r,s)$ where $E^0$ and $E^1$ are countable sets, and $r,s:E^1 \ra E^0$ are functions. We think of $E^0$ as the set of vertices in our directed graph, $E^1$ as our set of edges, and $r$ and $s$ as the range and source maps for $E^1$. We say a directed graph is \ulindex{row-finite} if for each $v \in V$, $|r\inv(v) |<\infty$.

A \ulindex{Cuntz-Krieger $E$-family} of a row-finite directed graph $(E^0, E^1, r,s)$ is a collection of operators $\{P_v\}_{v \in E^0} \cup \{S_e\}_{e \in E^1}$ in a $C^*$-algebra such that $\{P_v\}_{v \in E^0} $ are a family of pairwise orthogonal projections, and

\begin{enumerate}
\item [] (CK1) $P_{s(e)}=S_e^*S_e$ for all $e\in E^1$ and
\item [] (CK2) $P_v = \fsum_{\substack{e \in E^1 \\ r(e)=v}} S_eS_e^*$ whenever $r\inv(v)$ is nonempty.
\end{enumerate}
\end{definition}

The reader may wish to compare this definition to the definition of a representation of a $P$-graph (Definition \ref{$P$-Graph Representation}). Having done so, there are a few differences that we will wish to reconcile.

\begin{remark}
\sloppy Given a directed graph $(E^0, E^1, r,s)$, for $n>1$ we let $E^n=\{(e_1,e_2,...,e_n) : s(e_i)=r(e_{i+1}) \text{ for } 1\leq i<n\}$. Then $E^*= \bigcup_{n\geq 0} E^n$ forms an $\N$-graph in a natural way, where the operation is concatenation of tuples, and $d(e_1,...,e_n)=n$. This $\N$-graph is always finitely aligned, since for any $\N$-graph, 

\[MCE(\alpha, \beta) = \begin{cases} \{\alpha\}& \text{ if } \beta \leq \alpha \\ \{\beta\} &\text{ if } \alpha \leq \beta \\ \emptyset &\text{ otherwise} \end{cases} \]

\noindent so $|MCE(\alpha, \beta)| \leq 1 <\infty$.

Any Cuntz-Krieger $E$-family extends to a representation of $E^*$ by $S_{(e_1,e_2,...,e_n)}=S_{e_1}S_{e_2}...S_{e_n}$. In proving this, the (T1) and (T2) relators are immediate, the (T3) relator follows rapidly from (CK1), and the (T4) relator follows from the expression above for $MCE(\alpha, \beta)$.

One should note that the (CK2) condition is analogous to tightness, since for any vertex $v$, $r\inv(v)$ is finite and exhaustive for $v \Lambda$, and since $\{S_eS_e^*\}_{e \in r\inv(v)}$ are pairwise orthogonal projections, then 

 \[ \fprod_{e \in r\inv(v)} (P_v P_v^* -S_{e}S_{e}^* )=0 \text{ if and only if } P_v = \fsum_{\substack{e \in E^1 \\ r(e)=v}} S_eS_e^*. \]
 
In other words, the (CK2) condition is saying that the representation is tight in the case that $\mu \in \Lambda^0$ and $E \subset \Lambda^1$. Since every path in an $\N$-graph can be uniquely factored into paths of length 1, tightness in all cases is equivalent to this.
\end{remark}

Due to this last point, readers should be aware that in the literature representations of directed graphs (and higher rank graphs) are what we would call tight representations.

In the notation of \cite{raeburn2005graph}, $C^*(E)$ denotes the universal tight representation of $E$, and if $\{T_v\}_{v \in E^0}\cup \{Q_{e}\}_{e\in E^1}$ is a Cuntz-Krieger $E$-family, then $\pi_{T,Q}: C^*(E) \ra C^*(T,Q)$ denotes the canonical covering. 

The foundational result we wish to generalize is this, as stated in \cite[Theorem 2.2]{raeburn2005graph}:

\begin{theorem}[Gauge Invariant Uniqueness Theorem for Graphs] \label{GIUT for Graphs}
\sloppy Let $E=(E^0, E^1, r,s)$ be a row-finite directed graph, and suppose that $\{T_v\}_{v \in E^0}\cup \{Q_{e}\}_{e\in E^1}$ is a Cuntz-Krieger $E$-family in a $C^*$-algebra $B$ with each $Q_v \neq 0$. If there is a continuous action $\beta: \T \ra \operatorname{Aut} B$ such that $\beta_z(T_e)=zT_e$ for every $e \in E^1$ and $\beta_z(Q_v)=Q_v$ for every $v \in E^0$, then $\pi_{T,Q}$ is an isomorphism of $C^*(E)$ onto $C^*(T,Q)$.
\end{theorem}

\begin{remark}
The reader may recognize many of the hypotheses of Theorem \ref{GIUT for Graphs} from our earlier list of terminology for representations (Definition \ref{Representation Terminology}). The statement $Q_v \neq 0$ for $v \in E^0$ implies that $T_\mu \neq 0$ for all $\mu \in E^*$, so it is equivalent to our $\Lambda$-faithfulness. The continuous action $\beta$ is called a \ulindex{gauge action}\index{action!gauge} and is equivalent by Lemma \ref{Action Coaction Duality} to a gauge coaction. Also recall that tightness is a built-in hypothesis in the representations considered in \cite{raeburn2005graph}.

Therefore, we may restate Theorem \ref{GIUT for Graphs} like so: for any $\N$-graph, there is exactly one $\Lambda$-faithful, tight, gauge coacting representation up to canonical isomorphism.

What we mean by a gauge invariant uniqueness theorem for $P$-graphs is a similar statement: for any $P$-graph, there is exactly one $\Lambda$-faithful, tight, gauge coacting representation up to canonical isomorphism.

We have already seen that in generalizing from $\N$-graphs to $P$ graphs, we needed an additional hypothesis (finite alignment). We will see in Lemma \ref{GIUT Failure} that such a gauge invariant uniqueness theorem need not be true in general, so another hypothesis is needed. We spend Chapter 3 developing this hypothesis on $(G,P)$, and then in Chapter 4 prove a gauge invariant uniqueness theorem for $P$-graphs that satisfy this new hypothesis.

\end{remark}

\section{Reductions of Ordered Groups}

In this section we will develop the notion of a reduction of an ordered group. Roughly speaking, a reduction occurs when a positive cone $P$ in a group $G$ does not carry ``enough'' information about the group, and one is able to replace $(G,P)$ with some $(H,Q)$ while preserving the essential properties that will be required for a $P$-graph (namely, the structure of intervals $[1,p]$). A special case of a reduction (called a ``strong reduction'') is when one can embed the positive cone $P$ in another group $H$ such that this embedding extends to an order homomorphism from $(G,P)$ to $(H,P)$.

\subsection{Definition and Basic Properties}

\begin{definition}\label{Reduction Definition}
An order homomorphism $\varphi: (G,P) \ra (H,Q)$ is a \ulindex{reduction} if for all $p \in P$, $\varphi$ is a bijection between the interval  $[1,p]:=\{x \in P : 1 \leq x \leq p\}$ and the interval $[1,\varphi(p)]:=\{y \in Q : 1 \leq y \leq \varphi(p)\}$.

We say that $\varphi$ is a \ul{strong reduction}\index{reduction!strong} if $\varphi \mid_P :P \ra Q$ is a bijection between $P$ and $Q$.

We say $(G,P)$ has a (strong) reduction to an amenable group or (strongly) reduces to an amenable group if there is an amenable group $H$, a positive cone $Q \subset H$, and a (strong) reduction $\varphi: (G, P) \ra (H,Q)$. 

\end{definition}

Intuitively, a reduction is a map that preserves the order structure of $(G,P)$ ``locally'', meaning on every interval. For strong reductions, the reader should think of them as arising from embedding the positive cone $P$ in two distinct groups $(G,P)$ and $(H,P)$ such that there is a homomorphism from one to the other taking $P$ to itself bijectively. 

\begin{example}
Let $G=\langle a,b \rangle$, the free group on two generators. If $P=\{a,b\}^*$, then define $\varphi:(G,P) \ra (\Z, \N)$ by $\varphi(a)=\varphi(b)=1$, which is an order homomorphism.

Then $\varphi$ is a reduction since, for example, it maps the interval \sloppy $[e, abba]= \{e,a,ab,abb,abba\}$ bijectively onto the interval $[\varphi(e), \varphi(abba)]= [0,4]=\{0,1,2,3,4\}$. Since $\Z$ is amenable, then this shows that $(G,P)$ reduces to an amenable group.

However, $\varphi$ is not injective even when restricted to $P$, as $\varphi(a)=1=\varphi(b)$, so $\varphi$ is not a strong reduction.
\end{example}

It may not be immediately obvious from the definition that a strong reduction is a reduction, so we will prove this as a short lemma:

\begin{lemma}
Let $\varphi:(G,P) \ra (H,Q)$ be a strong reduction. Then it is a reduction.
\end{lemma}

\begin{proof}
We must show that if $\varphi$ is bijective as a map from $P$ to $Q$, then it is bijective from $[1,p]$ to $[1, \varphi(p)]$ for each $p \in P$. Fixing a $p \in P$, since $[1,p] \subset P$, then $\varphi$ is injective on $[1,p]$, so it suffices to show it is surjective onto $[1,\varphi(p)]$.

To this end, fix some $q_1 \in [1, \varphi(p)]$. Since $q_1 \leq \varphi(p)$, there is a $q_2 \in Q$ such that $q_1q_2=\varphi(p)$. Now since $q_1,q_2 \in Q$ and $\varphi$ is a bijection from $P$ to $Q$, there is a $p_1, p_2 \in P$ such that $\varphi(p_1)=q_1$ and $\varphi(p_2)=q_2$. Then $\varphi(p)=q_1q_2= \varphi(p_1) \varphi(p_2)=\varphi(p_1p_2)$. Since $p_1p_2 \in P$ and $\varphi$ is injective on $P$, then $p_1p_2=p$, so $p_1 \in [1,p]$, and in particular there is a $p_1 \in [1,p]$ such that $\varphi(p_1)=q_1$. That is, $\varphi$ is a bijection from $[1,p]$ onto $[1, \varphi(p)]$, as desired. 
\end{proof}

We'll now provide several alternative characterizations of being a reduction and a strong reduction.

\begin{proposition} \label{Reduction Criterion}
Let $\varphi:(G,P) \ra (H,Q)$ be an order homomorphism. Then the following are equivalent:

\begin{enumerate}
\item $\varphi$ is a reduction in the sense of Definition \ref{Reduction Definition}.
\item $d^P_Q:= \varphi \mid_P : P \ra Q$ is a functor that makes $P$ into a $Q$-graph.
\item For all $q \in Q, p \in P$ such that $q \leq \varphi(p)$, there exist unique $p_1,p_2 \in P$ such that $p_1p_2=p$ and $\varphi(p_1)=q$.
\item For each nonempty interval $[x,y] \subset G$, $\varphi$ is an order isomorphism between $[x,y]$ and $[\varphi(x), \varphi(y)]$.
\item The following two statements together:
\begin{enumerate}
\item $S \cap \ker \varphi=\{1\}$ where $S= \bigcup_{p \in P} [1,p][1,p] \inv $.
\item For all $p \in P$, $\varphi([1,p])=[1, \varphi(p)]$.
\end{enumerate}

\end{enumerate}
\end{proposition}

\begin{proof}
Fix $\varphi:(G,P) \ra (H,Q)$ an order homomorphism. We will prove that $(2) \Rightarrow(3) \Rightarrow (1) \Rightarrow (4) \Rightarrow (2)$ and $(1) \iff (5)$.

$(2 \Rightarrow 3)$ Fix $q \in Q$, $p \in P$ with $q \leq \varphi(p)= d_Q^P(p)$. Then since $P$ is a $Q$-graph, we can uniquely factorize $p$ as $p=p_1p_2$ where $q=d^P_Q(p_1)=\varphi(p_1)$, as desired.

$(3 \Rightarrow 1)$ Fix $p \in P$. By Lemma \ref{Order Homomorphism}, $\varphi([1,p]) \subset [1, \varphi(p)]$. It then suffices to show that if $q \in [1, \varphi(p)]$, then there exists a unique $s \in [1,p]$ with $\varphi(s)=q$. To this end, if $q \in [1, \varphi(p)]$, then by (3), there exists $p_1, p_2 \in P$ such that $p_1p_2=p$ and with $\varphi(p_1)=q$. Then $p_1 \in [1,p]$, so $p_1$ is such an $s$. If there were some $p_1' \in [1, p]$ with $\varphi(p_1)=\varphi(p_1')$, then by the former, there exists $p_2' \in P$ with $p_1'p_2' = p$, so by the uniqueness condition of (2), we have $p_1'=p_1$ and $p_2'=p_2$, and thus our $p_1$ is unique, as desired.

$(1 \Rightarrow 4)$ Fix $x, y \in G$, and suppose the interval $[x,y]= \{z \in G: x \leq z \leq y\}$ is nonempty. Since this interval is nonempty, then $x\leq y$, so $1 \leq x\inv y$, and we will write $p=x\inv y \in P$. By left-invariance, $x \leq z \leq y$ if and only if $1 \leq x\inv z \leq x\inv y$, so $[x,y]=x[1,x\inv y]=x[1,p]$. By left-invariance, $x[1,p]$ is order-isomorphic to $[1,p]$, and $\varphi(x)[1,\varphi(p)]$ is order-isomorphic to $[1, \varphi(p)]$, so it suffices to show that $[1,p]$ is order-isomorphic to $[1, \varphi(p)]$. By (1), $\varphi$ is a bijection between $[1,p]$ and $[1, \varphi(p)]$, so it suffices to show that $s \leq t \iff \varphi(s) \leq \varphi(t)$ for $s,t \in [1, p]$. The $\Rightarrow$ direction is immediate from Lemma \ref{Order Homomorphism}. For the other direction, suppose $s,t \in [1,p]$ and $\varphi(s) \leq \varphi(t)$. Then $\varphi(s) \in [1, \varphi(t)]$, and by (3), $\varphi$ maps $[1,t]$ surjectively onto $[1, \varphi(t)]$, so there exists some $s' \in [1,t]$ such that $\varphi(s')=\varphi(s)$. But $s' \in [1,t] \subseteq [1,p]$, and $\varphi$ is injective on $[1, p]$, so $s'=s$. Thus $s \in [1,t]$, so $s \leq t$.

$(4 \Rightarrow 2)$ Certainly $d^P_Q:= \varphi \mid_P$ is a functor from $P$ to $Q$. It then suffices to check unique factorization. To this end, suppose $p \in P$ and $q_1, q_2 \in Q$ with $d_Q^P(p)=q_1q_2$. Then $q_1 \in [1, d_Q^P(p)]= [1, \varphi(p)]$, so by (4) there is a unique $p_1 \in [1,p]$ with $q_1=\varphi(p_1)=d_Q^P(p_1)$. Since $p_1 \leq p$, then $p_2:= p_1\inv p \in P$. Thus $p$ is uniquely factorized as $p=p_1p_2$ with $d_Q^P(p_1)=q_1$.

$(1 \Rightarrow 5a)$ Suppose that $s \in S \cap \ker \varphi$. Then $s= q r\inv $ where $q,r \in [1,p]$ for some $p \in P$. Then $\varphi(r)= \varphi(s) \varphi(r) = \varphi(sr)= \varphi(q)$. By (1), $\varphi \mid_{[1,p]}$ is injective, so $q=r$ and thus $s=q r\inv = 1$, as desired. .

$(1 \Rightarrow 5b)$ is immediate from the fact that $\varphi$ is a bijection between the intervals $[1,p]$ and $[1, \varphi(p)]$.

$(5 \Rightarrow 1)$ Fix $p \in P$. By Lemma \ref{Order Homomorphism}, $\varphi([1,p]) \subseteq [1, \varphi(p)]$. It then suffices to show that if $q \in [1, \varphi(p)]$, then there exists a unique $s \in [1,p]$ with $\varphi(s)=q$. By (5b), there exists at least one $s \in [1,p]$ with $\varphi(s)=q$. Assume for the sake of contradiction that there were distinct $s_1, s_2 \in [1,p]$ with $\varphi(s_1)=q=\varphi(s_2)$. Then $s_1 s_2 \inv \in \ker \varphi$. But since $s_1 s_2 \inv \in [1,p] [1,p]\inv \subseteq S$, then by (5a) we have $s_1 s_2 \inv=1$, so $s_1=s_2$, a contradiction of distinctness. Thus $\varphi$ is bijective as a map from $[1,p]$ to $[1, \varphi(p)]$.

\end{proof}

\begin{lemma}\label{Strong Reduction}
Let $(G,P)$, $(H,Q)$ be ordered groups, and $\varphi:G \ra H$ a group homomorphism. The following are equivalent:

\begin{enumerate}
\item $\varphi$ is a strong reduction in the sense of Definition \ref{Reduction Definition}.
\item $\varphi \mid_P$ is an order isomorphism from $P$ to $Q$.
\end{enumerate} 

\end{lemma}

\begin{proof}
$(2 \Rightarrow 1)$ is immediate.

$(1 \Rightarrow 2)$ Since $\varphi \mid_P$ is already a bijection, we must show that for $p_1, p_2 \in P$, $p_1 \leq p_2 \iff \varphi(p_1) \leq \varphi(p_2)$. The $\Rightarrow$ direction is immediate from Lemma \ref{Order Homomorphism}. For the other direction, suppose that $\varphi(p_1) \leq \varphi(p_2)$. Then there is a $q \in Q$ such that $\varphi(p_1)q =\varphi(p_2)$. Since $\varphi$ is a surjection from $P$ to $Q$, there exists a $p \in P$ such that $\varphi(p)=q$, and hence $\varphi(p_1p)=\varphi(p_2)$. Since $\varphi$ is injective on $P$, then $p_1p=p_2$, and thus $p_1 \leq p_2$, as desired.
\end{proof}

While the theory of reductions of ordered groups can be entirely severed from the amenability of the groups involved, it is perhaps unsurprising that a reduction to an amenable group will allow us to do more interesting analysis of the $C^*$-algebras of $P$-graphs that will appear in later sections.

%Some readers will be pleased to know that the theory of reductions works perfectly well when walled off from the notion of amenability. However,  to learn that a reduction to an amenable group will allow us to do more interesting analysis of the $C^*$-algebras of $P$-graphs that will appear in later sections.

The following example shows that the existence of a reduction to an amenable group (or even a non-injective reduction!) depends on not just the group $G$, but also the positive cone $P$.

\begin{example} \label{No Amenable Reduction}
Let $G=\langle a,b \rangle$, the free group on two generators. If $P=\{a,b\}^*$, then we've seen that $(G,P)$ is an ordered group which reduces to an amenable group via the map $\varphi:(G,P) \ra (\Z, \N)$ given by $\varphi(a)=\varphi(b)=1$.

Suppose that $(G,R)$ is a total ordering on a group (meaning that $R \cup R\inv = G$ in addition to $R \cap R\inv= \{1\}$). Then for any reduction $\varphi$, by Proposition \ref{Reduction Criterion}(5), since $S=\bigcup_{r \in R} [1,r][1,r] \inv =G$, then $\ker \varphi = \ker \varphi \cap G=\{1\}$. That is, every reduction $\varphi$ is injective.

In particular, taking $G=F_2$ and $R$ a total ordering on $G$ (such as the ordering arising from the Magnus expansion given in \cite[Section 3.2]{clay2016ordered}), $(F_2,R)$ cannot reduce to an amenable group since the range group of a reduction will always contain a copy of $F_2$.

\end{example}

\begin{example}
For any group $(G,P)$, $\id_G$ is a reduction, so if $G$ is amenable, $\id_G$ is a reduction to an amenable group.
\end{example}

\begin{example}
Let $G=\langle a,b \rangle$, the free group on two generators and $P=\{a,b\}^*$. One can give a strong reduction of $(G,P)$ onto the amenable group $H=BS(1,2)= \langle c, t | tc=c^2t \rangle$. Let $Q= \{t, ct\}^*$, which is a positive cone in $BS(1,2)$, and let $\varphi:(G,P) \ra (H,Q)$ by $\varphi(a)=ct, \varphi(b)=t$. Certainly $\varphi$ is an order homomorphism. Then one may verify that $\varphi$ is bijective from $P$ to $Q$ in this way: if $\varphi(p_1)=\varphi(p_2)$, then write this word in $H$ as $c^i t^j$. It must be that $j$ is the length of $p_1$ and the length of $p_2$, so $p_1$ and $p_2$ have equal length. Then, $i$ will be the number that comes from substituting $1$ for $a$ and $0$ for $b$ in the original word and interpreting the result as a binary number. Two binary numbers with the same number of digits are equal if and only if their digits are equal, so it must be that $p_1=p_2$. Thus $\varphi$ is injective on $P$, and hence bijective on $P$, meaning $\varphi$ is a strong reduction.

\end{example}

The following result shows that reductions preserve weakly quasi-lattice order.

\begin{lemma}
Let $\varphi: (G,P) \ra (H,Q)$ be a reduction. Then:

\begin{enumerate}
\item If $(H,Q)$ is WQLO, then so is $(G,P)$.
\item If $\varphi(P)=Q$ and $(G,P)$ is WQLO, then so is $(H,Q)$.
\end{enumerate}
\end{lemma}

\begin{proof}
For (1), assume for the sake of contradiction that $(H,Q)$ is WQLO but $(G,P)$ is not. Then there are $x,y \in P$ with a common upper bound but no least common upper bound. That is, there are upper bounds $b_1 > b_2 >b_3>...$ such that $x,y \leq b_i$ for all $i \in \N$.

For all $i \in \N$, since $1 \leq x \leq b_i$, then $1 \leq b_i$. In particular, $b_i \in [1,b_1]$ for all $i \in \N$. Also note that $x,y \in [1,b_1]$.

Now, since $\varphi$ is a reduction, $\varphi$ is an order isomorphism from $[1,b_1]$ to $[1, \varphi(b_1)]$, so in $Q$, $\varphi(b_1)> \varphi(b_2) >...$ is a strictly decreasing sequence of common upper bounds of $\varphi(x)$ and $\varphi(y)$. This contradicts the fact that $(H,Q)$ is WQLO.

For (2), assume for the sake of contradiction that $\varphi(P)=Q$ and $(G,P)$ is WQLO but that $(H,Q)$ is not WQLO. Then there are $z,w \in Q$ with a common upper bound but no least common upper bound. That is, there are upper bounds $c_1 > c_2 >c_3>...$ such that $z,w \leq c_i$ for all $i \in \N$.

For all $i \in \N$, since $1 \leq z \leq c_i$, then $1 \leq c_i$. In particular, $c_i \in [1,c_1]$ for all $i \in \N$. Also note that $z,w \in [1,c_1]$.

Now, since $\varphi(P)=Q$, there is some $p \in P$ such that $\varphi(p)=c_1$. Since $\varphi$ is a reduction, $\varphi$ is an order isomorphism from $[1,p]$ to $[1, \varphi(p)]=[1,c_1]$. Let $x$, $y$, and $b_i$ denote the unique elements in $[1,p]$ such that $\varphi(x)=z, \varphi(y)=w$, and $\varphi(b_i)=c_i$ for all $i \in \N$. Then in $P$, $b_1>b_2>b_3...$ is a strictly decreasing sequence of common upper bounds of $x$ and $y$. This contradicts the fact that $(G,P)$ is WQLO.

\end{proof}

%\begin{lemma}
%Let $\varphi: (G,P) \ra (H,Q)$ be a strong reduction. Then if the ordering on $(G,P)$ is weakly quasi-lattice, so is the ordering on $(H,Q)$.
%\end{lemma}
%
%\begin{proof}
%Since $\varphi$ is a strong reduction, then $P$ is order isomorphic to $Q$. But being weakly quasi-lattice ordered is a property of just the positive cone, so $(H,Q)$ is also weakly quasi-lattice ordered.
%
%\end{proof}

\begin{remark}
There is a notion of amenability for semigroups, and $(G,P)$ may strongly reduce to an amenable group even if neither $G$ nor $P$ is amenable as a group or semigroup (respectively). For example, in $(F_2, P_2)$, the free group on 2 generators with positive cone $P_2$ which is the free monoid on 2 generators, we will show that $(F_2, P_2)$ strongly reduces to $(\Z \wr \Z, \N \wr \N)$, which is amenable, but $F_2$ is famously not amenable, and $P_2$ is also not amenable as a semigroup. 

The fact that $P_2$ is not amenable but can be embedded in an amenable group has been known for more than half a century. In \cite{hochster1969subsemigroups}, the author shows an embedding of $P_2$ into $\Z \wr \Z^2$.

\end{remark}

\subsection{Constructions}

In this section, we will show that the notion of ``(strongly) reduces to an (amenable) group'' behaves well with the usual group theory constructions such as composition, hereditary subgroups, direct products, and free products.

\subsubsection{Composition}

\begin{lemma} \label{Composition of Reductions}
A composition of (strong) reductions is a (strong) reduction. That is, if $\varphi: (G,P) \ra (H,Q)$ and $\psi: (H, Q) \ra (F,R)$ are two (strong) reductions, then $\psi \circ \varphi$ is a (strong) reduction.
\end{lemma}

\begin{proof}
Suppose $\varphi: (G,P) \ra (H,Q)$ and $\psi: (H, Q) \ra (F,R)$ are two reductions. Then we will show that $\psi \circ \varphi$ is a reduction.

Certainly, $\psi \circ \varphi$ is an order homomorphism, so by Lemma  \ref{Reduction Criterion} (3), we must show that for all $p \in P$, $\psi \circ \varphi$ bijectively maps $[1,p]$ onto $[1, \psi (\varphi(p))]$. 

To this end, we know that $\varphi: [1,p] \ra [1, \varphi (p)]$ is bijective, and $\psi: [1, \varphi(p)] \ra [1, \psi(\varphi(p))]$ are bijections since $\varphi$ and $\psi$ are reductions, so therefore their composition is a bijection, as desired.

For the composition of strong reductions, if $\varphi: (G,P) \ra (H,Q)$ and $\psi: (H, Q) \ra (F,R)$ are two strong reductions, then $\varphi$ maps $P$ bijectively onto $Q$ and $\psi$ maps $Q$ bijectively onto $R$, so $\psi \circ \varphi$ maps $P$ bijectively onto $R$, as desired.

\end{proof}

\subsubsection{Hereditary Subgroups}

In the context of ordered groups, the ``correct'' notion of a subgroup is often a hereditary subgroups in the sense of \cite[Corollary 5.6]{brownlowe2013co}. In this section we will remind the reader of the definition and some basic results about the concept.

\begin{definition}
Given an ordered group $(G,P)$, if $Q \subseteq P$ is a subsemigroup, we say that $Q$ is \ulindex{hereditary} in $P$ if $p_1,p_2 \in P$, $p_1p_2 \in Q$ implies that $p_1,p_2 \in Q$.

If $(G,P)$ is an ordered group, we will say that $(H,Q)$ is a \ul{hereditary subgroup}\index{hereditary!subgroup} if $H$ is a subgroup of $G$, $Q$ is a subsemigroup of $P \cap H$, and $Q$ is hereditary in $P$.
\end{definition}

\begin{lemma}\label{Hereditary Same Order}
Let $(H,P)$ be an ordered group, and $Q \subseteq P$ a hereditary subsemigroup. For the sake of clarity, let $\leq_P$ and $\leq_Q$ denote the orderings on $G$ arising from $P$ and $Q$, respectively. Then for all $q \in Q$ and $p \in P$, $1 \leq_P p \leq_P q$ if and only if $1 \leq_Q p \leq_Q q$. 

In particular, for $q \in Q$, the intervals $[1,q]_{\leq_P}$ and $[1,q]_{\leq_Q}$ are equal as sets.
\end{lemma}

\begin{proof}
Fix some $q \in Q, p \in P$. 

Suppose $1 \leq_P p \leq_P q$. Then $p \in P$ and there exists an $s \in P$ such that $ps=q$. Since $q \in Q$ and $Q$ is hereditary, then this implies that $p,s \in Q$, so $1 \leq_Q p \leq_Q q$, as desired.

Suppose $1 \leq_Q p \leq_Q q$. Then $p \in Q$ and there exists an $s \in Q$ such that $ps=q$. Since $Q \subseteq P$, then $1 \leq_P p \leq_P q$, as desired.

The ``in particular'' part is immediate.

\end{proof}

\begin{lemma} \label{Surjective Reduction}
Let $(G,P), (H,Q)$ be ordered groups, and $\varphi:G \ra H$ a homomorphism. Let $H'=\varphi(G)$, $Q'=\varphi(P)$, and let $\varphi'$ denote the codomain restriction of $\varphi$ to the codomain $H'$. Then the following are equivalent:

\begin{enumerate}
\item $\varphi$ is a reduction.
\item $(H',Q')$ is a hereditary subgroup of $(H,Q)$ and $\varphi'$ is a surjective reduction of $(G,P)$ onto $(H', Q')$.
\end{enumerate}

\end{lemma}

\begin{proof}
$(1 \Rightarrow 2)$ If $\varphi$ is a reduction, we will first show that $Q'=\varphi(P)$ is a hereditary subsemigroup of $Q$. If $q= \varphi(p) \in \varphi(P)$, and $rr'=q$ for some $r,r' \in Q$, then $r \in [1,\varphi(p)]$, so by Proposition \ref{Reduction Criterion}(3), there exists a unique $s \in [1,p]$ such that $\varphi(s)=r$. Then $\varphi(s)r'= \varphi(p)$, so $r'= \varphi(s\inv p)$. Since $s \leq p$, then there exists $s'\in P$ such that $ss'=p$, so $r'=\varphi(s')$. Thus $r,r' \in \varphi(P)$, so $\varphi(P)$ is a hereditary subsemigroup of $Q$.

Since additionally $H'<H$ and $Q' <Q \cap \varphi(G)$, then $(H', Q')$ is a hereditary subgroup of $(H,Q)$.

To show that $\varphi'$ is a surjective reduction, it is immediately surjective since its codomain was restricted to the image. Now fix a $p \in P$. Since $\varphi$ is a reduction, then by Proposition \ref{Reduction Criterion} (3), $\varphi$ maps $[1,p]$ bijectively onto $[1, \varphi(p)]$. By Lemma \ref{Hereditary Same Order}, we have $[1, \varphi(p)]= [1, \varphi'(p)]$ as sets, so $\varphi'$ maps $[1,p]$ bijectively onto $[1, \varphi'(p)]$. Then by Proposition \ref{Reduction Criterion} (3), $\varphi'$ is a reduction.

$(2 \Rightarrow 1)$ Fix a $p \in P$. Since $\varphi'$ is a reduction, then by Proposition \ref{Reduction Criterion} (3), $\varphi'$ maps $[1,p]$ bijectively onto $[1, \varphi'(p)]$. By Lemma \ref{Hereditary Same Order}, we have $[1, \varphi'(p)]_{\leq Q'}= [1, \varphi(p)]_{\leq Q}$ as sets, so $\varphi$ maps $[1,p]$ bijectively onto $[1, \varphi(p)]$. Then by Proposition \ref{Reduction Criterion} (3), $\varphi$ is a reduction.

\end{proof}

\begin{lemma} 
Let $(G,P)$ be an ordered group and $(H,Q)$ a hereditary subgroup of $(G,P)$. Then the inclusion map $i:(H,Q) \ra (G,P)$ is a reduction. 
\end{lemma}

\begin{proof}
Observe that the inclusion map, restricted to its image, is the identity map, and thus a surjective reduction. Then the inclusion map is a reduction by Lemma \ref{Surjective Reduction}.

\end{proof}

We can now conclude that ``reduces to an amenable group'' is closed under taking hereditary subgroups.

\begin{corollary}\label{Hereditary Subgroup Preserves Amenable Reduction}
Let $(G,P)$ be an ordered group and $(H,Q)$ a hereditary subgroup of $(G,P)$. If $(G,P)$ reduces to an amenable group, then $(H,Q)$ reduces to an amenable group.
\end{corollary}

\begin{proof}
Let $i: H \ra G$ denote the inclusion map, which by the previous lemma is a reduction.

If $\varphi:(G,P) \ra (F,R)$ is a reduction to an amenable group, then $\varphi \circ i: (H,Q) \ra (F,R)$ is a composition of reductions, and hence a reduction by Lemma \ref{Composition of Reductions}. It has amenable range, so it is a reduction to an amenable group.
\end{proof}

\subsubsection{Direct Products}

Now, we'll show that reductions respect direct products, which implies that ``reduces to an amenable group'' is preserved under direct products.

First, let us check that direct products preserve WQLO:

\begin{lemma}
Let $(G,P)$ and $(H,Q)$ be ordered groups. Let $P \times Q$ denote the submonoid of $G \times H$ generated by (the canonical copies of) $P$ and $Q$. Then $(G \times H, P \times Q)$ is an ordered group, and if $(G,P)$ and $(H,Q)$ are WQLO, then so is $(G \times H, P \times Q)$.
\end{lemma}

\begin{proof}
By definition, $P \times Q$ is a submonoid of $G \times H$, and it is immediate that $P \times Q \cap (P \times Q)\inv = (P \cap P\inv) \times (Q \cap Q\inv)= \{1_G\} \times \{1_H\}$, so $(G \times H, P \times Q)$ is an ordered group.

\sloppy If $(G,P)$ and $(H,Q)$ are WQLO, then one may quickly confirm that given $(p_1, q_1), (p_2, q_2) \in P \times Q$, their least upper bound is $(p_1 \vee p_2, q_1 \vee q_2)$.

\end{proof}

\begin{lemma} \label{Direct Product Preserves Reduction}
For $i=1,2$, let $(G_i, P_i)$ and $(H_i, Q_i)$ be ordered groups, $\varphi_i: G_i \ra H_i$ a homomorphism. Let $G=G_1 \times G_2$, and similarly define $P, H, Q$. Then there is a homomorphism $\varphi: G \ra H$ given by $\varphi(g_1, g_2)= (\varphi_1(g_1), \varphi_2(g_2))$, and $\varphi$ is an order homomorphism (respectively, reduction or strong reduction) if $\varphi_1$ and $\varphi_2$ both are.

\end{lemma}

\begin{proof}
The existence of such a homomorphism is an elementary fact of group theory.

If $\varphi_1, \varphi_2$ are both order homomorphisms, then $\varphi_1(P_1) \subseteq Q_1$ and $\varphi_2(P_2) \subseteq Q_2$, so $\varphi(P_1 \times P_2) \subseteq \varphi_1(P_1) \times \varphi_2(P_2)  \subseteq Q_1 \times Q_2$, as desired.

If $\varphi_1, \varphi_2$ are both reductions, then we will use Proposition \ref{Reduction Criterion} (3) as our notion of a reduction, so it suffices to show that for all $p_1 \in P_1, p_2 \in P_2$, $[(1,1), (p_1, p_2)]$ is mapped bijectively onto $[(1,1), (\varphi_1(p_1), \varphi_2(p_2)]$. But in $(G_1 \times G_2, P_1 \times P_2)$, we have that $[(1,1), (x_1, x_2)]= [1,x_1] \times [1,x_2]$, so we know that $\varphi= \varphi_1 \times \varphi_2$ bijectively carries $[(1,1), (p_1, p_2)]=[1,p_1] \times [1,p_2]$ onto $[(1,1), (\varphi_1(p_1), \varphi_2(p_2))]=[1, \varphi_1(p_1)] \times [1, \varphi_2(p_2)]$, as desired.

If $\varphi_1, \varphi_2$ are both strong reductions, then $\varphi_1$ and $\varphi_2$ map $P_1$ and $P_2$ bijectively onto $Q_1$ and $Q_2$ respectively, so for all $q=(q_1, q_2) \in Q$ there is a unique $p_1 \in P_1, p_2 \in P_2$ such that $\varphi_1(p_1)=q_1, \varphi_2(p_2)=q_2$, and hence $p=(p_1, p_2)$ is the unique element of $P$ with $\varphi(p)=q$, so $\varphi$ is a bijection from $P$ to $Q$, as desired.

\end{proof}

\begin{corollary}\label{Direct Product Preserves Amenable Reduction}
If $(G_1, P_1)$, $(G_2, P_2)$ are ordered groups which reduce to amenable ordered groups, then $(G_1 \times G_2, P_1 \times P_2)$ reduces to an amenable group.
\end{corollary}

\begin{proof}
For $i=1,2$, let $\varphi_i:(G_i, P_i) \ra (H_i, Q_i)$ denote the reduction of $(G_i,P_i)$ to an amenable ordered group.

By the previous Lemma, we know that $(G_1 \times G_2, P_1 \times P_2)$ reduces to $(H_1 \times H_2, Q_1 \times Q_2)$. Since $H_1$ and $H_2$ are amenable, then $H_1 \times H_2$ is amenable, so this is a reduction to an amenable group.
\end{proof}

\subsubsection{Free Products}

In this section we will show that the class of groups which have reductions to amenable groups is closed under (finite) free products. That is, we will show that if $(G, P)$ and $(H,Q)$ reduce to amenable groups, then $(G*H, P*Q)$ also reduces to an amenable group.

This will consist of two steps, analogous to Lemma \ref{Direct Product Preserves Reduction} and Corollary \ref{Direct Product Preserves Amenable Reduction} from the direct product case. A free product analogue of Lemma \ref{Direct Product Preserves Reduction} is straightforward. However, in the free product case, an analogue of Corollary \ref{Direct Product Preserves Amenable Reduction} is more difficult, since the free product of two amenable groups is almost never amenable. Therefore, our second step is longer.

Given ordered groups $(G_1, P_1)$ and $(G_2,P_2)$, we denote by $P_1 * P_2$ the submonoid of $G_1 *G_2$ generated by the naturally embedded copies of $P_1$ and $P_2$. First, we will prove a lemma characterizing the order on $(G_1*G_2, P_1 *P_2)$:

\begin{lemma} \label{Free Positive Cone}
Let $(G_1,P_1), (G_2, P_2)$ be ordered groups, and let $(G,P)=(G_1*G_2, P_1 *P_2)$. 

\begin{enumerate}
\item For $p \in P$, we may write $p=p_1p_2...p_{2n}$ where $p_i \in P_1$ if $i$ is odd, and $p_i \in P_2$ if $i$ even, and $p_i \neq 1$ for $1<i<2n$. 

\item This representation of $p$ is unique. That is, there is exactly one such choice of $n$ and elements $p_i$ such that $p_i \in P_1$ if $i$ is odd, and $p_i \in P_2$ if $i$ even, and $p_i \neq 1$ for $1<i<2n$. 

\item Given such a representation $p=p_1p_2...p_{2n}$, we have $[1,p] = \bigcup_{i=1}^{2n} X_i$, where $X_i= p_1...p_{i-1} [1, p_i]$.

\item Given such a decomposition $[1,p] = \bigcup_{i=1}^{2n} X_i$, if $a \in X_i$ and $b \in X_j$ for $i < j$, then $a \leq b$.

\end{enumerate}

\end{lemma}

\begin{proof}
For (1), by definition every element of $P$ can be written as a product of elements of $P_1$ and $P_2$. If two consecutive elements are from the same $P_i$, they can be combined, and if any intermediate term is a 1, it can be removed and the now-consecutive terms combined. Finally, if necessary a term of 1 can be put at the beginning or end of the expression to make the alternating product begin with a term from $P_1$ and end with a term from $P_2$. 

For (2), this representation is unique because of the Normal Form Theorem for Free Products (\cite[Chapter IV, Theorem 1.2]{lyndon2015combinatorial}).

For (3), fix a $p=p_1p_2...p_{2n}$. It is immediate that each $X_i$ is a subset of $[1,p]$, so it suffices to show containment in the other direction. To this end, suppose $q \in [1,p]$, so $q \in P$ and $p=qr$ for some $r\in P$. Then writing $q=q_1...q_{2m}$ and $r=r_1...r_{2k}$, we have that $p=p_1p_2...p_{2n}=q_1...q_{2m}r_1...r_{2k}$. We now have a few cases depending on whether or not $q_{2m}$ and $r_1$ are the identity.

\begin{itemize}
\item $q_{2m} =1, r_1=1$: Then $p=p_1p_2...p_{2n}= q_1...q_{2m-1}r_2...r_{2k}$ is the unique alternating presentation with nonidentity terms, so $p_1=q_1, p_2=q_2$, etc. In particular, $q=q_1...q_{2m-1}=p_1...p_{2m-1} \in X_{2m}= p_1...p_{2m-1} [1, p_{2m}]$.
\item $q_{2m} \neq 1, r_1 \neq 1$: Then $p=p_1p_2...p_{2n}=q_1...q_{2m}r_1...r_{2k}$ is the unique alternating presentation with nonidentity terms, so $p_1=q_1, p_2=q_2$, etc. In particular, $q=q_1...q_{2m}=p_1...p_{2m} \in X_{2m}= p_1...p_{2m-1} [1, p_{2m}]$. 
\item $q_{2m} \neq 1, r_1=1$: Then $p=p_1p_2...p_{2n}=q_1...(q_{2m}r_2)r_3...r_{2k}$ is the unique alternating presentation with nonidentity terms, so $p_1=q_1, p_2=q_2$, etc, ending with $p_{2m}=q_{2m}r_2$. Thus $q_{2m} \in [1, p_{2m}]$, so $q= q_1...q_{2m} =p_1p_2...p_{2m-1}q_{2m} \in X_{2m} = p_1...p_{2m-1} [1,p_{2m}]$. 
\item $q_{2m} =1, r_1 \neq 1$: Then $p=p_1p_2...p_{2n}=q_1...q_{2m-2}(q_{2m-1}r_1) r_2...r_{2k}$ is the unique alternating presentation with nonidentity terms, so $p_1=q_1, p_2=q_2$, etc, ending with $p_{2m-1}=q_{2m-1}r_1$. Thus $q_{2m-1} \in [1, p_{2m-1}]$, so $q= q_1...q_{2m-1} =p_1p_2...p_{2m-2}q_{2m-1} \in X_{2m-1} = p_1...p_{2m-2} [1,p_{2m-1}]$. 
\end{itemize}

In all four cases, $q$ is in some $X_i$, so $[1,p]= \bigcup_{i=1}^{2n} X_i$, as desired.

For (4), note that for each $y \in X_i$, we have $p_1p_2...p_{i-1} \leq y \leq p_1p_2...p_i$, so $a \leq p_1p_2...p_{i} \leq p_1p_2...p_{j-1} \leq b$, as desired. 
\end{proof}

We can now confirm that weak quasi-lattice order is preserved under free products.

\begin{lemma}
If $(G_1,P_1)$ and $(G_2,P_2)$ are ordered groups (respectively, WQLO groups) then so is $(G_1*G_2, P_1*P_2)$.
\end{lemma}

\begin{proof}
If $(G_1,P_1)$ and $(G_2,P_2)$ are ordered groups, then by definition $P_1*P_2$ is a submonoid of $G * H$. We must now check that $(P_1*P_2) \cap (P_1*P_2)\inv =\{1\}$, or equivalently that if $a,b \in P_1*P_2$ with $ab=1$, then $a=b=1$. But if $ab=1$, then by writing $a=p_1p_2...p_{2n}$ and $b=q_1q_2...q_{2m}$ as in the previous theorem, then $1=ab=p_1p_2...p_{2n}q_1q_2...q_{2m}$. By the Normal Form Theorem for Free Products (\cite[Chapter IV, Theorem 1.2]{lyndon2015combinatorial}, a reduced sequence in a free product can equal 1 if and only if its length is 1, so all the $p_i$s and $q_i$s in the products must cancel with each other. But since each term of the product is positive, they can only cancel if each term is equal to 1. Thus $a=b=1$, as desired, so $(G_1*G_2, P_1*P_2)$ is an ordered group.

If $(G_1,P_1)$ and $(G_2,P_2)$ are WQLO, we will check that $(G_1*G_2, P_1*P_2)$ is WQLO. To this end, fix some $y_1, y_2 \in P_1*P_2$ and suppose $y_1, y_2$ have some common upper bound $y_3 \in P_1 *P_2$. By part (1) of the previous lemma, we can write $y_3=p_1p_2p_3...p_{2n-1}p_{2n}$, where each $p_i \in P_1$ for $i$ odd and $p_i \in P_2$ for $i$ even. By part (3) of that lemma, $[1,y_3]= \bigcup_{i=1}^{2n} X_i$ where $X_i= p_1...p_{i-1} [1, p_i]$. Since $y_1, y_2 \in [1,y_3]$, then $y_1 \in X_i, y_2 \in X_j$ for some $i, j \leq 2n$. Without loss of generality, suppose $i \leq j$.

If $i<j$, then by part (4) of the previous Lemma, we have $y_1 \leq y_2$, so $y_1 \vee y_2=y_2$.

If $i=j$, then $y_1, y_2 \in X_i= p_1...p_{i-1} [1, p_i]$, so there are $r_1,r_2$ such that $y_1=p_1...p_{i-1}r_1$ and $y_2=p_1...p_{i-1}r_2$ where $r_1,r_2 \in P_1$ if $i$ is odd, and $r_1,r_2 \in P_2$ if $i$ is even. In either case, $p_1...p_{i-1} (r_1 \vee r_2)$ will be the supremum of $y_1$ and $y_2$.

\end{proof}

We will now show that if two groups have a reduction, their free product reduces to the free product of the reductions.

\begin{lemma} \label{Reduction of Free Product}
For $i=1,2$, let $\varphi_i:(G_i, P_i) \ra (H_i, Q_i)$ be a reduction. Then let
\[ G=G_1 * G_2 , P=P_1*P_2 ,H=H_1*H_2, \text{ and } Q=Q_1*Q_2, \]
 and let $\varphi: G \ra H$ be the homomorphism satisfying $\varphi \mid_{G_i}=\varphi_i$ for $i=1,2$. Then $\varphi:(G, P) \ra (H,Q)$ is a reduction.
\end{lemma}

\begin{proof}
Certainly, $\varphi$ is a homomorphism. By Proposition \ref{Reduction Criterion} (3), it suffices to show that for all $p \in P$, the interval $[1,p]$ is mapped bijectively onto $[1, \varphi(p)]$. 

Fix some $p \in P$, and then by Lemma \ref{Free Positive Cone}, we may write $p=p_1p_2...p_{2n}$ where $p_i$ is in $P_1$ if $i$ is odd, and in $P_2$ if $i$ is even. Then $[1,p] = \bigcup_{i=1}^{2n} X_i$, where $X_i= p_1...p_{i-1} [1, p_i]$. 

Now, for $1 \leq i \leq 2n$, let $q_i=\varphi_1(p_i)$ if $i$ is odd, and $q_i=\varphi_2(p_i)$ if $i$ is even. Thus $\varphi(p)=q_1...q_{2n}$. Note that for $1<i<2n$, $q_i \neq 1$ since $p_i \neq 1$ and by Proposition \ref{Reduction Criterion} (4), both $\varphi_1$ and $\varphi_2$ send strictly positive elements to strictly positive elements.

Thus $\varphi(p)=q_1...q_{2n}$ is the unique way to write $\varphi(p) \in Q$ as an alternating product of non-unit elements, so by Lemma \ref{Free Positive Cone}, $[1, \varphi(p)]= \bigcup_{i=1}^{2n} Y_i$ where $Y_i= q_1...q_{i-1} [1, q_i]$. But since $\varphi_1$ and $\varphi_2$ are reductions, we know that each $X_i$ is carried bijectively to its $Y_i$, so it suffices to explain why bijectivity is preserved when taking the union of the $X_i$ and $Y_i$s. It is immediate that the map from $\bigcup X_i$ to $\bigcup Y_i$ is surjective. 

For injectivity, first note that for $1 \leq i \leq j \leq n$,

\[ Y_i \cap Y_j= \begin{cases}
\emptyset & j>i+1 \\
\{ q_1...q_{i} \} & j=i+1\\
Y_i & i=j
\end{cases} \]

and similarly,

\[ X_i \cap X_j= \begin{cases}
\emptyset & j>i+1 \\
\{ p_1...p_{i} \} & j=i+1\\
X_i & i=j
\end{cases} \]

Then, if $\varphi(s_1) \in Y_i$ and $\varphi(s_2) \in Y_j$ for $s_1, s_2 \in [1,p]$ such that $\varphi(s_1)=\varphi(s_2)$, either $i=j$ so $s_1, s_2 \in X_i$ and thus by injectivity on $X_i$ we have $s_1=s_2$, or $j=i+1$, so $\varphi(s_1)=\varphi(s_2)= q_1...q_i$, so $s_1=s_2=p_1...p_i$. In either case, $s_1=s_2$, so we have injectivity of $\varphi$. Thus $\varphi$ is bijective on $[1,p]$ as desired, so it is a reduction.

\end{proof}

This completes the first step, showing that if two groups have reductions to amenable groups, their free product reduces to a free product of amenable groups. Now comes the harder step: showing that a free product of amenable groups reduces to an amenable group! 

The key is a construction from group theory called a wreath product. Recall the following definition of a wreath product:

\begin{definition}
Let $G$ and $H$ be groups. Let 

\[ G^H=\{ f: H \ra G \text{ a function} | \supp(f) \text{ is finite}\}\]

\noindent where $\supp (f) = \{h \in H: f(h) \neq 1\}$, and give $G^H$ a group structure by pointwise multiplication. Note that $G^H$ is isomorphic to a direct sum of $|H|$ copies of $G$, hence the notation.

Give $G^H$ an action $\alpha:H \ra Aut(G^H)$ by translation: $\left[ \alpha_h(f) \right] (h')= f(h\inv h')$.

Then we define the \ul{(restricted) wreath product}\index{wreath product} to be the semidirect product $G \wr H := G^H \rtimes_{\alpha} H$, so there is a (split) short exact sequence:

\[ 1 \ra G^H \ra G \wr H \ra H \ra 1.\]

\end{definition}

The following remark establishes the basic properties of wreath products, and are all routine to verify.

\begin{remark}

For each $g \in G$ and $h \in H$, let $g \delta_h$ denote the function in $G^H$ given by $g \delta_h(h') = \begin{cases} g & h =h'\\ 1 &h \neq h'\end{cases}$. Then there is a natural embedding of $G$ into $G \wr H$ by $g \mapsto (g \delta_{1_H}, 1_H)$ and there is a natural embedding of $H$ into $G \wr H$ by $h \mapsto (1_{G^H}, h)$. In a slight abuse of notation, we will write $g$ or $g \delta_1$ for the embedded element of $G$ in $G \wr H$, and write $h$ for the embedded element of $H$ in $G \wr H$.

Now note that 

\bea
\left[\alpha_{h_2}(g\delta_{h_1}) \right] (h_3) &=& g\delta_{h_1}(h_2\inv h_3) \\
&=& \begin{cases} g & h_1 =h_2\inv h_3 \\ 1 &h_1 \neq h_2\inv h_3 \end{cases}\\
&=& \begin{cases} g & h_2h_1 =h_3 \\ 1 &h_2 h_1\neq h_3 \end{cases}\\
&=& g \delta_{h_2h_1}(h_3)
\eea

\noindent so unbinding the $h_3$, we have that $\alpha_{h_2}(g\delta_{h_1})=g \delta_{h_2h_1}$. Now,

\bea
h_2 (g \delta_{h_1}) h_2\inv  &=& (1_G, h_2) (g \delta_{h_1}, 1_H) (1_G, h_2 \inv) \\
&=& (1_G, h_2) (g \delta_{h_1}, h_2 \inv)\\
&=& (1_G, h_2) (g \delta_{h_1}, h_2 \inv)\\
&=& ( \alpha_{h_2} (g\delta_{h_1}), 1_H) \\
&=& g\delta_{h_2h_1} \\
\eea

Since $G^H$ is generated by the $g \delta_{h}$, then together (the embedded copies of) $G$ and $H$ generate $G \wr H$.

Given $P,Q$ positive cones in $G, H$, we let $P \wr Q$ denote the monoid generated by the naturally embedded copies of $P$ and $Q$ within $G \wr H$.

With these natural embeddings of $G$ and $H$ into $G \wr H$, by the universal property of free products of groups, we get a homomorphism $\varphi: G *H \ra G \wr H$ given by sending (the naturally embedded copies of) $G$ and $H$ to (the naturally embedded copies of) $G$ and $H$. We will call this the natural homomorphism of $G*H$ into $G \wr H$. Since $G \wr H$ is generated by $G$ and $H$, it is immediate that the natural homomorphism is surjective. For the same reason, $P*Q$ will be mapped surjectively onto $P \wr Q$. We will show later that $\varphi: (G*H, P*Q) \ra (G \wr H, P \wr Q)$ is a strong reduction.

Finally, note that every element of $G^H$ can be written as a (possibly empty) product of a finite number of $g_i \delta_{h_i}$ where each $g_i$ is nonunital and each $h_i$ is distinct, and this presentation is unique up to a permutation of the $g_i \delta_{h_i}$. Note that the element $g_i \delta_{h_i}$ commutes with the element $g_j \delta_{h_j}$ as long as $h_i \neq h_j$. When $f \in G^H$ is written as $f= g_1 \delta_{h_1}... g_n \delta_{h_n}$ where the $h_i$ are distinct, then

\[ f(h)= 
\begin{cases} 
g_1 & \text{ if } h=h_1 \\
g_2 & \text{ if } h=h_2 \\
\vdots \\
g_n & \text{ if } h=h_n \\
1 & \text{ otherwise}
\end{cases}.\]
\end{remark}

Finally, we will remind the reader of this fact about amenability of wreath products:

\begin{lemma} \label{Wreath Product is Amenable}
Let $G$ and $H$ be amenable groups. Then $G \wr H$ is amenable.
\end{lemma}

\begin{proof}
Recall that the class of amenable groups are closed under finite direct sums, direct limits, and extensions.

From our definition of the wreath product, there is a short exact sequence

\[ 1 \ra G^H \ra G \wr H \ra H \ra 1.\]

Recall also that $G^H$ is isomorphic to the direct sum of $|H|$ copies of $G$. Thus $G^H$ is a direct limit of $\{G^n\}_{n \in \N}$. Since amenable groups are closed under finite direct sum and $G$ is amenable by hypothesis, then each $G^n$ is amenable, and since amenable groups are closed under direct limit, then $G^H$ is amenable.

Also, $H$ is amenable by hypothesis, so $G \wr H$ is an extension of an amenable group by an amenable group. Thus $G \wr H$ is amenable.

\end{proof} 

Now, we will have two results that establish why $\varphi: (G_1*G_2, P_1*P_2) \ra (G_1 \wr G_2, P_1 \wr P_2)$ is a reduction.

\begin{lemma}
Let $(G_1, P_1)$ and $(G_2,P_2)$ be ordered groups, and let $p= p_1p_2...p_{2n} \in P_1 \wr P_2$ be an element such that $p_i \in P_1$ for $i$ odd and $p_i \in P_2$ for $i$ even and $p_i \neq 1$ for $1<i<n$. Then, as an element of $P_1 \wr P_2 \subset G_1^{G_2} \rtimes G_2$, $p=(f, p_{\leq2n})$ where $p_{\leq2n}=p_2p_4...p_{2n}$ and 

\[ f(h) = 
\begin{cases} 
p_1 & \text{ if } h=1_{G_2} \\
p_3 & \text{ if } h=p_2 \\
p_5 & \text{ if } h=p_2p_4 \\
\vdots \\
p_{2n-1} & \text{ if } h=p_2...p_{2n-2} \\
1 & \text{ otherwise}
\end{cases}
\]

\end{lemma}

\begin{proof}
By elementary algebra, we may rewrite the product $p= p_1p_2p_3... p_{2n}$ as:

%\[ p= p_1p_2p_3... p_{2n} = p_1 (p_2 p_3 p_2 \inv)(p_2p_4 p_5  p_4 \inv p_2\inv)... (p_2 p_4... p_{2n-2} p_{2n-1} p_{2n-2}\inv ... p_4\inv p_2\inv) (p_2p_4... p_{2n})\]

\[ p= p_1p_2p_3... p_{2n} = p_1 (p_2 p_3 p_2 \inv)... (p_2 p_4... p_{2n-2} p_{2n-1} p_{2n-2}\inv ... p_4\inv p_2\inv) (p_2p_4... p_{2n})\]

With the shorthand that $p_{\leq0}= 1_H$ and $p_{\leq2i}=p_2p_4...p_{2i}$ for $1 \leq i \leq n$, we then have that

\[ p =\left[ (p_{\leq0}p_1p_{\leq0}\inv) (p_{\leq2} p_3 p_{\leq2} \inv) (p_{\leq4} p_5 p_{\leq4} \inv) ... (p_{\leq2n-2} p_{2n-1} p_{\leq2n-2} \inv ) \right]p_{\leq2n}.\]

Let $f=(p_{\leq0}p_1p_{\leq0}\inv) (p_{\leq2} p_3 p_{\leq2} \inv) (p_{\leq4} p_5 p_{\leq4} \inv) ... (p_{\leq2n-2} p_{2n-1} p_{\leq2n-2} \inv )$ be the bracketed term, so then $p$ is represented in $G_1^{G_2} \rtimes G_2$ by $(f,p_{\leq2n})$. It then suffices to show that $f$ assumes the values as claimed.

Recall that for $i$ odd, $p_i$ is embedded as $p_i \delta_{1_H}$, and for $i$ even,  $p_i$ is embedded as an element of $G_2$. By the calculation in the above remark, we then have that $p_{\leq2i} p_{2i+1} p_{\leq2i}\inv = p_{\leq2i} (p_{2i+1} \delta_{1_H})p_{\leq2i} \inv = p_{2i+1}\delta_{p_{\leq2i}}$, so $f= (p_1 \delta_{p_{\leq0}}) (p_3 \delta_{p_{\leq2}}) ... (p_{2n-1} \delta_{p_{\leq2n-2}})$. 

Note that $p_{\leq 2i+2} =p_{\leq 2i} p_{2i+2}$, and since $p_{2i+2} > 1$ for $0 \leq i <n-1$, we have that $p_{\leq2i}< p_{\leq 2i+2}$ for $0 \leq i <n-1$, and in particular the $p_{\leq2i}$ are distinct (except possibly that $p_{\leq2n-2}=p_{\leq2n}$).

Since the $\{p_{\leq0},..., p_{\leq2n-2} \}$ are distinct, then by the above remark, we know that 

\[ f= (p_1 \delta_{p_{\leq0}}) (p_3 \delta_{p_{\leq2}}) ... (p_{2n-1} \delta_{p_{\leq2n-2}})\]

\noindent is the unique representation of $f$ as a product of $g_i \delta_{h_i}$ (except possibly that the $p_1 \delta_{p_{\leq0}}$ term is equal to 1), so 

\[
f(h)= \begin{cases}
p_1 &\text{ if } h=p_{\leq0} \\
p_3 &\text{ if } h=p_{\leq2} \\
\vdots \\
p_{2n-1} &\text{ if } h=p_{\leq2n-2} \\
1 & \text{ otherwise}
\end{cases}
\]

\noindent as desired.

\end{proof}

\begin{proposition} \label{Free to Wreath Reduction}
Let $(G_1, P_1)$, $(G_2, P_2)$ be two ordered groups. Let $\varphi$ denote the natural homomorphism from $G_1* G_2$ onto $G_1 \wr G_2$. If $p,q  \in P:= P_1 *P_2$ and $\varphi(p)=\varphi(q)$, then $p=q$.

In particular, $\varphi$ is a strong reduction of $G_1* G_2$ onto $G_1 \wr G_2$.
\end{proposition}

\begin{proof}
Since $p,q \in P$, then by Lemma \ref{Free Positive Cone} we may write $p=p_1p_2...p_{2n}$ where $p_i$ is in $P_1$ if $i$ is odd, and in $P_2$ if $i$ is even, and all $p_i \neq 1$ for $1<i<n$. Similarly, we may write $q=q_1q_2...q_{2m}$ where $q_i$ is in $P_1$ if $i$ is odd, and in $P_2$ if $i$ is even, and all $p_i \neq 1$ for $1<i<m$.

Now, by the previous Lemma, we may write $\varphi(p)= (f_p, p_{\leq2n})$ and $\varphi(q)= (f_q, q_{\leq2m})$ where $p_{\leq2i}=p_2p_4...p_{2i}$, $q_{\leq2i}=q_2q_4...q_{2i}$, 

\[
f_p(h)= \begin{cases} 
p_1 & \text{ if } h=1_{G_2} \\
p_3 & \text{ if } h=p_2 \\
p_5 & \text{ if } h=p_2p_4 \\
\vdots \\
p_{2n-1} & \text{ if } h=p_2...p_{2n-2} \\
1 & \text{ otherwise}
\end{cases}
\text{ and } 
f_q(h)= \begin{cases} 
q_1 & \text{ if } h=1_{G_2} \\
q_3 & \text{ if } h=q_2 \\
q_5 & \text{ if } h=q_2q_4 \\
\vdots \\
q_{2m-1} & \text{ if } h=q_2...q_{2m-2} \\
1 & \text{ otherwise}
\end{cases}
\]

Since $\varphi(p)=\varphi(q)$, then $f_p=f_q$, so $p_1=q_1$. Furthermore, $\{p_2, p_2p_4,... ,p_2...p_{2n-2} \}= \supp (f_p) \setminus \{1\} = \supp (f_q) \setminus \{1\}=\{q_2, q_2q_4,... ,q_2...q_{2m-2} \}$, and these sets are strictly increasing sequences, so it must be that $n=m$ and $p_2=q_2$, $p_2p_4=q_2q_4$, and so on. By cancelling common terms, we get that $p_{2i}=q_{2i}$ for $1 \leq i <n$. Returning to our expressions for $f_p$ and $f_q$, we can now compare their values at $p_2=q_2, p_2p_4=q_2q_4$, etc  to get that $p_3=q_3, p_5=q_5$, and so on. Finally, recall that $(f_p, p_{\leq2n}) =\varphi(p)= \varphi(q)= (f_q, q_{\leq2n})$ so $p_{\leq2n}=q_{\leq2n}$ and by cancelling the common factor of $p_2...p_{2n-2}= q_2...q_{2n-2}$, we have that $p_{2n}=q_{2n}$.

Thus for all $1 \leq i \leq 2n$, we have that $p_i=q_i$, so $p=q$, as desired.

For the ``in particular'', by Corollary \ref{Strong Reduction} it suffices to show that $\varphi$ is a bijection of $P_1*P_2$ onto $P_1 \wr P_2$. It is immediate that $\varphi$ is surjective, and by the previous part of the proposition it must be injective. Thus $\varphi$ is a bijection of $P_1*P_2$ onto $P_1 \wr P_2$, so $\varphi$ is a strong reduction as desired.

\end{proof}

We have done the hard work already, so we may now reap our rewards:

\begin{proposition} \label{Amenable Reduction of Free Product}
Let $(G, P), (H,Q)$ be ordered groups, and suppose that $G$ and $H$ are amenable. Then the natural homomorphism $\varphi:(G*H, P *Q) \ra (G \wr H, P \wr Q)$ is a strong reduction to an amenable group.
\end{proposition}

\begin{proof}
By Propositon \ref{Free to Wreath Reduction}, $\varphi$ is a strong reduction. By Lemma \ref{Wreath Product is Amenable}, $G \wr H$ is amenable, so $\varphi$ is a strong reduction to an amenable group.

\end{proof}

And combining this with a previous result gives:

\begin{corollary} \label{Free Product Preserves Amenable Reduction}
The class of ordered groups which reduce to an amenable group is closed under finite free products.
\end{corollary}

\begin{proof}
It suffices to check two-term free products. Suppose that $(G_1,P_1)$ and $(G_2,P_2)$ are ordered groups with reductions $\varphi_1, \varphi_2$ to amenable groups $(H_1, Q_1)$ and $(H_2, Q_2)$ respectively. Let $G=G_1*G_2$, $P=P_1*P_2$, $H=H_1*H_2$, and $Q=Q_1*Q_2$.

Now by Lemma \ref{Reduction of Free Product} there is a reduction $\varphi$ of $(G,P)$ onto $(H,Q)$. By Proposition \ref{Amenable Reduction of Free Product}, since $(H,Q)$ is the free product of two amenable groups, then there is a strong reduction $\varphi':(H_1*H_2, Q_1 *Q_2) \ra (H_1 \wr H_2, Q_1 \wr Q_2)$ where $H_1 \wr H_2$ is amenable. Then $\varphi' \varphi:(G, P) \ra (F,R)$ is a composition of reductions, hence a reduction by Lemma \ref{Composition of Reductions}, and since its range is amenable, $(G,P)$ reduces to an amenable group.
\end{proof}

\subsection{Summary}
\label{sectionSummary}

To summarize our results on reductions to amenable groups: 

\AmenableReductionTheorem

\begin{proof}
If $(G,P)$ is an ordered group with $G$ amenable, then $\id_G: (G,P) \ra (G,P)$ is a reduction to an amenable group, so every amenable ordered group contains a reduction to an amenable group.

We've seen that this class is closed under hereditary subgroups in Corollary \ref{Hereditary Subgroup Preserves Amenable Reduction}, direct products in Lemma \ref{Direct Product Preserves Amenable Reduction}, and free products in Corollary \ref{Free Product Preserves Amenable Reduction}.

\end{proof}

\begin{corollary} \label{N^2*N Reduces}
Let $G=\Z^2*\Z$ and $P=\N^2*\N$. Then $(G,P)$ strongly reduces to an amenable group.
\end{corollary}

\begin{proof}
By Proposition \ref{Free to Wreath Reduction}, $(\Z^2*\Z, \N^2*\N)$ strongly reduces to $(\Z^2 \wr \Z, \N^2 \wr \N)$, which is amenable by Lemma \ref{Wreath Product is Amenable}.
\end{proof}

Although the theory we have developed gives a rich class of examples of groups which reduce to amenable groups, there are still many open questions worthy of exploration:

\begin{question}
Are there any ordered groups that reduce to an amenable group, but do not strongly reduce to an amenable group?
\end{question}

\begin{question}
Is the class of ordered groups which reduce to amenable groups closed under other group theoretic constructions (direct limits, HNN extensions, amalgamated free products, graph products, etc)? 
\end{question}

For the latter question, we would conjecture that the answer is yes, and that this should follow from this class being closed under quotients by ``positive relators'':

\begin{conjecture} \label{IncomparableQuotient}
Let $(G,P)$ be an ordered group. Let $p_1, p_2 \in P$ be incomparable (meaning $p_1 \not \leq p_2$ and $p_1 \not \geq p_2$), and let $N$ denote the normal subgroup of $G$ generated by $p_1p_2\inv$. Then:

\begin{enumerate}[label=\alph*.]
\item If $(G,P)$ reduces to an amenable group, then $(G/N, P/N)$ reduces to an amenable group.
\item If $(G,P)$ strongly reduces to an amenable group, then $(G/N, P/N)$ strongly reduces to an amenable group.
\end{enumerate}
\end{conjecture}

Of these two sub-conjectures, (b) seems more plausible: one could imagine ``carrying the quotient forward'' into the reduction. However, attempting that in the case of a non-strong reduction of (a) seems more difficult. For instance, we've seen that $(F_2,P_2)$ reduces to $(\Z, \N)$, and taking $p_1=ab, p_2=ba$ results in the relator $ab=ba$, so $F_2/N\cong \Z^2$ and $P_2/N \cong \N^2$, but $(\Z^2, \N^2)$ does not reduce to $(\Z, \N)$ or any quotient thereof. While $(\Z^2, \N^2)$ reduces to a different amenable group (itself), it is not clear in general how one would avoid this type of issue.

\section{Gauge-Invariant Uniqueness for $P$-graphs}

In this section, our goal is to prove a gauge-invariant uniqueness theorem for $P$-graphs: there is exactly one representation of a $P$-graph which is $\Lambda$-faithful, tight, and has a gauge coaction. However, we will show in Lemma \ref{GIUT Failure} that the gauge-invariant uniqueness theorem can fail in general, so we require an additional hypothesis. That additional hypothesis is that $(G,P)$ can reduce to an amenable ordered group, and we will prove a gauge-invariant uniqueness theorem for such $P$-graphs in Theorem \ref{uniquenesstheorem}.

\subsection{A Co-Universal Algebra}

In this section, we will prove a weaker form of a gauge invariant uniqueness theorem which says that there is exactly one representation of a $P$-graph which is $\Lambda$-faithful, tight, and has a \emph{normal} gauge coaction. This result is a slight generalization of \cite[Theorem 5.3]{brownlowe2013co} to the context of \emph{weakly} quasi-lattice ordered groups.

To prove this, we define and study the balanced algebra of a representation, culminating in Theorem \ref{Kernel Generated by Bolts 2}, which classifies the balanced algebras of a graph.

To give a brief summary of the argument:

\begin{enumerate}
\item Use the structure of ``infinite paths'' in $\Lambda$ to construct a tight, $\Lambda$-faithful representation called the ultrafilter representation $f$.
\item For a representation $t$, we define the balanced algebra $\mathcal B(t)= \closedspan\{t_\mu t_\nu^* : d(\mu)=d(\nu)\}$ (Lemma \ref{Balanced Subalgebra}). If there is a gauge coaction on $t$, there is a conditional expectation $\Phi_t: C^*(t) \ra \mathcal B(t)$ which is given by 

\[ \Phi_t (t_\mu t_\nu^*) = \begin{cases} t_\mu t_\nu^* &\text{ if } d(\mu)=d(\nu) \\ 0 &\text{ otherwise} \end{cases},\]

\noindent and this conditional expectation is faithful if and only if the gauge coaction is normal (Lemma \ref{Balanced Conditional Expectation}). We define a balanced covering to be a map $\psi^t_s: \mathcal B(t) \ra \mathcal B(s)$ given by $t_\mu t_\nu^* \mapsto s_\mu s_\nu^*$.
\item Recalling that $\mathcal T$ denotes the Toeplitz representation of $\Lambda$ from Definition \ref{Toeplitz Definition}, in Theorem \ref{Kernel Generated by Bolts 2} we show that $\ker \psi^{\mathcal T}_t$ is generated by the ``bolts'' and range projections it contains. Using this, we show in Lemma \ref{Balanced Tight is Minimal} that any tight representation is balanced covered by any $\Lambda$-faithful representation.
\item In Lemma \ref{Covering Extension} we show that in the presence of appropriate coactions, we can lift a balanced covering to a covering of the entire algebras.
\item Finally, we show in Theorem \ref{Co-Universal Algebra} that there is a unique tight, $\Lambda$-faithful representation with a normal gauge coaction, and that this representation is co-universal for $\Lambda$-faithful coacting representations in the sense of \cite[Theorem 5.3]{brownlowe2013co}.
\end{enumerate}

\subsubsection{The Ultrafilter Representation}

In this subsection, we construct the ultrafilter representation $f$ of a $P$-graph $\Lambda$, and demonstrate that it is $\Lambda$-faithful and tight. This construction follows the one in \cite[Section 3]{brownlowe2013co} and is included here for completeness.

\begin{definition}
Let $(G,P)$ be a WQLO group, and $\Lambda$ a finitely aligned $P$-graph. Recall that $\Lambda$ has a partial order given by $\alpha \leq \beta$ if $\beta \in \alpha \Lambda$. 

A \ulindex{filter} of $\Lambda$ is a nonempty subset $U \subseteq \Lambda$ such that:

\begin{enumerate}
\item [(F1)] If $\mu \in U$ and $\lambda \leq \mu $, then $\lambda \in U$ 
\item [(F2)] If $\mu, \nu \in U$, there exists a $\lambda \in U$ such that $\mu \leq \lambda$ and $\nu \leq \lambda$.
\end{enumerate}

Given a filter $U$, it is nonempty so it contains some $\mu$. By $F1$, $r(\mu) \in U$, and by $F2$ since no distinct vertices have a common extension, then $r(\mu)$ is the only vertex in $U$. Therefore, for any $\mu, \nu \in U$, $r(\mu)= r(\nu)$, so we write $r(U)$ for this unique vertex.

A \ulindex{ultrafilter} of $\Lambda$ is a maximal filter. We denote the set of filters by $\widehat \Lambda$ and the set of ultrafilters by $\widehat \Lambda _\infty$. 

\end{definition}

\begin{lemma} \label{Filters are Contained in Ultrafilters}
Every filter is contained in an ultrafilter.
\end{lemma}

\begin{proof}
The proof is a standard Zorn's Lemma argument: fix a filter $U$, and let $\widehat \Lambda_U = \{ V \in \widehat \Lambda : U \subseteq V\}$. We will show that $\widehat \Lambda_U$ contains a maximal element by Zorn's lemma, and that this maximal element of $\widehat \Lambda_U$ is also maximal in $\widehat \Lambda$.

To apply Zorn's Lemma, we must check that $\widehat \Lambda_U = \{ V \in \widehat \Lambda : U \subseteq V\}$ is nonempty and that chains in $\widehat \Lambda_U$ have an upper bound in $\widehat \Lambda_U$. The former is immediate since $U \in \widehat \Lambda_U$. If $\mathcal C$ is a chain in $\widehat \Lambda_U$, then let $W= \bigcup_{V \in \mathcal C} V$, which we will show is a filter. To check F1, if $\mu \in W$ and $\nu \leq \mu$, then $\mu \in V$ for some $V \in C$, so $\nu \in V$ and thus $\nu \in W$. To check F2, if $\mu, \nu \in W$, then $\mu \in V_1$, $\nu \in V_2$ for some $V_1, V_2 \in \mathcal C$. Since $\mathcal C$ is a chain, either $V_1 \subseteq V_2$ or $V_2 \subseteq V_1$. Without loss of generality assuming the former, we have that $\mu, \nu \in V_2$, so there is some $\lambda \in V_2 \subseteq W$ with $\mu, \nu\leq \lambda$ as desired. Since $W$ also contains $U$, then $W \in \widehat \Lambda_U$, so by Zorn's Lemma $\widehat \Lambda_U$ contains a maximal element $U_\infty$.

Finally, we will show that $U_\infty$ is an ultrafilter, meaning that we will check that $U_\infty$ is maximal in $\widehat \Lambda$. If it were not, there would be some $V \supsetneq U_\infty \supseteq U$, but then we'd have $V \in \widehat \Lambda_U$, so $U_\infty$ was not maximal in $\widehat \Lambda_U$, a contradiction. Thus $U_\infty$ is indeed maximal in $\widehat \Lambda$, so $U_\infty \in \widehat \Lambda_\infty$ as desired.

\end{proof}

\begin{lemma}
Let $U$ be an ultrafilter, and suppose $\mu \in U$, $E \subset \mu \Lambda$, and that $E$ is exhaustive for $\mu \Lambda$. Then $E \cap U$ is nonempty.
\end{lemma}

\begin{proof}
We split into two cases: U is finite and U is infinite. If $U$ is finite, it has a greatest element $\omega$ by F2. Since $\mu \leq \omega$, then $\omega \in \mu \Lambda$, and since $E$ is exhaustive for $\mu \Lambda$ there is some $\alpha \in E$ with $MCE(\alpha, \omega) \neq \emptyset$. That is, we can find some $\beta \in MCE(\alpha, \omega)$, so define $U'= \{\nu : \nu \leq \beta\}$. Then $U'$ is a filter containing $U$, and since $U$ was an ultrafilter then $U'=U$. Thus $\beta \in U$ and by F1, $\alpha \leq \beta \in U$, as desired.

Suppose instead that $U$ is infinite. Since $\Lambda$ is countable, so is $U$, so let give $U$ an enumeration $U=\{\nu_1, \nu_2,...\}$ and without loss of generality suppose $\mu=\nu_1$. By the F2 property, for each $n \in \N$, we may choose an $\eta_n$ such that $\eta_{n-1} \leq \eta_n$ if $n>1$ and $\nu_i \leq \eta_n$ for all $1 \leq i \leq n$. Then in particular $\{\eta_n\}_{n \in \N}$ is a rising sequence of elements of $U$ such that $\eta_n \in \mu \Lambda$ for all $n$, and by the F1 property $\lambda \in U$ if and only if $\lambda \leq \eta_n$ for some $n$. Finally, we may assume that the sequence $\{\eta_n\}_{n\in \N}$ is strictly increasing by passing to a subsequence.

Now, for each $\eta_n$, since $\eta_n \in \mu \Lambda$ and $E$ is exhaustive is $\mu \Lambda$, there is an $\alpha_n \in E$ such that $\eta_n$ has a common extension with $\alpha_n$. But $E$ is finite, so one $\alpha \in E$ must appear infinitely many times in the sequence $\{\alpha_n\}_{n \in \N}$. Thus there are infinitely many $n$ with $MCE(\eta_n, \alpha)$ nonempty. Our goal now is to show that $\alpha \in U$ by constructing a potentially bigger ultrafilter that contains $\alpha$, but since $U$ is already an ultrafilter it must be that $\alpha$ is already in $U$.
 
To this end, for $i \leq j$, if $\lambda \in MCE(\eta_j, \alpha)$, we can factorize $\lambda = \zeta \zeta'$ where $d(\zeta) = d(\eta_i) \vee d(\alpha)$, which implies $\zeta \in MCE(\eta_i, \alpha)$, and therefore $\lambda \in \zeta \Lambda$ for some $\zeta \in MCE(\eta_i, \alpha)$. Thus

\[ MCE(\eta_j, \alpha) = \bigcup_{\zeta \in MCE(\eta_i, \alpha)} \left[ MCE(\eta_j, \alpha) \cap \zeta \Lambda \right]. \] 

Considering $i$ as fixed but taking the union over all $j \geq i$, we have that

\bea \bigcup_{j \geq i} MCE(\eta_j, \alpha) &=&\bigcup_{j \geq i} \left( \bigcup_{\zeta \in MCE(\eta_i, \alpha)}  \left[ MCE(\eta_j, \alpha) \cap \zeta \Lambda \right]  \right) \\
&=&\bigcup_{\zeta \in MCE(\eta_i, \alpha)} \left( \bigcup_{j \geq i}MCE(\eta_j, \alpha) \cap \zeta \Lambda  \right)  
\eea

Since infinitely many $MCE(\eta_j, \alpha)$ are nonempty and the sequence $\{\eta_n\}_{n \in \N}$ is strictly increasing, the lefthand side is infinite, and thus the righthand side must have one of the terms $\bigcup_{j \geq i}MCE(\eta_j, \alpha) \cap \zeta \Lambda$ being infinite. In particular, when $i=1$, there is some $\zeta_1$ such that $MCE(\eta_j, \alpha) \cap \zeta_1 \Lambda$ is infinite. Then we may repeat a version of this argument: we have

\[ \bigcup_{j \geq 2} MCE(\eta_j, \alpha)  \cap \zeta_1\Lambda  = \bigcup_{\zeta \in MCE(\eta_2, \alpha)} \left( \bigcup_{j \geq 2}MCE(\eta_j, \alpha) \cap \zeta \Lambda \cap \zeta_1 \Lambda \right)\]

\noindent so one of the terms $\bigcup_{j \geq 2}MCE(\eta_j, \alpha) \cap \zeta \Lambda \cap \zeta_1$ must be infinite, so we take $\zeta_2$ to be the element of $MCE(\eta_2, \alpha)$ giving that infinite term. Next, since $\zeta_1 \leq \zeta_2$ then $\zeta_2 \Lambda \cap \zeta_1= \zeta_2\Lambda$, and we may repeat this process inductively.

In this way we create a sequence $\{\zeta_n\}_{n \in \N}$ with $\zeta_n \in MCE(\eta_n, \alpha)$ and $\zeta_{n-1} \leq \zeta_{n}$ for all $n$. Finally, let $U_{\infty}= \{ \lambda: \lambda \leq \zeta_n \text{ for some }n \in \N\}$. Then it is immediate to verify that $U_\infty$ is a filter. Furthermore, if $\lambda \in U$, then $\lambda \leq \eta_n$ for some $n$, so $\lambda \leq \eta_n \leq \zeta_n$, so $\lambda \in U_{\infty}$. That is, $U \subseteq U_\infty$, but $U$ was an ultrafilter, so it must be that $U=U_{\infty}$. Since $\alpha \leq \zeta_1$, then $\alpha \in U_{\infty} =U$, as desired.

In either case, we have shown that there is some $\alpha \in E \cap U$. 
\end{proof}

The following result is \cite[Lemma 3.4]{brownlowe2013co}:

\begin{lemma}
Let $(G,P)$ be a WQLO group, $\Lambda$ a finitely-aligned $P$-graph, $\lambda \in \Lambda$ and let $U$ and $V$ be filters. If $r(U)=s(\lambda)$ and if $\lambda \in V$, we define 

\[ \lambda \cdot U:= \bigcup_{\mu \in U} \{ \alpha \leq \lambda \mu\} \text{ and }\] 

\[\lambda^* \cdot V := \{ \mu \in \Lambda: \lambda \mu \in V\}\] 

Then $\lambda \cdot U$ and $\lambda^* \cdot V $ are filters and if $U$ and $V$ are ultrafilters, then so are $\lambda \cdot U$ and $\lambda^* \cdot V $. Finally, $\lambda^* \cdot (\lambda \cdot U) = U$ and $\lambda \cdot (\lambda^* \cdot V)=V$. 
\end{lemma}

We are now ready to define our $\Lambda$-faithful tight representation by letting $\Lambda$ act by translation on its set of ultrafilters:

\begin{lemma}\label{Ultrafilter Representation}
Let $(G,P)$ be a WQLO group, and $\Lambda$ a finitely-aligned $P$-graph. For $\lambda \in \Lambda$, define an operator $f_\lambda \in \mathcal B(\ell^2(\widehat \Lambda_\infty))$ by 

\[f_\lambda e_U = \begin{cases} e_{\lambda \cdot U} &\text{ if }r(U)=s(\lambda) \\ 0 &\text{ otherwise} \end{cases}.\]

Then $f$ is a representation of $\Lambda$ which is $\Lambda$-faithful and tight. We call $f$ the \ulindex{ultrafilter representation}\index{representation!of a $P$-graph!ultrafilter}.
\end{lemma}

\begin{proof}

Defining the operators $f_\lambda$ as given above, a typical inner product argument shows that $f_\lambda^* e_V= \begin{cases}e_{\lambda^* \cdot V} &\text{ if } \lambda \in V \\ 0 &\text{ otherwise} \end{cases}$.

To check the T1 operator, it is immediate that if $v \in \Lambda^0$, then $f_v$ is a projection, namely a projection onto the subspace spanned by the set of ultrafilters containing $v$. To show that the $\{f_v\}_{v \in \Lambda^0}$ are orthogonal, first note that this family commutes, and fix distinct $v,w \in \Lambda^0$. For any ultrafilter $U$, $r(U)$ is the unique vertex contained in $U$, so either $v \not \in U$ or $w \not \in U$. In the former case, $f_w f_v e_U= 0$ and in the latter case $f_v f_w e_U=0$, so in either case $f_vf_w=f_wf_v=0$.

The T2 operator is immediate.

The T3 operator is immediate from the previous lemma.

For the T4 operator, observe that $f_\lambda f_\lambda^*$ is projection onto the ultrafilters containing $\lambda$. Therefore, $f_\mu f_\mu^*f_\nu f_\nu^*$ is projection onto the ultrafilters containing $\mu$ and $\nu$. By the F2 and F1 properties, such an ultrafilter would contain some $\lambda \in MCE(\mu, \nu)$. In fact, it would contain exactly one such term, since if it contained two distinct $\lambda_1, \lambda_2 \in MCE(\mu, \nu)$, by the F2 property it would contain a common extension of $\lambda_1$ and $\lambda_2$, but no such extension exists by the uniqueness of factorizations. Thus the set of ultrafilters containing $\mu$ and $\nu$ is precisely the disjoint union of the ultrafilters containing a $\lambda \in MCE(\mu, \nu)$. As operators, this is to say that $f_\mu f_\mu^*f_\nu f_\nu^*=\fsum_{\lambda \in MCE(\mu, \nu)} f_\lambda f_\lambda^*$.

Thus $f$ is indeed a representation.

To show that $f$ is $\Lambda$-faithful, fix some $\lambda \in \Lambda$. Since $\{s(\lambda)\}$ is a filter, then by Lemma \ref{Filters are Contained in Ultrafilters}, there is an ultrafilter $U$ containing $s(\lambda)$, in which case $r(U)=s(\lambda)$. Then, $f_\lambda \cdot e_U= e_{\lambda \cdot U} \neq 0$, so $f_\lambda \neq 0$.

To show that $f$ is tight, fix some $\mu \in \Lambda$ and $E \subset \mu \Lambda$ which is finite and exhaustive for $\mu \Lambda$. Let $B= \fprod_{\alpha \in E} (f_\mu f_\mu^* -f_\alpha f_\alpha^*)$ be the corresponding bolt (see Definition \ref{Bolt Definition}). We'll now show that for any ultrafilter $U$, $Be_U=0$, and this implies that $B=0$ as desired.

First, if $U$ does not contain $\mu$, then it also does not contain any of the $\alpha \in E$, so $f_\mu f_\mu^*e_U= 0 = f_\alpha f_\alpha^* e_U$, and thus $Be_U=0$.

If instead $U$ does contain $\mu$, by the previous lemma, then $U$ contains one of the $\alpha \in E$. 

Then, since $\mu, \alpha \in U$, $f_\mu f_\mu^* e_U= e_U= f_\alpha f_\alpha^* U$, so $B e_U= \fprod_{\alpha \in E} (f_\mu f_\mu^*-f_\alpha f_\alpha^*) e_U= 0$. 

We have now shown that $Be_U=0$ for all ultrafilters $U$, so $B=0$ as an operator in $\mathcal B(\ell^2(\widehat \Lambda_\infty))$. Thus $f$ is tight.

\end{proof}

Since the ultrafilter representation is $\Lambda$-faithful and tight, it is natural to ask if it also has a gauge coaction. The answer is often no, as we will show in the next example.

\begin{example}
Let $(G,P)$ be a WQLO group, and suppose that $P$ is directed, meaning that any two elements of $P$ have a common upper bound. Examples of directed positive cones include $\N^k$ and any total order. Then letting $\Lambda=P$, $\Lambda$ itself is a filter, so $\Lambda$ is the unique ultrafilter. Thus $\widehat \Lambda_\infty= \{\Lambda\}$, so the ultrafilter representation is given by $f_\mu= 1$ for all $\mu \in \Lambda$. In particular, the ultrafilter representation fails to have a gauge coaction if there are distinct $\mu, \nu \in \Lambda=P$, since such a coaction $\delta$ cannot send $f_\mu=f_\nu$ to the distinct elements $f_\mu \otimes U_{d(\mu)}$ and $f_\mu \otimes U_{d(\nu)}$.

\end{example}

\subsubsection{Balanced Algebras and Balanced Coverings}

We will now discuss a particularly nice AF subalgebra of a $P$-graph algebra. Our approach in this section mimics \cite[Section 4]{brownlowe2013co}, although we use slightly different notation.

\begin{lemma} \label{Balanced Subalgebra}
Let $(G,P)$ be a WQLO group, $\Lambda$ a finitely-aligned $P$-graph, and $t$ a representation of $\Lambda$. Then $\mathcal B(t)= \closedspan \{ t_\mu t_\nu^* : d(\mu) =d(\nu)\}$ is a closed $*$-subalgebra of $C^*(t)$. We call $\mathcal B(t)$ the \ulindex{balanced (sub)algebra}.
\end{lemma}

\begin{proof}
It is immediate that $\mathcal B(t)$ is a closed $*$-invariant subspace, so it suffices to check that it is closed under multiplication. To this end, suppose $t_\mu t_\nu^*, t_\alpha t_\beta^*$ satisfy $d(\mu)=d(\nu)$ and $d(\alpha)=d(\beta)$. Then,

\bea
(t_\mu t_\nu^*) (t_\alpha t_\beta^*) &=& (t_\mu t_\nu^*)( t_\nu t_\nu^* t_\alpha t_\alpha^*) (t_\alpha t_\beta^*) \\
&=& (t_\mu t_\nu^*)( \fsum_{\lambda \in MCE(\nu, \alpha)} t_\lambda t_\lambda^*)(t_\alpha t_\beta^*) \\
&=& \fsum_{\lambda \in MCE(\nu, \alpha)} t_{\mu (\nu \inv \lambda)} t_{\beta (\alpha \inv \lambda)}^* \\
\eea

\noindent where $\nu \inv \lambda$ denotes the unique path such that $\nu (\nu \inv \lambda)=\lambda$ and similarly $\alpha \inv \lambda$ denotes the unique path such that $\alpha (\alpha \inv \lambda)=\lambda$. Then note that 

\[ d(\mu (\nu \inv \lambda)) = d(\mu) d(\nu)\inv d(\lambda) = d(\lambda) = d(\alpha) d(\beta) \inv d(\lambda) = d(\alpha (\beta\inv \lambda))\]

\noindent so the product is in $\mathcal B(t)$ as desired. 
\end{proof}

In the following lemma, we rephrase Lemma \ref{Graded Conditional Expectation} for the context of a gauge coaction on a $P$-graph $C^*$-algebra:

\begin{lemma} \label{Balanced Conditional Expectation}
Let $(G,P)$ be a WQLO group, $\Lambda$ a finitely-aligned $P$-graph, $t$ a representation of $\Lambda$ and suppose $t$ has a gauge coaction $\delta$. Then there is a conditional expectation $\Phi_t : C^*(t) \ra \mathcal B(t)$ such that

\[ 
\Phi_t (t_\mu t_\nu^*) = \begin{cases} t_\mu t_\nu^* &\text{ if } d(\mu)=d(\nu) \\ 0 &\text{ otherwise} \end{cases}.
\]
\end{lemma}

\begin{proof}
Let $A=C^*(t)$. For any $\mu, \nu \in \Lambda$, observe that $\delta(t_\mu t_\nu^*) = t_\mu t_\nu^* \otimes U_{d(\mu) d(\nu) \inv}$, so $t_\mu t_\nu^* \in A_{d(\mu) d(\nu) \inv}$. Then $t_\mu t_\nu^* \in A_e$ if and only if $d(\mu)=d(\nu)$, so $\mathcal B(t)= A_e$.

Let $\Phi_t : C^*(t) \ra \mathcal B(t)$ be the conditional expectation arising from Lemma \ref{Graded Conditional Expectation}. If $d(\mu) =d (\nu)$, then $t_\mu t_\nu^* \in A_e$, and $\Phi_t$ fixes its range space, so $\Phi_t(t_\mu t_\nu^*)=t_\mu t_\nu^*$. If $d(\mu) \neq d(\nu)$, then $t_\mu t_\nu^* \in A_g$ for $g = d(\mu) d(\nu) \inv \neq e$, and by Lemma \ref{Graded Conditional Expectation}, $\Phi_t$ vanishes on all $A_g$ for $g\neq e$, so $\Phi_t(t_\mu t_\nu^*)=0$, as desired.

\end{proof}

Recall from Definition \ref{Normal Coaction Definition} that $\Phi_t$ is a faithful conditional expectation if and only if $\delta$ is a normal coaction.

\begin{definition}
Let $(G,P)$ be a WQLO group, $\Lambda$ a $P$-graph, and $s,t$ two representations of $\Lambda$. Let us say that a \ul{balanced covering}\index{cover!balanced}\index{cover!of the balanced subalgebra} is a (necessarily surjective) $*$-homomorphism $\psi^t_s: \mathcal B(t) \ra \mathcal B(s)$ given by

 \[ \psi^t_s (t_\mu t_\nu^*)= s_\mu s_\nu^* \]
 
\noindent  for all $\mu, \nu \in \Lambda$ with $d(\mu)=d(\nu)$.

If such a $*$-homomorphism exists, we will write $t \geq_{bal}s$\index{$\geq_{bal}$}. It is immediate that $\geq_{bal}$ is reflexive and transitive, but it may not be a partial ordering since it may not be antisymmetric (that is, there may be two representations whose balanced algebras are isomorphic, but which are not canonically isomorphic as representations). If $t \geq_{bal} s$ and $s \geq_{bal} t$, we will write $t \cong_{bal} s$, and say that their balanced algebras are canonically isomorphic.
\end{definition}

When there may be ambiguity between a balanced covering a canonical covering, we will write $\geq_{rep}$ to clarify that we mean a canonical covering of the full $P$-graph $C^*$-algebra. Note that a canonical covering gives rise to a balanced covering by restricting the canonical covering to the balanced algebra. That is, $t \geq_{rep} s$ implies $t \geq_{bal} s$. The converse can fail, but we prove a partial converse in Lemma \ref{Covering Extension}.

\subsubsection{The Kernel of $\psi^{\mathcal T}_t$}

This section is devoted to proving the following fact which is our Theorem \ref{Kernel Generated by Bolts 2}: letting $\mathcal T$ denote the Toeplitz representation from Definition \ref{Toeplitz Definition}, then for any representation $t$, $\ker \psi^{\mathcal T}_t$ is generated (as an ideal) by the ``bolts'' (see Definition \ref{Bolt Definition}) and $\mathcal T_\mu \mathcal T_\mu^*$ it contains. This is our analogue of \cite[Theorem 4.9]{brownlowe2013co}, although slightly generalized to allow the $\ker \psi^{\mathcal T}_t$ to contain $\mathcal T_\mu \mathcal T_\mu^*$.

This result is highly involved in the sense that the proof is many times longer than the statement. The key to the argument is to show that the property ``$\ker \psi^{\mathcal T}_t \cap A$ is generated by the bolts and $\mathcal T_\mu \mathcal T_\mu^*$ it contains'' is preserved under direct limits, and that this property is true for a dense collection of subalgebras $A \subseteq\mathcal B(\mathcal T)$. This will take several lemmas to establish.

\begin{definition}\label{Bolt Definition}
Let $(G,P)$ be a WQLO group, let $\Lambda$ be a finitely aligned $P$-graph and let $t$ be a representation of $\Lambda$. We will say that a \ulindex{bolt} in $C^*(t)$ is an element of the form

\[ \fprod_{\alpha \in E} (t_\mu t_\mu^* - t_{ \alpha} t_{ \alpha}^*) \]

\noindent where $\mu \in \Lambda$ and $E \subset \mu \Lambda$ is finite and exhaustive for $\mu \Lambda$. We say it is a \ul{proper bolt}\index{bolt!proper} if $E \subset s(\mu) \Lambda \setminus s(\mu)$. 
\end{definition}

The term ``bolt'' comes from a tortured metaphor: a representation is tight if all of its bolts are fastened down (i.e. equal to 0).

\begin{lemma}\label{Generated by Bolts Preserved Under Limit}
Let $(G,P)$ be a WQLO group, let $\Lambda$ be a finitely aligned $P$-graph, and let $t$ a representation of $\Lambda$.

Let $\{A_n\}_{n \in \N}$ be an increasing sequence of subalgebras of $\mathcal B(\mathcal T)$, and let $A= \overline {\bigcup_{n \in \N} A_n}$. Suppose that for each $n \in \N$, $\ker \psi^{\mathcal T}_t \cap A_n$ is generated by the bolts and $\mathcal T_\mu \mathcal T_\mu^*$ it contains. Then $\ker \psi^{\mathcal T}_t \cap A$ is generated by the bolts and $\mathcal T_\mu \mathcal T_\mu^*$ it contains.
\end{lemma}

\begin{proof}
Let $J= \ker \psi^{\mathcal T}_t \cap A$ which is an ideal in $A$, and let $I$ denote the ideal generated by the bolts and $\mathcal T_\mu \mathcal T_\mu^*$ contained in $J$. Certainly $I \subseteq J$, and it suffices to show $I=J$.

By \cite[Lemma III.4.1]{davidson1996c},

\[ J= \overline { \bigcup_{n \in \N} (J \cap A_n)} \]

But note that $J \cap A_n=  \ker \psi^{\mathcal T}_t \cap A \cap A_n = \ker \psi^{\mathcal T}_t \cap A_n$, so $J \cap A_n$ is generated by the bolts and $\mathcal T_\mu \mathcal T_\mu^*$ it contains. Such a bolt or $\mathcal T_\mu \mathcal T_\mu^*$ is certainly in $J \supseteq J \cap A_n$, and thus $J \cap A_n \subseteq I$, and taking the union and closure, we have

\[ I \supseteq  \overline { \bigcup_{n \in \N} (J \cap A_n)} =J\]

\noindent so $I=J$ as desired.
 
\end{proof}

We will now build towards our dense family of subalgebras.

\begin{definition}
Let $(G,P)$ be a WQLO group, let $\Lambda$ be a finitely aligned $P$-graph, and let $S \subseteq \Lambda$. We will say that $S$ is \ulindex{MCE closed} if for all $\mu, \nu \in S$, $MCE(\mu, \nu) \subseteq S$.

We will say that $S$ is \ulindex{substitution closed} if for all $\mu, \nu \in S$ with $d(\mu)=d(\nu)$ and $s(\mu)=s(\nu)$, and for all $\alpha \in \Lambda$, we have that $\mu \alpha \in S$ implies $\nu \alpha \in S$.

For $S \subseteq \Lambda$, we will write $D(S)=\{d(\mu): \mu \in S\}$.
\end{definition}

As we will see in the next lemma, the condition that $S$ is MCE closed and substitution closed is the correct condition in order to make $A_S= \closedspan \{ \mathcal T_\mu \mathcal T_\nu^* : \mu, \nu \in S, d(\mu)=d(\nu)\}$ a subalgebra of the balanced algebra.

\begin{lemma} \label{Contains Bolt or Range Projection}
Let $(G,P)$ be a WQLO group and let $\Lambda$ be a finitely-aligned $P$-graph. For each $S \subseteq \Lambda$, let $A_S= \closedspan \{ \mathcal T_\mu \mathcal T_\nu^* : \mu, \nu \in S, d(\mu)=d(\nu)\}$. Then:

\begin{enumerate}
\item If $S$ is MCE closed and substitution closed, then $A_S$ is a closed $^*$-subalgebra of $\mathcal B(\mathcal T)$.
\item If $S$ is MCE closed and substitution closed, and $D(S)= \{ d(\mu): \mu \in S\}$ contains a minimal element $m$, then $S'=\{ \mu \in S: d(\mu) \neq m\}$ is MCE closed and substitution closed. 
\item Let $S, D(S), m,$ and $S'$ be as above. For any representation $t$ of $\Lambda$, if $\mu, \nu \in \Lambda^m \cap S = \{\lambda \in \Lambda \cap S: d(\lambda)=m\}$ with $s(\mu)=s(\nu)$, then $t_\mu t_\nu^* \in \psi^{\mathcal T}_t (A_{S'})$ implies that either $t_\mu=0$ or there is a finite set $E \subset \mu \Lambda \cap S'$ which is exhaustive for $\mu \Lambda$ such that the bolt $B=\fprod_{\alpha\in E} (t_\mu t_\mu^* -t_\alpha t_\alpha^*)$ is equal to $0$.
\end{enumerate}

\end{lemma}

\begin{proof}
(1): It is immediate that $A_S$ is a closed $*$-invariant subspace of $\mathcal B(\mathcal T)$. It then suffices to check that it is closed under multiplication. If $\mathcal T_{\mu} \mathcal T_{\nu}^*,\mathcal T_{\alpha} \mathcal T_{\beta}^* \in A_S$, then

\bea
(\mathcal T_{\mu} \mathcal T_{\nu}^*)(\mathcal T_{\alpha} \mathcal T_{\beta}^*) &=&  (\mathcal T_{\mu} \mathcal T_{\nu}^*) \fsum_{\lambda \in MCE(\nu, \alpha)} \mathcal T_\lambda \mathcal T_\lambda^* (\mathcal T_{\alpha} \mathcal T_{\beta}^*) \\
&=& \fsum_{\lambda \in MCE(\nu, \alpha)} \mathcal T_{\mu (\nu\inv \lambda)} \mathcal T_{\alpha (\beta \inv\lambda)}^*.
\eea

Note that since $S$ is MCE closed, then each $\lambda \in MCE(\nu, \alpha)$ is in $S$, and since $S$ is substitution closed, then $\mu (\nu\inv \lambda), \alpha (\beta \inv\lambda) \in S$. Thus $(\mathcal T_{\mu} \mathcal T_{\nu}^*)(\mathcal T_{\alpha} \mathcal T_{\beta}^*) \in A_S$, as desired.

(2): Given $\mu, \nu \in S'$, $MCE(\mu, \nu) \subseteq S$ since $S$ is MCE closed, and for $\lambda \in MCE(\mu, \nu)$, we have $d(\lambda) \geq d(\mu) \neq m$, so $d(\lambda) \neq m$, so $\lambda \in S'$. Similarly, given $\mu, \mu\alpha, \nu \in S'$ with $d(\mu)=d(\nu)$, we have that $\nu \alpha \in S$ since $S$ is substitution closed, but $d(\nu \alpha)=d(\mu \alpha) \neq m$, so $\nu \alpha \in S'$. Thus $S'$ is MCE closed, and substitution closed.

(3): Suppose that $t_\mu t_\nu^* \in \psi^{\mathcal T}_t (A_{S'})= \closedspan \{ t_\alpha t_\beta^* : \alpha, \beta \in S', d(\alpha)=d(\beta)\}$. Then there is an element of $\spanset \{ t_\alpha t_\beta^* : \alpha, \beta \in S', d(\alpha)=d(\beta)\}$ within a distance of $1$ of $t_\mu t_\nu^*$, which is to say there are $c_i \in \C$ and $\alpha_i, \beta_i \in S'$ with $d(\alpha_i)=d(\beta_i)$ such that writing $x=t_\mu t_\nu^*-\fsum_{i=1}^n c_i t_{\alpha_i} t_{\beta_i}^*$, then $\norm{x}<1$.

Now, let $y=t_\mu t_\mu^* x$, we have $\norm{y}\leq \norm{t_\mu}^2\cdot \norm{x} <1$, and 

\bea
 y&=& t_\mu t_\mu^* \left(t_\mu t_\nu^*-\fsum_{i=1}^n c_i t_{\alpha_i} t_{\beta_i}^*\right) \\
 &=& t_\mu t_\nu^*- \fsum_{i=1}^n c_i t_\mu t_\mu^* t_{\alpha_i} t_{\beta_i}^* \\
 &=& t_\mu t_\nu^*- \fsum_{i=1}^n \fsum_{\lambda \in MCE(\alpha_i, \mu)}  c_i t_\lambda t_{\beta (\alpha\inv \lambda)}\\
 &=&t_\mu t_\nu^*- \fsum_{j=1}^N c_j t_{\lambda_j} t_{\eta_j}
\eea

\noindent where the last line is simply relabelling the previous line. Note that the sum is still indeed finite since $\Lambda$ is finitely aligned. Also, since $S$ is substitution closed and MCE closed, then $\lambda_j, \eta_j \in S$ for all $j$. Since $d(\lambda_j) \geq d(\alpha_i) \neq m$ and $m$ was minimal, then $d(\lambda_j) \neq m$ and thus $\lambda_j \in S'$. Also, $\lambda_j$ is a common extension of $\mu$ and some $\alpha_i$, so $\lambda_j \in \mu \Lambda$.

Now let $E= \{\lambda_1,...,\lambda_N\}$, which is finite and satisfies $E \subseteq \mu \Lambda \cap S'$ by our previous remarks. Suppose for the sake of contradiction that $E$ were not exhaustive for $\mu \Lambda$. Then there would be some $\gamma \in \mu \Lambda$ such that $MCE(\gamma, \lambda_i)=\emptyset$ for all $\alpha_i \in E$. Then, $\norm{(t_\gamma t_\gamma^*)y} \leq \norm{t_\gamma}^2 \norm{y} <1$, and

\bea
(t_\gamma t_\gamma^*)y &=&(t_\gamma t_\gamma^*)(t_\mu t_\nu^*- \fsum_{j=1}^N c_j t_{\lambda_j} t_{\eta_j}) \\
&=&t_\gamma t_\gamma^*t_\mu t_\nu^*
\eea

\noindent since each $MCE(\lambda_i, \gamma)=\emptyset$ for $1 \leq i \leq N$. Next, since $\gamma \in \mu \Lambda$, then $\gamma = \mu \gamma'$ for some $\gamma' \in \Lambda$, so we may simplify this expression to 

\[ (t_\gamma t_\gamma^*) (t_\mu t_\nu^*) = t_{\mu \gamma'} t_{\gamma'}^* t_{\mu}^* t_\mu t_\nu^* =t_{\mu \gamma'} t_{\nu \gamma'}^*\]

which is a partial isometry. Since partial isometries have norm either 0 or 1, and we've seen it has norm less than 1, we have that $0=t_{\mu \gamma'} t_{\nu \gamma'}^*$, and since $s(\mu \gamma')=s(\mu)=s(\nu)=s(\nu \gamma')$, we have $0=t_{s(\mu \mu_1)}=t_{s(\mu)} =t_\mu^* t_\mu= t_\mu$, showing that either $E$ is exhaustive or $t_\mu=0$, as desired.

Finally, we wish to show that $B=0$, where $B$ is the bolt $B=\fprod_{\lambda\in E} (t_\mu t_\mu^* -t_\lambda t_\lambda^*)$. To this end, we will make a similar argument that $By$ is a partial isometry of norm less than 1, so it must be 0. The latter is immediate: $\norm{By} \leq \norm{B}\cdot \norm{y}<1$. Now let us simplify the expression $By$:

\bea B y &=& B(t_\mu t_\nu^*- \fsum_{j=1}^N c_j t_{\lambda_j} t_{\eta_j}) \\
&=&B t_\mu t_\nu^* - \fsum_{j=1}^N c_j \left(\fprod_{k=1}^N (t_\mu t_\mu^* - t_{\lambda_k}t_{\lambda_k}^* ) \right)  t_{\lambda_j} t_{\eta_j})\\
&=&B t_\mu t_\nu^* - \fsum_{j=1}^N c_j \left(\fprod_{k \neq j}^N (t_\mu t_\mu^* - t_{\lambda_k}t_{\lambda_k}^* ) \right)  \left((t_\mu t_\mu^* - t_{\lambda_j}t_{\lambda_j}^*) t_{\lambda_j} \right) t_{\eta_j})\\
&=&B t_\mu t_\nu^* - \fsum_{j=1}^N 0\\
&=&B t_\mu t_\nu^*
\eea

\noindent where the main simplification occurs since $(t_\mu t_\mu^* - t_{\lambda_j}t_{\lambda_j}^*) t_{\lambda_j} = t_{\lambda_j} -t_{\lambda_j}=0$.

Thus $By=B t_\mu t_\nu^*$, so it is again a partial isometry of norm strictly less than 1, and thus $By=0$. Finally, since $B t_\mu t_\nu^*=0$, then

\bea
0&=& (B t_\mu t_{\nu}^*)(B t_\mu t_{\nu}^*)^* \\
&=&B t_\mu t_\nu^* t_\nu t_\mu^* B \\
&=&B t_\mu t_\mu^* B \\
\eea
and recalling that $B$ is a subprojection of $t_\mu t_\mu^*$, we have that $0=B$, completing the proof of (3).

\end{proof}

The next lemma shows that a sufficient collection of $A_S$ have the property that $A_S \cap \ker \psi^{\mathcal T}_t$ is generated by the bolts and $\mathcal T_\mu \mathcal T_\mu^*$s it contains.

\begin{lemma}\label{Kernel Generated by Bolts 1}
Let $(G,P)$ be a WQLO group and let $\Lambda$ be a finitely-aligned $P$-graph. For each $S \subseteq \Lambda$ which is MCE closed and substitution closed, let $A_S= \closedspan \{ \mathcal T_\mu \mathcal T_\nu^* : \mu, \nu \in S, d(\mu)=d(\nu)\}$. Let $t$ be a representation of $\Lambda$, and let $\psi^{\mathcal T}_t: \mathcal B(\mathcal T) \ra \mathcal B(t)$ be the balanced covering. If $D(S)=\{d(\mu) : \mu \in S\}$ is finite, then $\ker \psi^{\mathcal T}_t \cap A_S$ is generated (as an ideal) by the bolts and $\mathcal T_\mu \mathcal T_\mu^*$s it contains.

\end{lemma}
%%%%%%%

\begin{proof}
Given a $MCE$-closed and substitution closed subset $S \subseteq \Lambda$, let $J_{S}$ denote $\ker \psi^{\mathcal T}_t \cap A_S$, which is an ideal in $A_S$. Let $I_S$ denote the ideal in $A_S$ generated by the bolts and $\mathcal T_\mu \mathcal T_\mu^*$s contained in $J_S$. Then certainly $I_S \subseteq J_S$, and our claim is equivalent to proving that $J_S/I_S=0$.

We proceed by induction on $|D(S)|$. In the base case of $|D(S)|=0$, the claim is trivial since $A_S=J_S=I_S=0$.

In the inductive case, suppose that $|D(S)|>0$, and that the claim is true for all $D'$ with $|D'|<|D(S)|$. Since $D(S)$ is finite, it has some minimal element $m$. First consider the case where $S^m=\{ \mu \in S: d(\mu)=m\}$ is finite. Note that $S^m$ is itself substitution closed and $MCE$-closed and that $A_{S}=A_{S'}+A_{S^m}$.

Fix some $x+I_S \in J_S/I_S$, in which case $x \in J_S$. Since $x \in J_S \subseteq A_S= A_{S'}+A_{S^m}$, we may write $x= x_m+ x'$, where $x_m \in A_{S^m}$ and $x' \in A_{S'}$. Additionally, since $A_{S^m}$ is finite dimensional, then we may write $x_m= \fsum_{i=1}^n c_{i} \mathcal T_{\mu_i} \mathcal T_{\nu_i}^*$ where $d(\mu_i)=d(\nu_i)=m$ for all $i$. There are potentially many such representations of $x+I_S$ as $x_m+x'+I_S$. We will assume without loss of generality we have chosen the decomposition with the fewest terms in the sum $x_m= \fsum_{i=1}^n c_{i} \mathcal T_{\mu_i} \mathcal T_{\nu_i}^*$, and our goal is to show that in fact there are zero terms in the sum and thus $x_m=0$.

To this end, suppose for the sake of contradiction that $x_m= \fsum_{i=1}^n c_{i} \mathcal T_{\mu_i} \mathcal T_{\nu_i}^*$ has a nonzero summand, and without loss of generality that it is the first summand. That is, suppose $c_{1} \mathcal T_{\mu_1} \mathcal T_{\nu_1}^* \neq 0$.

Now, 

\bea 
(\mathcal T_{\mu_1} \mathcal T_{\mu_1}^*)x(\mathcal T_{\nu_1} \mathcal T_{\nu_1}^*) &=& (\mathcal T_{\mu_1} \mathcal T_{\mu_1}^*) \left(  \fsum_{i=1}^n c_{i} \mathcal T_{\mu_i} \mathcal T_{\nu_i}^* +x'\right) (\mathcal T_{\nu_1} \mathcal T_{\nu_1}^*)\\
&=&  c_{1} \mathcal T_{\mu_1} \mathcal T_{\nu_1}^* +(\mathcal T_{\mu_1} \mathcal T_{\mu_1}^*) x' (\mathcal T_{\nu_1} \mathcal T_{\nu_1}^*).
\eea

Since $x \in J_S \subseteq \ker \psi^{\mathcal T}_t$, by applying $\ker \psi^{\mathcal T}_t$ to both sides, we have that

\[ 0 = c_i t_{\mu_1} t_{\nu_1}^* + \psi^{\mathcal T}_t((\mathcal T_{\mu_1} \mathcal T_{\mu_1}^*) x' (\mathcal T_{\nu_1} \mathcal T_{\nu_1}^*))\]

\noindent so solving for $ t_{\mu_1} t_{\nu_1}^*$, we have

\[  t_{\mu_1} t_{\nu_1}^*=\psi^{\mathcal T}_t( \frac{-1}{c_1} (\mathcal T_{\mu_1} \mathcal T_{\mu_1}^*) x' (\mathcal T_{\nu_1} \mathcal T_{\nu_1}^*)) \]

\noindent but $x' \in A_{S'}$, and $A_{S^m} A_{S'} \subseteq A_{S'}$, so the righthand side is an element of $\psi^{\mathcal T_t}(A_{S'})$. Now by part (3) of Lemma  \ref{Contains Bolt or Range Projection}, we have either $t_{\mu_1}=0$ or there is a finite exhaustive set $E \subset {\mu_1} \Lambda$ with $d(\alpha) \in S \setminus \{m\}$ for all $\alpha \in E$, and such that the bolt $B=\fprod_{\alpha\in E} (t_{\mu_1} t_{\mu_1}^* -t_\alpha t_\alpha^*)$ is equal to $0$.

In the former case, $\mathcal T_{\mu_1} \mathcal T_{\mu_1}^* \in I_S$, so taking $y=c_1 (\mathcal T_{\mu_1} \mathcal T_{\mu_1}^*)\mathcal T_{\mu_1}\mathcal T_{\nu_1}^*=c_1 \mathcal T_{\mu_1}\mathcal T_{\nu_1}^*$, we have $y \in I_S$, so $x+I_S= (x_m-y) +x'+I_S$ is a decomposition with one fewer terms in the sum $x_m-y$, a contradiction. In the latter case, $\fprod_{\alpha \in E}( \mathcal T_{\mu_1} \mathcal T_{\mu_1}^* - \mathcal T_{\alpha} \mathcal T_{\alpha}^*)$ is a bolt in $\ker \pi^{\mathcal T}_t$, so $B_1=\fprod_{\alpha \in E}( \mathcal T_{\mu_1} \mathcal T_{\mu_1}^* - \mathcal T_{\alpha} \mathcal T_{\alpha}^*) \in I_S$. Let $y= c_1 B_1 \mathcal T_{\mu_1}\mathcal T_{\nu_1}^*$, and then we may write $x+I_S=x_m-y+x'+I_S$ which will have zeroed out the $c_1 \mathcal T_{\mu_1} \mathcal T_{\nu_1}^*$ summand in the $x_m$ term (while possibly adding more terms to the $x'$ part, which is acceptable), again a contradiction.

In either case, we found a contradiction with $x_m$ being the representative with the fewest summands. Thus $x_m$ had no terms in its summation, and thus $x_m=0$. That is to say, $x=x_m+x'=x'$, where $x' \in A_{S'}$, and therefore $x \in A_{S'} \cap \ker^{\mathcal T}_t =J_{S'}$. By the inductive hypothesis, since $D'=\{d(\mu) : \mu \in S'\}= D(S) \setminus \{m\}$ has fewer terms than $D(S)$, we have that $J_{S'}=I_{S'}$ so $x+I_S \in I_{S'}+I_{S}=I_S$, so $J_S \subseteq I_S$, as desired.

Now in the case that $S^m$ is not finite, we let $F_n$ be an increasing sequence of finite subsets of $S^m$ with $S^m = \bigcup F_n$, in which case $S' \cup F_n$ is MCE closed and substitution-closed, so the previous case applies to $A_{S' \cup F_n}$. But $A_S= \overline {\bigcup_{n=1}^\infty A_{S' \cup F_n}}$, and so the general case follows from Lemma \ref{Generated by Bolts Preserved Under Limit}.
\end{proof}

We can now add a second dash of Lemma \ref{Generated by Bolts Preserved Under Limit} to get our desired result:

\begin{theorem} \label{Kernel Generated by Bolts 2}
Let $(G,P)$ be a WQLO group and let $\Lambda$ be a finitely-aligned $P$-graph. Let $t$ be a representation of $\Lambda$, and $\psi^{\mathcal T}_t: \mathcal B(\mathcal T) \ra \mathcal B(t)$ be the balanced covering. Then $\ker \psi^{\mathcal T}_t $ is generated (as an ideal) by the bolts and $\mathcal T_\mu \mathcal T_\mu^*$s it contains.

\end{theorem}

\begin{proof}
As above, for any $S \subseteq \Lambda$, let $D(S)=\{d(\mu) : \mu \in \Lambda\}$. Note that taking a substitution preserves the degree of a path, and taking minimal common extensions of paths can take joins of degrees (that is, replace paths of length $p$ and $q$ with paths of length $p \vee q$). Therefore, if $F \subseteq \Lambda$ and $S$ is the set of all paths in $\Lambda$ obtained by taking substitutions and MCEs of $F$, then $D(S) \subseteq \{ p_1 \vee p_2 \vee... \vee p_n : p_i \in D(F)\}$. In particular, if $F$ is finite, then $|D(S)| \leq 2^{|F|}<\infty$, so any finite $F \subseteq \Lambda$ is contained in a set $S \subseteq \Lambda$ which is MCE closed and substitution closed and for which $D(S)$ is finite.

Since $\Lambda$ is countable, we can enumerate the elements as $\lambda_1, \lambda_2,...$. For each $n \in \N$, let $F_n=\{\lambda_1,..., \lambda_n\}$, and let $S_n$ denote the set of all paths in $\Lambda$ obtained by taking substitutions and MCEs of $F_n$.

Then $S_n$ is substitution closed and MCE closed and by the above argument, $D(S_n)<\infty$. Therefore, by Lemma \ref{Kernel Generated by Bolts 1}, $\ker \psi^{\mathcal T}_t \cap A_{S_n}$ is generated by the bolts and $\mathcal T_\mu \mathcal T_\mu^*$ it contains. 

But recall that $\mathcal B(\mathcal T) = \closedspan \{ \mathcal T_\mu \mathcal T_\nu^* : d(\mu)=d(\nu)\}$, and for any $\mu, \nu \in \Lambda$ with $d(\mu)=d(\nu)$, since we have an enumeration, there exists some $m, n \in \N$ with $\mu= \lambda_m, \nu=\lambda_n$, so $\mu, \nu \in F_{\max(m,n)} \subset S_{\max(m,n)}$, so $ \mathcal T_\mu \mathcal T_\nu^* \in A_{S_{\max(m,n)}}$. Thus $\mathcal B(\mathcal T) = \overline{ \bigcup_{n \in \N} A_{S_n}}$. 

Then by Lemma \ref{Generated by Bolts Preserved Under Limit}, $\ker \psi^{\mathcal T}_t \cap \mathcal B(\mathcal T) =\ker \psi^{\mathcal T}_t$ is generated by the bolts and $\mathcal T_\mu \mathcal T_\mu^*$ it contains. 
\end{proof}

The above result classifies ideals in the balanced algebra $\mathcal B(\mathcal T)$, and when combined with the Factors Through Theorem (Lemma \ref{Factors Through Theorem}) shows that two balanced algebras are isomorphic if and only if they have the same set of bolts and range projections which are set equal to 0. We make use of this fact in the following lemma:

\begin{lemma} \label{Balanced Tight is Minimal}
Let $(G,P)$ be a WQLO group, $\Lambda$ a $P$-graph, and let $s$ and $t$ be two representations of $(G,P)$. Then

\begin{enumerate}
\item If  $t$ is tight and $s$ is $\Lambda$-faithful, then $t \leq_{bal} s$. 
\item If $t$ and $s$ are both tight and $\Lambda$-faithful, then $t \cong_{bal} s$.
\end{enumerate}
\end{lemma}

\begin{proof}
For the first claim, let $\mathcal T$ denote the universal representation of $\Lambda$. Then, there are balanced coverings $\psi^{\mathcal T}_{s}: \mathcal B(\mathcal T) \ra \mathcal B(s)$ and $\psi^{\mathcal T}_{t}: \mathcal B(\mathcal T) \ra \mathcal B(t)$. We wish to use the Factors Through Theorem (Lemma \ref{Factors Through Theorem}) to conclude there is a balanced covering $\psi^s_t: \mathcal B(s) \ra \mathcal B(t)$ given by $s_\mu s_\nu^* \ra t_\mu t_\nu^*$, which we can do if and only if $\ker \psi^{\mathcal T}_{s} \subseteq \ker \psi^{\mathcal T}_{t}$. 

By Theorem \ref{Kernel Generated by Bolts 2}, $\ker \psi^{\mathcal T}_{s}$ and $ \ker \psi^{\mathcal T}_{t}$ are generated by the bolts and $\mathcal T_\mu \mathcal T_\mu^*$ they contain. Since $s$ is $\Lambda$-faithful, its kernel contains no $\mathcal T_\mu \mathcal T_\mu^*$ (and possibly some bolts), and since $t$ is tight, its kernel contains every bolt (and possibly some $\mathcal T_\mu \mathcal T_\mu^*$). Thus $\ker \psi^{\mathcal T}_{s} \subseteq \ker \psi^{\mathcal T}_{t}$, so $t \leq_{bal} s$ as desired.

For the second claim, by applying (1) twice, we have $t \leq_{bal} s$ and $s \leq_{bal} t$, so $t \cong_{bal} s$.
\end{proof}

\subsubsection{Normalizations of Coactions}

We will now touch briefly on the normalization of a coaction. For a more thorough introduction, the reader is directed to \cite[Appendix A.7]{echterhoff2006categorical}

\begin{remark}
Given a discrete coaction $(A, G, \delta)$, define $j_A: A \ra A \otimes C^*_r(G)$ by $j_A = (\id_A \otimes \pi^{U}_L) \circ \delta$, which is a $*$-homomorphism. 

Let $A^n=j_A(A) \cong A/\ker j_A$. Then by Lemma A.55 and Definition A.56 of \cite{echterhoff2006categorical} there is a coaction $(A^n, G,\delta^n)$ which is normal, and such that $j_A$ is $\delta$-$\delta^n$ equivariant, meaning that $\delta^n \circ j_A = (j_A \otimes \id_{G}) \circ \delta$. We call $(A^n, G, \delta^n)$ the \ul{normalization}\index{normalization!of a coaction} of $(A,G, \delta)$.

Note that \cite{echterhoff2006categorical} uses a definition of normality which is equivalent to ours, but not identical. Readers may wish to read \cite[Lemma 1.4]{quigg1996discrete} for a proof of the equivalence of the two definitions.

\end{remark}

The following lemma is analogous to Lemma \ref{Coactionization} but refers to a ``normalization'' process for representations instead of a ``coactionization'' process. Note that while coactionization made the representation larger with respect to $\leq_{rep}$, normalization will make the representation smaller.

\begin{lemma}\label{Normalization}
Let $(G,P)$ be a WQLO group, and $\Lambda$ a $P$-graph. Let $t$ be a representation of $\Lambda$ with a gauge coaction $\delta$. Then,

\begin{enumerate}
\item Let $j_A: C^*(t) \ra C^*(t)^n$ denote the quotient map taking the cosystem $(C^*(t), G, \delta)$ to its normalization $(C^*(t)^n, G, \delta^n)$, and for $\lambda \in \Lambda$ let $\tilde t_\lambda := j_A(t_\lambda)$. Then $\tilde t$ is a representation of $\Lambda$.
\item There is a canonical covering $t \mapsto \tilde t$.
\item The coaction $\delta^n$ on $C^*(\tilde t)$ is normal and is a gauge coaction.
\item $t$ is canonically isomorphic to $\tilde t$ if and only if $\delta$ is a normal coaction.
\item There is a balanced covering $\psi^t_{\tilde t}: \mathcal B(t) \ra \mathcal B(\tilde t)$, which is a balanced isomorphism.
\item $\tilde t$ is $\Lambda$-faithful (respectively, tight) if and only if $t$ is.
\end{enumerate}
\end{lemma}

\begin{definition}
We call the representation $\tilde t$ the given in the previous lemma the \ul{normalization}\index{normalization!of a representation} of $t$.
\end{definition}

\begin{proof}[Proof of \ref{Normalization}]
(1) The fact that $\tilde t_\lambda$ is a representation is immediate since it is the image of a representation under a homomorphism.

(2) $j_A$ is a canonical covering.

(3) $\delta^n$ is normal by construction. To show it is a gauge coaction, recall from Definition A.56 and Lemma A.55 of \cite{echterhoff2006categorical} that the canonical covering $j_A$ is $\delta$-$\delta^n$ equivariant, meaning that $\delta^n \circ j_A = (j_A \otimes \id_{G}) \circ \delta$. Then, for any $\mu \in \Lambda$, 

\bea 
\delta^n( \tilde t_{\mu}) &=& \delta^n (j_A(t_\mu)) \\
&=& (j_A \otimes \id_{G}) \circ \delta(t_\mu) \\
&=& (j_A \otimes \id_{G}) (t_\mu \otimes U_{d(\mu)}) \\
&=& \tilde {t}_\mu \otimes U_{d(\mu)} \\
\eea
\noindent as desired.

(4) By \cite[Lemma 1.4]{quigg1996discrete}, a coaction is normal (in the sense of its conditional expectation $\Phi$ being faithful) if and only if $j_A$ is injective, which is equivalent to $j_A$ being a canonical isomorphism.

(5) The balanced covering $\psi^t_{\tilde t}$ exists because it is a restriction of the canonical covering $j_A$ from part (2). Since it is a balanced covering, it is surjective, so it suffices to show $\psi^t_{\tilde t}$ is injective. To this end, recall that $\psi^t_{\tilde t}= j_A \mid_{\mathcal B(t)}$ and that $j_A := (\id_A \otimes \pi^U_L) \circ \delta$. Now for $x \in \mathcal B(t)$ we have $\delta(x)= x \otimes U_e$, so $j_A(x)=( (\id_A \otimes \pi^U_L) \circ \delta) (x)= x \otimes L_e=x\otimes 1$. Then since $\norm{\cdot}$ is a $C^*$-cross norm, $\norm{j_A(x)}= \norm{x} \cdot \norm{1}= \norm{x}$, so $x=0$ if and only if $\psi^t_{\tilde t}(x)=j_A(x)=0$, as desired. 

(6) Having proven (5), this is immediate: since tightness is a relation among elements of the balanced algebra, $t$ is tight if and only if $\tilde t$ is tight. Similarly, $t$ is $\Lambda$-faithful if and only if every $t_\lambda \neq 0$ if and only if every $t_\lambda ^* t_\lambda \neq 0$, which is a relation among elements of the balanced algebra, so $t$ is $\Lambda$-faithful if and only if $\tilde t$ is.

\end{proof}

\begin{corollary} \label{Normal Coactionization}
Let $(G,P)$ be a WQLO group, and let $\Lambda$ be a finitely aligned $P$-graph. For any $\Lambda$-faithful, tight representation $t$ of $\Lambda$, the normalization of the coactionization of $t$ is a $\Lambda$-faithful, tight representation with a normal gauge coaction.
\end{corollary}

\begin{proof}
Recall that $t'$ denotes the coactionization of $t$ as in Lemma \ref{Coactionization} and let $\widetilde{t'}$ denoting the normalization of $t'$ as in Lemma \ref{Normalization}. By those two lemmas, $\widetilde{t'}$ has a normal gauge coaction, and since $t$ is $\Lambda$-faithful and tight, so is $\widetilde{t'}$.
\end{proof}

\subsubsection{The Co-Universal Algebra}

The following result shows that a balanced covering extends to a covering of the entire algebras in the context of a normal coaction:

\begin{lemma} \label{Covering Extension}
Let $(G,P)$ be a WQLO group, let $\Lambda$ be a finitely-aligned $P$-graph, and let $s$ and $t$ be two representations of $\Lambda$. If $s \leq_{bal} t$, $s$ and $t$ have gauge coactions, and the gauge coaction on $s$ is normal, then $s \leq_{rep} t$.
\end{lemma}

\begin{proof}

Our argument will make use of several functions, organized according to the following diagram (which will commute):

\begin{center}
\begin{tikzcd}
C^*(\mathcal T) \arrow[r, "\pi^{\mathcal T}_{t}"] \arrow[d, "\pi^{\mathcal T}_{s}"] &
C^*(t) \arrow[r, "\Phi_t"] \arrow[dl, dashed, "\exists \pi^t_s"] &
 \mathcal B(t)\arrow[d, "\psi^t_s"]\\
C^*(s)  \arrow[rr, "\Phi_s"] &&
\mathcal B(s)
\end{tikzcd}
\end{center}

Here $\mathcal T$ denotes the universal representation of $\Lambda$, and $\pi^{\mathcal T}_t$ and $\pi^{\mathcal T}_s$ are the canonical coverings of $C^*(t)$ and $C^*(s)$ respectively. We wish to show the existence of the canonical covering $\pi^t_s$. The algebras $\mathcal B(t)$ and $\mathcal B(s)$ are the balanced subalgebras as in Lemma \ref{Balanced Subalgebra} and the maps $\Phi_{t}: C^*(t) \ra \mathcal B(t)$ and $\Phi_s: C^*(s) \ra \mathcal B(s)$ are the conditional expectations arising from Lemma \ref{Balanced Conditional Expectation}. From that lemma, these conditional expectations satisfy

\[ \Phi_{ t} ( t_\mu  t_\nu^*)= \begin{cases}  t_\mu  t_\nu^* & \text{ if } d(\mu)=d(\nu) \\
0 & \text{ otherwise} \end{cases} \text{ and similarly } \Phi_{s} (s_\mu s_\nu^*)= \begin{cases} s_\mu s_\nu^* & \text{ if } d(\mu)=d(\nu) \\
0 & \text{ otherwise} \end{cases}.\] 

Also note that by Definition \ref{Normal Coaction Definition}, $\Phi_s$ is faithful because the gauge coaction on $s$ is normal. Finally, $\psi^t_s$ is the balanced covering that exists because $t \geq_{bal} s$.

We will now verify that this diagram commutes. That is, we will show that 

\[ \psi^{t}_s \circ \Phi_{t}  \circ \pi^{\mathcal T}_{t} = \Phi_s \circ \pi^{\mathcal T}_s. \]

Since $C^*(\mathcal T)= \overline \spanset\{ \mathcal T_\mu \mathcal T_\nu : \mu, \nu \in \Lambda\}$ and all the maps are linear and continuous, it suffices to check that the two functions agree on each $\mathcal T_\mu \mathcal T_\nu^*$. To this end, we fix some $\mu, \nu \in \Lambda$, and we have

\bea
\psi^{t}_s \circ \Phi_{t}  \circ \pi^{\mathcal T}_{t}  (\mathcal T_\mu \mathcal T_\nu^*) &=&\psi^{ t}_s \circ \Phi_{ t} ( t_\mu  t_\nu^*) \\
&=&\psi^{ t}_s \left(  \begin{cases}  t_\mu  t_\nu^* & \text{ if } d(\mu)=d(\nu) \\
0 & \text{ otherwise} \end{cases} \right) \\
&=& \begin{cases} s_\mu s_\nu^* & \text{ if } d(\mu)=d(\nu) \\
0 & \text{ otherwise} \end{cases}\\
&=& \Phi_{s} (s_\mu s_\nu^*) \\
&=&  \Phi_s \circ \pi^{\mathcal T}_s( \mathcal T_\mu \mathcal T_\nu^*)
\eea

\noindent as desired.

We will now show that $\ker \pi^{\mathcal T}_t \subseteq \ker  \pi^{\mathcal T}_s$. To this end, suppose that $x \in \ker \pi^{\mathcal T}_t$. Then $x^*x \in \ker \pi^{\mathcal T}_t$, so

\bea
0 &=&\psi^{t}_s \circ \Phi_{t}  \circ \pi^{\mathcal T}_{t}(x^*x) \\
&=& \Phi_s \circ \pi^{\mathcal T}_s(x^*x) 
\eea

\noindent and since $\pi^{\mathcal T}_s(x^*x) \geq 0$ and $\Phi_s$ is faithful, then we must have that $ \pi^{\mathcal T}_s(x^*x) =0$, so $ \pi^{\mathcal T}_s(x) =0$, and thus $x \in \ker \pi^{\mathcal T}_s$ as desired.

Thus by Lemma \ref{Factors Through Theorem} (the Factors Through Theorem), there exists a map $\pi^t_s$ satisfying $\pi^{\mathcal T}_s= \pi_s^t \circ \pi^{\mathcal T}_t$, from which it is immediate that $\pi^t_s$ is a canonical covering. Thus $s \leq_{rep} t$.

\end{proof}

Now we may prove the main theorem of this section, which is a slight generalization of \cite[Theorem 5.3]{brownlowe2013co} from the context of quasi-lattice ordered groups to weakly quasi-lattice ordered groups. We have also rephrased the result:

\begin{theorem}
\label{Co-Universal Algebra}
Let $(G,P)$ be a WQLO group, and let $\Lambda$ be a finitely aligned $P$-graph. If $s$ is a $\Lambda$-faithful, tight representation with a normal gauge coaction, and $t$ is a $\Lambda$-faithful gauge-coacting representation, then $s \leq t$.

In particular, there is a unique representation $S$ of $\Lambda$ which is $\Lambda$-faithful, tight, and has a normal gauge coaction, and $S \leq t$ for any $\Lambda$-faithful gauge-coacting representation $t$. 

We write $C^*_{min}(\Lambda)$ for $C^*(S)$, and call $C^*_{min}(\Lambda)$ the \ulindex{co-universal} algebra of the graph.

\end{theorem}

\begin{proof}
Suppose that $s$ was a $\Lambda$-faithful, tight representation with a normal gauge coaction, and that $t$ was any $\Lambda$-faithful gauge-coacting representation. Since $s$ is tight and $t$ is $\Lambda$-faithful, by Lemma \ref{Balanced Tight is Minimal}, $s \leq_{bal} t$. Then since $s$ has a normal gauge coaction and $t$ has a gauge coaction, by Lemma \ref{Covering Extension}, $s \leq_{rep} t$.

To show the existence of such an $S$, recall that $f$ denotes the ultrafilter representation from Lemma \ref{Ultrafilter Representation}, which is $\Lambda$-faithful and tight by that lemma. By Corollary \ref{Normal Coactionization}, $S:= \widetilde{f'}$ is a $\Lambda$-faithful, tight representation with a normal gauge coaction.

If $t$ were another $\Lambda$-faithful, tight representation that had a normal gauge coaction, then by the first part of the claim $t \leq_{rep} S$ and $S \leq_{rep} t$, so $t \cong_{rep} S$. That is, $S$ is unique up to canonical isomorphism.

\end{proof}

\subsection{The Tight Algebra}

The following result is a relatively straightforward combination of the properties of $C^*_{min}(\Lambda)$ and the construction of a universal algebra (Lemma \ref{Universal Algebra}).
 
\begin{proposition} \label{Tight Representation Definition}
Let $(G,P)$ be a WQLO group, and $\Lambda$ a finitely-aligned $P$-graph.

\begin{enumerate}
\item There is a tight representation $T$ of $\Lambda$ which is universal for tight representations. We will denote the algebra generated by this representation as $C^*_{tight}(\Lambda)$\index{representation!of a $P$-graph!universal tight}.
\item There is a canonical covering $\pi_S^T: C^*_{tight}(\Lambda) \ra C^*_{min}(\Lambda)$.
\item $T$ is $\Lambda$-faithful. 
\item $T$ has a gauge coaction $\delta$.
\item The gauge coaction $\delta$ on $C^*_{tight}(\Lambda)$ is normal if and only if $\pi_S^T$ is an isomorphism, in which case $C^*_{tight}(\Lambda) \cong C^*_{min}(\Lambda)$.
\item If $\delta$ is normal, there is a gauge-invariant uniqueness theorem of this form: if $t$ is another $\Lambda$-faithful, gauge coacting, tight representation, then $C^*_{min}(\Lambda)$ is canonically isomorphic to $C^*(t)$. 
\end{enumerate}
\end{proposition}

\begin{proof}
For (1), the relators (T1)-(T4) and the tightness condition are all polynomial relations in the generators, and $(T1)$ and $(T3)$ together imply that all the generators are partial isometries, so by Lemma \ref{Universal Algebra}, there is a universal $C^*$-algebra for such tight representations. Let us denote this representation by $T$ and the algebra it generates by $C^*_{tight}(\Lambda)$.

For (2), by Theorem \ref{Co-Universal Algebra}, the representation $S$ which generates $C^*_{min}(\Lambda)$ is a tight representation, so by the universality of $T$, there is a canonical covering $\pi^T_S: C^*_{tight}(\Lambda) =C^*(T) \ra C^*(S)=C^*_{min}(\Lambda)$. 

For (3), by Theorem \ref{Co-Universal Algebra}(1), $S$ is $\Lambda$-faithful, so each $S_\lambda$ is nonzero, and since $\pi^T_S(T_\lambda)=S_\lambda$, then each $T_\lambda$ is nonzero as well.

For (4), let $T'$ denote the coactionization of $T$ given by Proposition \ref{Coactionization}. By part (7) of that proposition, since $T$ is tight, then $T'$ is also tight, and therefore by the universality of $T$, there is a canonical covering $T_\lambda \mapsto T_\lambda'=T_\lambda \otimes U_{d(\lambda)}$, which is the desired gauge coaction. 

For (5), if $S \cong T$, then the gauge coaction on $T$ is normal since the gauge coaction on $S$ is normal. Conversely, if the gauge coaction on $T$ is normal, then by the uniqueness part of Theorem \ref{Co-Universal Algebra}, the canonical covering $\pi_S^T$ is an isomorphism.

For (6), suppose that $\delta$ is normal and $t$ is a $\Lambda$-faithful gauge coacting, tight representation. Then by the universality of $C^*_{tight}(\Lambda) = C^*(T)$, $T \geq_{rep} t$ (here we are using the partial order notation of Lemma \ref{Representation Order}). Similarly, by the co-universality of $C^*_{min}(\Lambda) =C^*(S)$, $t \geq_{rep} S$. Finally, by (5) since $\delta$ is normal, then $T \cong S$. Putting these together, we have that $S \cong T \geq_{rep} t \geq_{rep} S$, so $t \cong S$.

\end{proof}

This is a pleasant result if $\delta$ is normal, but there are examples arising from non-amenable groups where $\delta$ is not normal:

\begin{lemma} \label{GIUT Failure}
Let $(G,P)$ be a group with a total ordering (meaning $G= P \cup P\inv$ in addition to $P \cap P\inv = \{1\}$), and let $\Lambda=P$, thought of as a small category with one object and with degree functor $\id_P$. Then any tight representation $t$ of $\Lambda=P$ extends uniquely to a unitary representation $\bar t$ of $G$ by $\bar t_g = \begin{cases} t_g &\text{ if }g \in P \\ t_{g\inv}^* &\text{ if } g \in P \inv \end{cases}$. 

Therefore, $C^*_{tight}(\Lambda) \cong C^*(G)$, and in particular the gauge coaction on $C^*_{tight}(\Lambda)$ is normal if and only if $G$ is amenable. 
\end{lemma}

\begin{proof}
Let $(G,P)$ be a group with a total ordering, and let $\Lambda=P$, thought of as a small category with the natural degree functor. Let $t$ be a tight representation of $P$. 

We will first show that $t_p$ is a unitary for each $p \in P$. For all $p,q \in P$, $\max(p,q) \geq p,q$, so $p$ and $q$ have a common extension. That is, every $p \in P$ is exhaustive for $P$, so by tightness, we have that $1=t_et_e^*= t_pt_p^*$ for all $p \in P$. Since also $t_pt_p^*=t_{s(p)}=t_e=1$ by the (T3) relator, this says that the $t_p$s are unitaries. 

Now, extend $t$ to $\bar t$ on all of $G=P \cup P\inv$ by $\bar t_g = \begin{cases} t_g &\text{ if }g \in P \\ t_{g\inv}^* &\text{ if } g \in P \inv \end{cases}$, or equivalently $\bar t_p=t_p$ and $\bar t_{p\inv}=t_p^*$ for $p \in P$. We wish to show that $\bar t$ is a representation of all of $G$. To this end, it suffices to check that it is a group homomorphism, and since it extends a representation of $P$, it suffices to check that multiplication is correct for one element in $P$ and one element in $P\inv$. 

To this end, suppose $p,q \in P$. We have several cases depending on which products are positive. If $1 \leq p\inv q$, then

\[ \bar t_{p\inv} \bar t_{q}= t_p^* t_q = t_p^* t_pt_{p \inv q} =t_{p\inv q} = \bar t_{p\inv q} \]

\noindent and if instead $1 \geq p\inv q$, then $1 \leq q\inv p$, so

\[ \bar t_{p\inv} \bar t_{q} =t_p^* t_q = (t_q t_{q\inv p})^* t_q= t_{q\inv p}^* t_q^*t_q=t_{q\inv p}^*=\bar t_{p\inv q}\]

\noindent so in either case $\bar t_{p\inv} \bar t_{q} =\bar t_{p\inv q}$. 

For products the other way, if $p q \inv \geq 1$, then

\[ \bar t_p \bar t_{q\inv}= t_p t_q^* = t_{p q \inv} t_qt_q^* = t_{pq \inv} = \bar t_{pq \inv}  \]

\noindent while if $pq\inv \leq 1$, then $q p\inv \geq 1$, so

\[ \bar t_p \bar t_{q\inv}= t_p t_q^* = t_p (t_{q p\inv} t_p)^* = t_p t_p^* t_{q p\inv}^* = t_{q p \inv}^*= \bar t_{pq\inv} \]

\noindent so in either case $\bar t_p \bar t_{q\inv}= \bar t_{pq \inv}$. Thus $\bar t$ is indeed a representation of $G$. 

To show that $\bar t$ is the unique extension of $t$ to a unitary representation of $G$, observe that if $\tilde t$ were another such representation, then for all $p \in P$, $1= \tilde t_p \tilde t_{p\inv}= t_p \tilde t_{p\inv}$, so since $t_p$ is a unitary we get that $t_p^* = \tilde t_{p\inv}$, so $\tilde t = \bar t$.

Thus the algebra generated by the universal tight representation $T$ is the algebra generated by the universal group representation. That is, $C^*_{tight}(\Lambda)= C^*(G)$, and the gauge coaction on $C^*_{tight}(\Lambda)$ is the standard coaction $\delta_G$ on $C^*(G)$ given in Example \ref{Delta_G}. 

Now for normality, by Lemma \ref{Amenable Coactions are Normal}, if $G$ is amenable, then all of its coactions are amenable, and in particular $\delta_G$ is amenable. Conversely, if the gauge coaction on $C^*_{tight}(\Lambda) \cong C^*(G)$ is normal, then considering $\{\C U_g\}_{g \in \C}$ as a topological grading over $G$ with conditional faithful expectation, by \cite[Proposition 3.7]{exel1996amenability}, $C^*(G)$ is naturally isomorphic to $C^*_r( \{\C U_g\}_{g \in \C})=C^*_r(G)$, so the full and reduced $C^*$-algebras are isomorphic and thus $G$ is amenable. 

\end{proof}

\begin{corollary}
If $(G,P)=(F_2,R)$ where $R$ a total ordering on $G$ such as the Magnus expansion given in Section 3.2 of \cite{clay2016ordered}, then $C^*_{tight}(\Lambda) \not \cong C^*_{min}(\Lambda)$. Since both $C^*_{tight}(\Lambda)$ and $C^*_{min}(\Lambda)$ are $\Lambda$-faithful, tight, gauge coacting representations, there is no gauge-invariant uniqueness theorem for this $R$-graph.

\end{corollary}
\begin{proof}
By the previous lemma, $C^*_{tight}(\Lambda) \cong C^*(G)$, and by Proposition \ref{Tight Representation Definition} $C^*_{tight}(\Lambda)$ is a tight, $\Lambda$-faithful representation of $\Lambda$ with a gauge coaction. Since $G=F_2$ is not amenable, the gauge coaction is not normal. But by part (5) of \ref{Tight Representation Definition}, since the gauge coaction is not normal, then $\pi^T_S$ is not an isomorphism, so $C^*_{min}(\Lambda)$ is another tight, $\Lambda$-faithful representation of $\Lambda$ with a gauge coaction, but which is not canonically isomorphic to $C^*_{tight}(\Lambda)$.

\end{proof}

The reader may recognize the ordered group $(F_2, R)$ from Example \ref{No Amenable Reduction} earlier in this manuscript, where we used it as an example of an ordering that had no reduction to an amenable group. This recognition may ignite a spark of hope: perhaps the existence of a reduction to an amenable group prevents the obstruction seen in the previous lemma and guarantees that $\delta$ is normal. We shall see in the next section that this is indeed the case.

\subsection{Reductions and Representations}

In this section, we develop the notion that if $(G,P)$ reduces to $(H,Q)$, then $P$-graphs are $Q$-graphs and they have the same representations.

\begin{remark}
In this section, we will be considering a category $\Lambda$ as a graph with respect to more than one semigroup (ex: as a $P$-graph and as a $Q$-graph). In cases where there may be confusion, we will write $C^*_{min}(\Lambda, P)$ and $C^*_{tight}(\Lambda, P)$ to indicate the semigroup $P$, and so on for other semigroups. 
\end{remark}

First we will show that if $\Lambda$ is a $P$-graph, and $(G,P)$ has reduction $(H,Q)$, then $\Lambda$ is also a $Q$-graph in a natural way.

\begin{lemma} \label{P Graph is Q Graph}
Let $(G,P)$ be an ordered group, and let $(H,Q)$ be a reduction of $(G,P)$ with degree functor $d^P_Q:P \ra Q$. Then a $P$-graph $\Lambda$ with degree functor $d_P^\Lambda$ is also a $Q$-graph with degree functor $d=d^P_Q \circ d_P^\Lambda$.
\end{lemma}

\begin{proof}
Certainly, $d$ maps $\Lambda$ into $Q$, so, it suffices to check that $\Lambda$ has unique factorization with respect to $d$.

Let $\lambda \in \Lambda$, and $q \in Q$ satisfy $q \leq d(\lambda)$. We must show that there is a unique $\alpha, \beta \in \Lambda$ such that $\lambda =\alpha \beta$ and $d(\alpha)=q$. 

For existence, since $d_P^\Lambda(\lambda) \in P$ is a path of length $d^P_Q(d_P^\Lambda(\lambda))=d(\lambda)\geq q$, then by unique factorization of $d^P_Q$, there exist unique $p, p_1 \in P$ such that $pp_1 =d_P^\Lambda(\lambda)$ and $d^P_Q(p)=q$. Then $p \leq d_P^\Lambda(\lambda)$, so by unique factorization of $d_P^\Lambda$, there exist unique $\alpha, \beta \in \Lambda$ such that $\alpha \beta = p p_1$ and $d_P^\Lambda (\alpha)=p$. Then $d(\alpha)= d^P_Q(d_P^\Lambda(\alpha))= d^P_Q(p)=q$, so the desired $\alpha, \beta$ exist.

For uniqueness, if $\alpha', \beta' \in \Lambda$ were two other paths with $\lambda =\alpha' \beta'$ and $d(\alpha')=q$, then $d_P^\Lambda (\lambda)= d_P^\Lambda(\alpha') d_P^\Lambda(\beta')$ would be a factorization of $d_P^\Lambda(\lambda)$ where the first term has degree $d^P_Q(d_P^\Lambda(\alpha'))=d(\alpha')= q$, so by uniqueness of the factorization of $d^P_Q$, we have that $d_P^\Lambda(\alpha')=p$. Then by the uniqueness of the factorization $d_P^\Lambda$, we have $\alpha'=\alpha$, as desired.

\end{proof}

Next, we will show that whether $\Lambda$ is regarded as a $P$-graph or a $Q$-graph, it has the same representations.

\begin{lemma}
Let $(G,P)$ and $(H,Q)$ be two WQLO groups, and $\Lambda$ a small category with two functors $d^\Lambda_P: \Lambda \ra P$ and $d^\Lambda_Q: \Lambda \ra Q$ such that $\Lambda$ is a $P$-graph with respect to $d^\Lambda_P$ and a $Q$-graph with respect to $d^\Lambda_Q$. Then,

\begin{enumerate}
\item $\Lambda$ is finitely-aligned as a $P$-graph if and only if it is finitely-aligned as a $Q$-graph. 
\item The (T1)-(T4) relators are the same for $\Lambda$ independently of being a $P$-graph or a $Q$-graph. 
\item A representation is tight independently of $\Lambda$ being a $P$-graph or a $Q$-graph. 
\end{enumerate}
\end{lemma}

\begin{proof}
For (1), recall that $MCE(\mu, \nu)$ depends only on the category structure of $\Lambda$ by definition, so it is finite regardless of the degree functor given to $\Lambda$.

Similarly for (2) and (3), the T1-T4 and tightness relations as written only depend on the category structure of $\Lambda$. 

\end{proof}

\begin{corollary} \label{Tight Isomorphism}
Let $(G,P)$ and $(H,Q)$ be two WQLO groups, and $\Lambda$ a small category with two functors $d^\Lambda_P: \Lambda \ra P$ and $d^\Lambda_Q: \Lambda \ra Q$ such that $\Lambda$ is a finitely-aligned $P$-graph with respect to $d^\Lambda_P$ and a finitely-aligned $Q$-graph with respect to $d^\Lambda_Q$. Then $C^*_{tight}(\Lambda, P)$ is canonically isomorphic to $C^*_{tight}(\Lambda, Q)$.
\end{corollary}

\begin{proof}
Let $a$ and $b$ be the universal tight representations of $\Lambda$ as a $P$-graph and $Q$-graph, respectively, so that $C^*_{tight}(\Lambda, P)= C^*(a)$ and $C^*_{tight}(\Lambda, Q)=C^*(b)$.

By the previous Lemma, both the $a_\lambda$s and the $b_\lambda$s satisfy the equivalent T1-T4 and tightness relators, so by their respective universal properties, there are canonical covers taking $a_\lambda \mapsto b_\lambda$ and $b_\lambda \mapsto a_\lambda$, which are inverses, and hence canonical isomorphisms.
\end{proof}

\begin{lemma} \label{Coactions Carry Over}
Let $\varphi: (G,P) \ra (H,Q)$ be a reduction of WQLO  groups. Let $\Lambda$ be a finitely-aligned $P$-graph, and $t$ a representation of $\Lambda$. If $t$ has a gauge coaction by $G$, then it has a gauge coaction by $H$ (when $\Lambda$ inherits the structure of a $Q$-graph as described in \ref{P Graph is Q Graph}). 

In particular, if $\delta$ is the gauge coaction by $G$, then $\epsilon:= (\id_{C^*(t)} \otimes \bar \varphi) \circ \delta$ is the gauge coaction by $H$, where $\bar \varphi:C^*(G) \ra C^*(H)$ is given by $U_g \mapsto V_{\varphi(g)}$ .
\end{lemma}

\begin{proof}
First let us fix some notation. Let $\{U_g\}_{g \in G}$ and $\{V_h\}_{h \in H}$ denote the universal unitary representations of $G$ and $H$ respectively, so that $C^*(G)$ and $C^*(H)$ are the closed span of the $U_g$ and $V_h$, respectively. Let $\delta$ denote the gauge coaction on $C^*(t)$, $d^\Lambda_P$ the degree function on $\Lambda$ as a $P$-graph, and recall that the degree functor on $\Lambda$ as a $Q$-graph was $d^\Lambda_Q = \varphi \circ d^\Lambda_P$.

Then we have that $g \mapsto V_{\varphi(g)}$ is a unitary representation of $G$, so by the universal property of $C^*(G)$, there is a $*$-homomorphism $\bar \varphi:C^*(G) \ra C^*(H)$ given by $U_g \mapsto V_{\varphi(g)}$. We will prove that there is a gauge coaction of $H$ on $C^*(t)$ given by $\epsilon:= (\id_{C^*(t)} \otimes \bar \varphi) \circ \delta$. First note that by Lemma \ref{Tensor of Homomorphisms}, there is such a function $\epsilon$, it is a $*$-homomorphism, and it maps $C^*(t) \ra C^*(t) \otimes C^*(H)$. 

But now observe that 

\bea
\epsilon(t_\lambda) &=& (\id_A \otimes \bar \varphi) \circ \delta (t_\lambda) \\
 &=& (\id_A \otimes \bar \varphi) (t_\lambda \otimes U_{d_P^\Lambda(\lambda)}) \\
 &=& t_\lambda \otimes V_{\varphi(d_P^\Lambda(\lambda))} \\
 &=& t_\lambda \otimes V_{d^\Lambda_Q(\lambda)}
\eea

so by Proposition \ref{Coactionization}(3), $\epsilon$ is a gauge coaction by $H$, as desired. 
\end{proof}

\subsection{Gauge-Invariant Uniqueness for Ordered Groups which Reduce to Amenable Groups}

The following is the mathematical core of our gauge-invariant uniqueness theorem:

\begin{proposition} \label{Uniqueness Theorem}
Let $(G,P)$ be a WQLO group and suppose there is a reduction to an amenable group $\varphi:(G,P) \ra (H,Q)$. Then for any finitely-aligned $P$-graph $\Lambda$, the following are canonically isomorphic:

\begin{enumerate}
\item [a.] $C^*_{tight}(\Lambda, P)$.
\item [b.] $C^*_{tight}(\Lambda, Q)$.
\item [c.] $C^*_{min}(\Lambda, P)$.
\item [d.] $C^*_{min}(\Lambda, Q)$.
\item [e.] $C^*(t)$, where $t$ is any representation of $\Lambda$ as a $P$-graph which is tight, $\Lambda$-faithful, and which has a gauge coaction by $G$.
\end{enumerate}
\end{proposition}

\begin{proof}
For the sake of abbreviation, let us denote the representations generating $C^*_{tight}(\Lambda, P)$, $C^*_{tight}(\Lambda, Q)$, $C^*_{min}(\Lambda, P),$ and $C^*_{min}(\Lambda, Q)$ as $A,$ $B,$ $C,$ and $D$, respectively. Our argument is organized by the following diagram:

\begin{center}
\begin{tikzcd}
C^*(A)= C^*_{tight}(\Lambda, P) \arrow[r,"\pi_B^A"] \arrow[d, "\pi_C^A"]
& C^*(B)= C^*_{tight}(\Lambda, Q) \arrow[d, "\pi_D^B"] \\
C^*(C)= C^*_{min}(\Lambda, P) \arrow[r, "\pi^C_D"]
& C^*(D)= C^*_{min}(\Lambda, Q)
\end{tikzcd}
\end{center}

All arrows here denote canonical coverings.

The map $\pi_B^A$ exists and is an isomorphism by Corollary \ref{Tight Isomorphism}.

The maps $\pi_C^A$ and $\pi_D^B$ are the canonical coverings of Proposition \ref{Tight Representation Definition}(2). Since $H$ is amenable, then by Lemma \ref{Amenable Coactions are Normal} the gauge coaction by $H$ on $C^*_{tight}(\Lambda, Q)$ is normal, and so $\pi_D^B$ is a canonical isomorphism by Proposition \ref{Tight Representation Definition} (5).

Combining these two, we have that $A \cong B \cong D$.

The map $\pi_D^C$ arises from the work established in the previous lemmas: by Lemma \ref{Coactions Carry Over} $C^*(C)$ carries a gauge coaction by $H$. Since also $C$ is $\Lambda$-faithful by \ref{Co-Universal Algebra}, then by the minimality of $C^*_{min}(\Lambda, Q)$, there is a canonical covering $\pi_D^C: C^*_{min}(\Lambda, P) \ra C^*_{min}(\Lambda, Q)$. 

Now, with the order notation of Lemma \ref{Representation Order}, we have that $A \geq_{rep} C \geq_{rep} D \cong B \cong A$, so $A \cong B \cong C \cong D$. That is, (a.)-(d.) are canonically isomorphic.

For (e.), we have that $A \geq_{rep} t \geq_{rep} C \cong A$ by the respective universality and co-universality of $C^*(A)$ and $C^*(C)$, so $A \cong t \cong C$.
\end{proof}

One nice outcome is that this result is independent of the amenable group $(H,Q)$ and the reduction $\varphi$, and simply requires that one exist. In particular, we have a gauge-invariant uniqueness theorem for WQLO groups which have a reduction to an amenable group:

\uniquenesstheorem

\begin{proof}
Let $T$ be the universal tight representation as in Proposition \ref{Tight Representation Definition} (1), so $C^*(T) \cong C^*_{tight}(\Lambda)$. By that theorem, $T$ is $\Lambda$-faithful, tight, and has a gauge coaction. It is also universal for tight representations.

By Proposition \ref{Uniqueness Theorem}, $C^*(T) \cong C^*_{min}(\Lambda)$, so $T$ is co-universal for representations which are $\Lambda$-faithful and have gauge coaction.

To see that $T$ is the unique representation (up to canonical isomorphism) which is $\Lambda$-faithful, tight, and has a gauge coaction, if there were such another representation $t$, it would be covered by $T$ (since $t$ is tight, and $T$ is universal for tight representations) and it would cover $T$ (since $t$ is $\Lambda$-faithful and has a gauge coaction and $T$ is co-universal for such representations). Thus $t\cong T$.

\end{proof}

\begin{remark}
In the case that $(G,P)$ reduces to an amenable group, we believe that the $C^*$-algebra generated by this unique $\Lambda$-faithful, tight, gauge coacting representation deserves the title of \emph{the} $C^*$-algebra of the graph. We will write $C^*(\Lambda)$ for this algebra.
\end{remark}

\subsection{Additional Results in the case of a Strong Reduction}

In this section we will consider what further results can be provided about a $P$-graph algebra under the assumption that $(G,P)$ strongly reduces to an amenable group. 

\begin{lemma} \label{Automatic Normality of Gauge Coactions}
Let $(G,P)$ be a WQLO group, and suppose $(G,P)$ has a strong reduction to an amenable group. Then every gauge coaction on a finitely-aligned $P$-graph is normal. 
\end{lemma}

\begin{proof}
Let $\varphi:(G,P) \ra (H,Q)$ be the strong reduction where $H$ is amenable. Fix some finitely-aligned $P$-graph $\Lambda$ and a representation $s$ of $\Lambda$ which has a gauge coaction by $G$.

Then by Lemma \ref{Coactions Carry Over}, $s$ has a gauge coaction by $H$. For sake of clarity, let $\Phi_s$ and $\Omega_s$ denote the conditional expectations on $C^*(s)$ as in Lemma \ref{Graded Conditional Expectation} which respectively arise from its $G$-coaction and $H$-coaction.

By Lemma \ref{Amenable Coactions are Normal}, since $H$ is amenable, then $\Omega_s$ is faithful. Our desired result is that $\Phi_s$ is faithful.

To this end, it suffices to show that $\Omega_s=\Phi_s$ as maps on $C^*(s)$, and since $C^*(s)= \closedspan\{s_\alpha s_\beta^* : \alpha, \beta \in \Lambda\}$ and both functions are continuous and linear, it suffices to check that $\Omega_s=\Phi_s$ on a single term $s_\alpha s_\beta^*$.

Recall that by Lemma \ref{Graded Conditional Expectation}, we have

\[
\Phi_s(s_\alpha s_\beta^*) = \begin{cases} s_\alpha s_\beta^* & \text{ if }d(\alpha)=d(\beta)  \\ 0 & \text{ if } d(\alpha) \neq d(\beta) \end{cases}
\]

\noindent and

\[
\Omega_s(s_\alpha s_\beta^*) = \begin{cases} s_\alpha s_\beta^* &\text{ if }\varphi(d(\alpha))=\varphi(d(\beta))  \\ 0 & \text{ if } \varphi(d(\alpha)) \neq \varphi(d(\beta)) \end{cases}.
\]

But $\varphi$ is a strong reduction, so for the positive elements $d(\alpha), d(\beta) \in P$, we have $d(\alpha)=d(\beta)$ if and only if $\varphi(d(\alpha))=\varphi(d(\beta))$.

Thus $\Omega_s=\Phi_s$ on each spanning element, so they are equal as functions, and thus $\Phi_s$ is faithful, so $s$ is normal.

\end{proof}
%%%%%

%%%%

\begin{lemma} \label{Toeplitz Criteria}
Let $(G,P)$ be a WQLO group with a strong reduction onto an amenable group, let $\Lambda$ be a $P$-graph, and let $s$ be a representations of $(G,P)$. If $s$ is $\Lambda$-faithful, has a gauge coaction, and every proper bolt in $s$ is nonzero, then $s \cong \mathcal T$.

\end{lemma}

\begin{proof}
Since $\mathcal T$ covers $s$, it suffices to show that $\ker \pi^{\mathcal T}_s=0$.

To this end, consider the diagram

\begin{center}
\begin{tikzcd}
C^*(\mathcal T) \arrow[r, "\Phi_{\mathcal T}"] \arrow[d, "\pi^{\mathcal T}_{s}"] &
 \mathcal B(\mathcal T)\arrow[d, "\psi^\mathcal T_s"]\\
C^*(s)  \arrow[r, "\Phi_s"] &
\mathcal B(s)
\end{tikzcd}
\end{center}

\noindent which commutes by a routine computation on the spanning elements $\mathcal T_\mu \mathcal T_\nu^*$.

Fix some $x \in \ker \pi^{\mathcal T}_s=0$. Then $\pi^{\mathcal T}_s (x^*x)=0$, so 

\[ 0=\Phi_s(\pi^{\mathcal T}_s (x^*x))= \psi^{\mathcal T}_s (\Phi_{\mathcal T}(x^*x)).\]

Now by Theorem \ref{Kernel Generated by Bolts 2}, since $\ker \psi^{\mathcal T}_s$ is generated by the bolts and $\mathcal T_\mu \mathcal T_\mu^*$s it contains, and by the hypotheses on $s$ this kernel contains no proper bolts or $\mathcal T_\mu \mathcal T_\mu^*$, then we have $\ker \psi^{\mathcal T}_s=0$, and thus $\Phi_{\mathcal T}(x^*x)=0$. 

By Lemma \ref{Automatic Normality of Gauge Coactions}, $\Phi_{\mathcal T}$ is faithful, so $x^*x=0$ and thus $x=0$. Thus $\ker \pi_s^{\mathcal T}= \{0\}$, as desired.
\end{proof}

\begin{corollary}
Let $\Lambda$ be a finitely-aligned $\N^2*\N$-graph. Then any gauge coaction by $G=\Z^2*\Z$ is automatically normal.
\end{corollary}

\begin{proof}
Recall that $(\Z^2 *\Z, \N^2*\N)$ strongly reduces to $(\Z^2 \wr \Z, \N^2 \wr\N)$, which is amenable. Then by Lemma \ref{Automatic Normality of Gauge Coactions}, the gauge coaction is normal.

\end{proof}

\begin{lemma} \label{Conditional Expectation is Equivalent to Coaction}
Let $(G,P)$ be a WQLO group and suppose that $(G,P)$ has a strong reduction to an amenable group. Let $\Lambda$ be a finitely-aligned $P$-graph, and let $t$ be a representation of $\Lambda$. Then the following are equivalent:

\begin{enumerate}

\item $t$ has a gauge coaction by $G$
\item There is a bounded linear map $\Phi_t: C^*(t) \ra \mathcal B(t)$ satisfying
\[ \Phi_t (t_\mu t_\nu^*)= \begin{cases} t_\mu t_\nu^* & \text{ if } d(\mu)=d(\nu) \\ 0 &\text{ if } d(\mu) \neq d(\nu) \end{cases} \]
\end{enumerate}

\end{lemma}

\begin{proof}
$(1 \Rightarrow 2)$ If $t$ has a gauge coaction, then such a map $\Phi_t$ exists by Lemma \ref{Balanced Conditional Expectation}.

$(2 \Rightarrow 1)$ If such a $\Phi_t$ exists, then by Theorem 19.1 and Definition 19.2 of \cite{exel2017partial}, writing $B_g= \closedspan\{t_\mu t_\nu^* : d(\mu) d(\nu)\inv = g\}$, we have that $\mathcal B=\{B_g\}_{g \in G}$ are linearly independent and form a topological grading of $C^*(t)$. 

In particular, $\mathcal B$ is a Fell bundle, and we can form the cross-sectional algebra $C^*(\mathcal B)$ as in Remark \ref{Fell Bundle Conditional Expectation}. Each graded component $B_g$ is naturally embedded in $C^*(\mathcal B)$, so for each $\lambda \in \Lambda$, let $s_\lambda$ denote the operator arising from embedding $t_\lambda \in B_{d(\lambda)}$ into $C^*(\mathcal B)$.

But one may immediately check the T1-T4 relators to check that $\{s_\lambda\}_{\lambda \in \Lambda}$ is a representation of $\Lambda$ in $C^*(\mathcal B)$. Note that since the $\{s_\lambda\}_{\lambda \in \Lambda}$ generate each fiber $B_g$, then they generate all of $C^*(\mathcal B)$, so $C^*(\mathcal B)=C^*(s)$.

Our goal is now to show that $s$ has a gauge coaction and that $s \cong t$. For the former, by \cite[Proposition 3.3]{quigg1996discrete}, there is a coaction $\delta$ on $C^*(\mathcal B)$ satisfying

\[ \delta(b)= b\otimes U_g \text{ for all } b \in B_g, g \in G.\]

In particular, since $s_\lambda \in B_{d(\lambda)}$, then $\delta(s_\lambda)=s_\lambda \otimes U_{d(\lambda)}$, so $\delta$ is a gauge coaction on $s$.

Now we must show that $s \cong t$. By \cite[Theorem 19.5]{exel2017partial}, since $C^*(t)$ is a topologically graded $C^*$-algebra, with grading $\mathcal B$ there is a commutative diagram of surjective $*$-homomorphisms

\begin{center}
\begin{tikzcd}
C^*(\mathcal B)=C^*(s) \arrow[rr, "L"]  \arrow[dr, "\pi^s_t"] &&C^*_{r}(\mathcal B) \\
& C^*(t) \arrow[ur, "\psi"]
\end{tikzcd}
\end{center}

\noindent where $L$ denotes the regular representation from  \cite[Definition 17.6]{exel2017partial} (in that definition, this map is called $\Lambda$, but we have changed its name to avoid confusion with the $P$-graph $\Lambda$).

Recall that $C^*(\mathcal B)$ is topologically graded with conditional expectation $\Phi_s$. Since $G$ strongly reduces to an amenable group, then by Lemma \ref{Automatic Normality of Gauge Coactions}, the gauge action $s$ is normal and thus the conditional expectation $\Phi_s$ is faithful. But by \cite[Proposition 19.6]{exel2017partial}, the kernel of the regular representation $L$ is given by

\[ \ker (L)= \{x \in C^*(\mathcal B) : \Phi_s(x^*x)=0\} \]

\noindent and since $\Phi_s$ is faithful, then $\ker L=\{0\}$. Since $L$ is injective and $L= \psi \circ \pi_{s}^t$, then $\pi_s^t$ is injective, so $s \cong t$, so $t$ has a gauge coaction as desired.

\end{proof}

The reader should note that if such a bounded linear map exists, then there is a gauge coaction, so that bounded linear map is the conditional expectation of Lemma \ref{Graded Conditional Expectation}.

The next result shows that if you represent a $P$-graph on a Hilbert space in such a way that the Hilbert space also has a ``$P$-grading'' which interacts with the grading on $\Lambda$ in a natural way, then this representation has a gauge coaction.

\begin{lemma} \label{Graded Hilbert Space}
Let $(G,P)$ be a WQLO group and suppose that $(G,P)$ has a strong reduction to an amenable group. Let $\Lambda$ be a finitely-aligned $P$ graph. Let $t$ be a representation of $\Lambda$ in $\mathcal B(\mathcal H)$ for some Hilbert space $\mathcal H$, and suppose that $\mathcal H= \bigoplus_{p \in P} \mathcal H_p$ such that $t_\mu \mathcal H_p \subseteq \mathcal H_{d(\mu)p}$ for all $\mu \in \Lambda$, $p \in P$. Then $t$ has a gauge coaction.
\end{lemma}

\begin{proof}
By Lemma \ref{Conditional Expectation is Equivalent to Coaction}, it suffices to show that there is a bounded linear map $\Phi_t: C^*(t) \ra \mathcal B(t)$ satisfying

\[ \Phi_t (t_\mu t_\nu^*)= \begin{cases} t_\mu t_\nu^* & \text{ if } d(\mu)=d(\nu) \\ 0 &\text{ if } d(\mu) \neq d(\nu) \end{cases}. \]

For each $p \in P$, let $Q_p \in \mathcal B(\mathcal H)$ denote the projection onto $\mathcal H_p$. Then define 

\[ \Phi_t(x) = \fsum_{\substack{p \in P \\ WOT \text{ limit}} } Q_p x Q_p \]

\noindent which is bounded and converges in the weak operator topology because the $\{Q_p\}_{p \in P}$ are pairwise orthogonal projections. We then wish to show that 

\[ \Phi_t (t_\mu t_\nu^*)= \begin{cases} t_\mu t_\nu^* & \text{ if } d(\mu)=d(\nu) \\ 0 &\text{ if } d(\mu) \neq d(\nu) \end{cases} \]

\noindent for all $\mu, \nu \in \Lambda$. 

To this end, fix some $q \in P$, $h_q \in \mathcal H_q$, and $x \in \mathcal H$. Then we have that

\bea
\langle \Phi_t (t_\mu t_\nu^*) h_q , x \rangle &=& \fsum_{\substack{p \in P \\ WOT \text{ limit}} } \langle Q_p t_\mu t_\nu^* Q_p h_q, x \rangle \\
&=&\langle Q_{q} t_\mu t_\nu^* Q_{q} h_q, x \rangle \\
&=&\langle Q_{q}  t_\mu t_\nu^* h_q,  x \rangle 
\eea

\noindent since $Q_p h_q =0$ if $p \neq q$. Since our choice of $x \in \mathcal H$ was arbitrary, we may unbind the $x$ term, so we have shown that $\Phi_t (t_\mu t_\nu^*) h_q=Q_{q}  t_\mu t_\nu^* h_q$.

But note that $t_\mu t_\nu^* h_q \in \mathcal H_{d(\mu) d(\nu) \inv q}$ and $q=d(\mu) d(\nu) \inv q$ if and only if $d(\mu)=d(\nu)$, so 

\[ \Phi_t (t_\mu t_\nu^*) h_q= Q_{q}  t_\mu t_\nu^* h_q =\begin{cases}
 t_\mu t_\nu^* h_q &\text{ if } d(\mu)=d(\nu) \\ 
 0 &\text{ if } d(\mu) \neq d(\nu) 
\end{cases}. \]

Since this is true for each $q \in P$ and $h_q \in \mathcal H_q$, and the $\mathcal H_q$ span all of $\mathcal H$, we may unbind the $h_q$ to get that  

\[ \Phi_t (t_\mu t_\nu^*) =\begin{cases}
 t_\mu t_\nu^* &\text{ if } d(\mu)=d(\nu) \\ 
 0 &\text{ if } d(\mu) \neq d(\nu) 
\end{cases} \]

\noindent as desired. 

\end{proof}

\begin{corollary} \label{Left Regular is Toeplitz}
Let $(G,P)$ be a WQLO group and suppose that $(G,P)$ has a strong reduction to an amenable group. Let $\Lambda$ be a finitely-aligned $P$-graph. Then $L\cong \mathcal T$, where $L$ denotes the regular representation of $\Lambda$ and $\mathcal T$ denotes the Toeplitz representation of $\Lambda$.

\end{corollary}

\begin{proof}
To show that $L$ is universal, by Lemma \ref{Toeplitz Criteria} it suffices to show that $L$ it is $\Lambda$-faithful, has a gauge coaction, and that every proper bolt is nonzero. By Remark \ref{Regular Representation Properties}, $L$ is $\Lambda$-faithful and its proper bolts are nonzero. 

To show that $L$ has a gauge coaction, we will appeal to Lemma \ref{Graded Hilbert Space}. Recall that $L$ is a representation in $\mathcal B(\ell^2(\Lambda))$ by 

\[ 
L_\mu e_\alpha= \begin{cases}
e_{\mu \alpha} & \text{ if } s(\mu)=r(\alpha) \\
0 & \text{ if } s(\mu) \neq r(\alpha)
\end{cases}. \]

Therefore, writing $\Lambda^p= \{\mu \in \Lambda: d(\mu)=p\}$, we can see that $\ell^2(\Lambda) = \bigoplus_{p \in P} \ell^2(\Lambda^p)$, and a simple calculation shows that $L_\mu \ell^2(\Lambda^p) \subseteq \ell^2(\Lambda^{d(\mu)p})$. Therefore, by Lemma \ref{Graded Hilbert Space}, $L$ has a gauge coaction.

Thus by by Lemma \ref{Toeplitz Criteria}, $L \cong \mathcal T$.

\end{proof}

The following result shows that if $\Lambda$ has no ``infinite paths'', then the ultrafilter representation is the universal tight representation.

\begin{lemma} \label{Ultrafilter is Universal}
Let $(G,P)$ be a WQLO group and suppose that $(G,P)$ has a strong reduction to an amenable group. Let $\Lambda$ be a finitely-aligned $P$-graph. Recall that $\Lambda$ has a partial order given by $\alpha \leq \beta$ if $\beta \in \alpha \Lambda$. Suppose that with respect to this partial order, every chain in $\Lambda$ has an upper bound in $\Lambda$. Then the ultrafilter representation $f$ of $\Lambda$ is the universal tight representation of $\Lambda$.

\end{lemma}

\begin{proof}
By our gauge-invariant uniqueness theorem for $P$-graphs, it suffices to show that $f$ is $\Lambda$-faithful, tight, and has a gauge coaction. The former two properties are shown in Lemma \ref{Ultrafilter Representation}. It then suffices to use Lemma \ref{Graded Hilbert Space} to show that $f$ has a gauge coaction.

To this end, we wish to show that the ultrafilters in $f$ correspond to maximal paths in $\Lambda$. More formally, for each path $\alpha \in \Lambda$, let $U_\alpha = \{\mu \in \Lambda: \mu \leq \alpha\}$. It is straightforward to check that $U_\alpha$ is a filter, and that it is an ultrafilter if and only if $\alpha$ is maximal. We wish to show that all ultrafilters in $\Lambda$ are of the form $U_\alpha$.

Fix some ultrafilter $U$. If $U$ is finite, it must contain a maximal element $\alpha$, in which case $U=U_\alpha$. If instead $U$ is infinite, since $U \subseteq \Lambda$ is countable, we may enumerate $U= \{\alpha_n\}_{n \in \N}$. Inductively define a sequence $\{\beta_n\}_{n \in \N} \subseteq U$ by $\beta_1=\alpha_1$, and for $n>1$, let $\beta_n$ be a common upper bound of $\beta_{n-1}$ and $\alpha_n$ which is guaranteed to exist by the $F2$ property of filters. Note that $U= \bigcup_{n=1}^\infty U_{\beta_n}$. 

Then $\{\beta_n\}_{n \in \N}$ is a chain in $\Lambda$, so by our hypothesis, it has an upper bound $\beta \in \Lambda$. Then $U_\beta \supseteq U_{\beta_n}$ for all $n$, and since $U= \bigcup_{n=1}^\infty U_{\beta_n}$, then $U_\beta \supseteq U$, and since $U$ is an ultrafilter, $U=U_\beta$, as desired.

Therefore, we can partition the set of ultrafilters by the degree of their unique maximal element. That is, writing $\widehat \Lambda_\infty^p = \{U_\alpha \in \widehat \Lambda_\infty: d(\alpha)=p\}$, we have $\widehat \Lambda_\infty= \bigsqcup_{p \in P} \widehat \Lambda_\infty^p$, and therefore 

\[ \ell^2(\widehat \Lambda_\infty) = \bigoplus_{p \in P} \ell^2(\widehat \Lambda_\infty^p) .\]

Finally, a simple calculation shows that $f_\mu \ell^2(\widehat \Lambda_\infty^p) \subseteq \ell^2(\widehat \Lambda_\infty^{d(\mu) p})$, so by Lemma \ref{Graded Hilbert Space}, $f$ has a gauge coaction, as desired.

\end{proof}

\begin{example}
Let $G= BS(1,2) = \langle c, t | tc=c^2t \rangle$, and let $P=\{c,t\}^*$. Note that $BS(1,2)$ is amenable, and therefore has a strong reduction to an amenable group (itself). Define a $P$-graph $\Lambda$ as in this diagram:

\begin{center}
\begin{tikzcd}
v_1  \arrow[d, "t_1"] \arrow[rr, "c_1"] &&
v_2 \arrow[d, "t_2"]\\
v_3 \arrow[r, "c_2"]  &
v_4 \arrow[r, "c_3"] &
v_5\\
\end{tikzcd}
\end{center}

\noindent so $\Lambda =\{v_1,v_2,v_3,v_4,v_5, c_1,c_2,c_3, t_1,t_2, c_2t_1, c_3c_2, c_3c_2t_1=c_1t_2\}$, where the range and source maps are as indicated in the diagram, the operation is concatenation of paths, and the degrees are given by $d(v_i)=1$, $d(c_i)=c$, $d(t_i)=t$, and so on.

Since $\Lambda$ is finite, all chains have an upper bound, so by Lemma \ref{Ultrafilter is Universal}, the ultrafilter representation is the universal tight representation and the set of ultrafilters is in bijection with the maximal paths in $\Lambda$. Thus $\widehat \Lambda_\infty = \{v_1, c_1, t_1, c_2t_1, c_3c_2t_1=t_2c_1\}$, so $C^*(f)$ is represented on a 5-dimension vector space and $C^*(f) \subseteq M_5$, the space of 5-by-5 matrices. A direct calculation with respect to the basis $\{ e_{v_1}, e_{c_1}, e_{t_1}, e_{c_2t_1}, e_{c_3c_2t_1}=e_{t_2c_1}\}$ shows that

\bea
f_{v_1} &=& E_{11} \\
f_{c_1} &=& E_{21} \\
f_{t_1} &=& E_{31} \\
f_{c_2t_1} &=& E_{41} \\
f_{t_2c_1} &=& E_{51} \\
\eea

\noindent where $E_{ij}$ denotes the matrix consisting of all 0s except for a single 1 in the $i$th row and $j$th column (a matrix unit). Since these matrix units generate $M_5$, then $C^*(f) = C^*(\Lambda)=M_5$.

\end{example}

It is notable that the results in this section apply only to groups with strong reductions, instead of any reduction. The only ``bottleneck'' for this is Lemma \ref{Automatic Normality of Gauge Coactions}, which says that gauge coactions are normal when $(G,P)$ strongly reduces to an amenable group. We conjecture that a (non-strong) reduction to an amenable group would suffice:

\begin{conjecture}
Let $(G,P)$ be a WQLO group, let $\Lambda$ be a finitely aligned $P$-graph, and suppose $(G,P)$ reduces to an amenable group. Then every gauge coaction on a representation of $\Lambda$ is normal.

Therefore, the other results in this section only require the hypothesis that $(G,P)$ reduces to an amenable group.
\end{conjecture}

\subsection{Applications to Kirchberg Algebras in the UCT Class}

\begin{definition}
A $C^*$-algebra $A$ is called a \ulindex{Kirchberg algebra} if it is purely infinite, simple, nuclear, and separable.

A $C^*$-algebra is said to be in the \ulindex{UCT class} if it satisfies the hypotheses of the universal coefficient theorem. (For a more thorough treatment, the reader is directed to the introduction of \cite{barlak2020cartan}). 
\end{definition}

\begin{proposition} \label{Kirchberg UCT Algebras Result}
Let $A$ be a Kirchberg algebra in the UCT class. Then there is a finitely aligned $\N^2 * \N$-graph such that $A$ is stably isomorphic to $C^*(\Lambda)$.
\end{proposition}

\begin{proof}
By \cite[Corollary 6.3]{brownlowe2013co}, there is a finitely aligned $\N^2 * \N$-graph $\Lambda$ such that $A$ is stably isomorphic to $C^*_{min}(\Lambda)$. But since $\N^2*\N$ strongly reduces to an amenable group by Corollary \ref{N^2*N Reduces}, then $C^*_{min}(\Lambda)=C^*(\Lambda)$ by the gauge-invariant uniqueness theorem for $P$-graphs (Theorem \ref{uniquenesstheorem}). Thus $A$ is stably isomorphic to $C^*(\Lambda)$.
\end{proof}

\section{Appendix 1: Tight Representations}

In this section, our goal is to justify our use of the term ``tight'' in the context of representations of $P$-graphs by showing that a representation is tight in the sense of Definition \ref{Representation Terminology}(c) if and only if it corresponds to a tight representation of a semilattice in the sense of Exel's definition in \cite{exel2009tight}.

In this section, we fix $(G,P)$ a WQLO group, $\Lambda$ a finitely-aligned $P$-graph, and $t: \Lambda \ra C^*(t)$ a representation. We will often write $q_\lambda =  t_\lambda t_\lambda^*$.

It can be easy to be confused about what symbols denote paths, sets of paths, and sets of sets of paths. To avoid confusion, we will try to stick to the convention that greek letters like $\alpha$ denote paths in $\Lambda$, lower case letters (and $\mu \Lambda$ where $\mu \in \Lambda$) denote subsets of $\Lambda$, and capital letters denote collections of subsets.

\subsection{Background}

Recall the following definition of the tightness from \ref{Representation Terminology}(c), slightly modified to keep with our notation for this section: 

\begin{definition}
A representation $t$ of a graph $\Lambda$ is \ulindex{tight} if for every $\mu \in \Lambda$ and finite $z \subseteq \mu \Lambda$ which is exhaustive for $\mu \Lambda$, we have that

 \[ \fprod_{\alpha \in z} (q_\mu -q_\alpha )=0 \] 
 
\noindent where $q_\lambda =  t_\lambda t_\lambda^*$.
 
\end{definition}

This condition was referred to a tightness before, but in this section will be referred to as \ul{H-tightness}\index{tight!H-tightness} to avoid confusion. 

The following terminology is from \cite[Section 2]{exel2009tight}:

\begin{definition}
In \cite{exel2009tight}'s context, a partially ordered set contains a smallest element, denoted $0$. A semilattice is a partially ordered set where for all $x,y \in X$, the set $\{z \in X: z\leq x,y\}$ has a maximum element, denoted $x \wedge y$. That is, $x \wedge y$ is the infimum of $x$ and $y$, and every two elements have an infimum (although it may be $0$).

Given, $x,y \in X$, we say $x$ and $y$ are disjoint and write $x \perp y$ if there is no nonzero $z \in X$ with $z \leq x,y$. Otherwise we say that $x$ and $y$ intersect, and write $x \Cap y$.

\end{definition}

\begin{definition}
Let $E$ be a semilattice. Given finite $X, Y \subseteq E$, let $E^{X,Y}$ denote the subset of $E$ given by

\[ E^{X,Y}= \{z \in E: z \leq x \ \forall x\in X, z \perp y \ \forall y \in Y\}.\]

Given any subset $F \subseteq E$, we shall say that a subset $Z \subseteq F$ is a \ul{cover}\index{cover!of a semilattice} for $F$ if, for every nonzero $x \in F$ there exists $z \in Z$ such $x$ and $z$ have a common extension.
\end{definition}

\begin{definition}
Let $\sigma:E \ra \mathcal B$ be a representation of a semilattice $E$ in a Boolean algebra $\mathcal B$. We shall say that $\sigma$ is \ulindex{tight} if for every finite subsets $X, Y \subseteq E$ and for every finite cover $Z$ for $E^{X,Y}$, one has that

\[ \bigvee_{z \in Z} \sigma(z) \geq \bigwedge_{x \in X} \sigma(x) \wedge \bigwedge_{y \in Y} \neg \sigma(y).\]

Since the reverse inequality is always true, tightness is equivalent to equality in this expression.
\end{definition}

To avoid confusion with H-tightness, we will refer to this notion of tightness as \ul{E-tightness}\index{tight!E-tightness}.

\subsection{Building our Semilattice}

We'll now define our semilattice $E$ and develop its basic properties.

\begin{remark}
Given $\Lambda$, let $E(\Lambda)= \{ \bigcup_{i=1}^n \alpha_i \Lambda : \alpha_i \in \Lambda\}$ denote the collection of finite unions of $\alpha \Lambda$s, ordered under containment. When $\Lambda$ is clear from context, we will write $E$ for $E(\Lambda)$. The empty set serves as the zero element for $E$, and two elements $a,b$ of $E$ have a least lower bound $a \cap b$, which is also in $E$ by finite alignment of $\Lambda$. Thus $E$ is a semilattice.

Recall that $\Lambda$ is given a partial ordering $\leq$ by $\alpha \leq \beta$ if and only if $\beta =\alpha \alpha_1$. Note that $\alpha \leq \beta$ if and only if $\beta \Lambda \subseteq \alpha \Lambda$.

For $a \in E$, we let $m(a)=\{\alpha \in a: \alpha \text{ minimal in } a\}$ denote the set of maximal elements of $a$ (with respect to the ordering on $\Lambda$). For any $a \in E$, we may write $a= \bigcup_{i=1}^n \alpha_i \Lambda$ for some $\alpha_i \in \Lambda$, so it follows that $m(a) \subseteq \{\alpha_i\}_{i=1}^n$, so $m(a)$ is finite. Note also that $m(a)$ is empty if and only if $a$ is empty. Finally, note that $a= \bigcup_{\alpha \in m(a)} \alpha \Lambda$. So in this way, $m$ provides a unique representation of each element of $E$, and provides a bijection between $E$ and the set of finite, pairwise incomparable subsets of $\Lambda$.
\end{remark}

If $t$ be a representation of $\Lambda$, let $q_\alpha= t_\alpha t_\alpha^*$ for all $\alpha \in \Lambda$, and let $\mathcal B_t \subseteq C^*(t)$ be the Boolean algebra generated by $1$ and $\{q_\alpha : \alpha \in \Lambda\}$, where $q_\alpha \wedge q_\beta =q_\alpha q_\beta$.

Finally, given such a representation $t$, define $\sigma_t: E \ra \mathcal B_t$ by 

\[ \sigma_t: a= \bigcup_{\alpha \in m(a)} \alpha \Lambda \mapsto \bigvee _{\alpha \in m(a)} t_\alpha t_\alpha^*=\bigvee _{\alpha \in m(a)} q_\alpha. \] When the representation $t$ is clear from context, we will write $\sigma$ and $\mathcal B$ for $\sigma_t$ and $\mathcal B_t$, respectively. 

\subsection{Technical Lemmas}

\begin{lemma} [Intersection Criterion]
If $a,b \in E$, then $a \Cap b$ if and only if there are $\alpha \in m(a), \beta \in m(b)$ such that $\alpha \Cap \beta$ (in the sense of having a common extension in $\Lambda$).

\end{lemma}

\begin{proof}
Observe that $a \cap b= \bigcup_{\alpha \in m(a), \beta \in m(b)} \alpha \Lambda \cap \beta \Lambda$, so the lefthand side is nonempty if and only if one of the terms on the righthand side are nonempty. But $\alpha \Lambda \cap \beta \Lambda$ is nonempty if and only if $\alpha$ has a common extension with $\beta$.
\end{proof}

\begin{lemma} \label{Basic Cover}
Let $Z$ be a finite cover of $E^{X,Y}$. Then there is another finite cover $Z'$ of $E^{X,Y}$ such that every element of $Z'$ is of the form $\alpha \Lambda$, and $\bigvee_{z \in Z} z = \bigvee_{z' \in Z'} z'$.
\end{lemma}

\begin{proof}
Let $B= m( \bigcup_{z \in Z} z)$ and let $Z'= \{\beta \Lambda : \beta \in B\}$. Then by construction $B$ and hence $Z'$ are finite sets, and by our initial remark on $m$, we know that $\bigcup_{z \in Z} z=\bigvee_{z \in Z} z$ can be reconstituted from its minimal elements in the sense that $\bigvee_{z \in Z} z = \bigvee_{z' \in Z'} z'$. 

It then suffices to show that $Z'$ is a cover of $E^{X,Y}$. To first show that it is contained in $E^{X,Y}$, we know that if $\beta \Lambda \in Z'$, then $\beta \in B$, so $\beta$ must be minimal in some $z_0 \in Z$. Then for all $x \in X$ and $y \in Y$, $\beta \Lambda \leq z_0 \leq x$, and $\beta\Lambda \wedge y \leq z_0 \wedge y =0$, so $\beta \Lambda \leq x$ and $\beta \Lambda \perp y$ for all $x \in X, y \in Y$. Thus $Z' \subseteq E^{X,Y}$.

To show it covers, fix some nonzero $w \in E^{X,Y}$. Then since $Z$ covers $E^{X,Y}$, there is a $z \in Z$ such that $z \Cap w$. Then by the intersection criterion, there is a $\mu \in m(z), \nu \in m(w)$ such that $\mu \Cap \nu$. Since $\mu \in m(z)$, then $\mu \Lambda \leq \bigcup_{\beta \in B} \beta \Lambda$, so there is some $\beta \in B$ with $\beta \leq \mu$. Then $\beta \Cap \nu$, so $\beta \Lambda \Cap w$. Thus $Z'$ is a finite cover.

\end{proof}

\begin{lemma} \label{Y is a subset}
Let $X, Y \subseteq E$ be finite sets. Then there is a finite set $Y' \subseteq E$ such that $E^{X, Y}=E^{X,Y'}$ and $Y' \subseteq E^{X, \emptyset}$.
\end{lemma}

\begin{proof}
By the remark in \cite{exel2009tight} following Definition 2.5, we may assume without loss of generality, that $X=\{x\}$, a singleton. Let $Y'= \{ y \wedge x :x \in X\}$ which is finite. Then by construction, $y'=y \wedge x \leq x$ for all $y' \in Y'$, so $Y' \subseteq E^{X, \emptyset}$. It then suffices to show that $E^{X, Y}=E^{X,Y'}$.

If $w \in E^{X, Y}$, then $w \leq x$ and for all $y\in Y$, $w \perp y$, so $w \wedge y=0$, and thus $w \wedge y'= w \wedge y \wedge x =0$  for all $y' \in Y'$. Thus $w \in E^{X,Y'}$.

Conversely, if $w \in E^{X, Y'}$, then $w \leq x$ and for all $y' \in Y$, $w \perp y'$, so $w \wedge y \wedge x=0$. But $w \leq x$, so $w \wedge y \wedge x= (w\wedge x) \wedge y= w \wedge y=0$, so $w \in E^{X,Y}$.

Thus $E^{X, Y}=E^{X,Y'}$ as desired.
\end{proof}

\begin{lemma} \label{Join Product Equivalence}
Let $\mu \in \Lambda$ and $B \subseteq \mu \Lambda$ be finite. Then $q_\mu = \bigvee_{\beta \in B} q_\beta$ if and only if $0 = \fprod_{\beta \in B} (q_\mu -q_\beta)$.
\end{lemma}

\begin{proof}
The result follows from a chain of equivalences:

\bea
q_\mu &=& \bigvee_{\beta \in B} q_\beta  \\
0 &=& q_\mu- \bigvee_{\beta \in B} q_\beta \text{ by subtraction} \\
0 &=& \bigwedge_{\beta \in B} (q_\mu- q_\beta) \text{ by De Morgan's laws} \\
0 &=& \fprod_{\beta \in B} (q_\mu- q_\beta) \text{ by definition of } \bigwedge.
\eea
\end{proof}

\subsection{The Proof of Equivalence}

Here we prove the following proposition. The main interest is the equivalence of the first and last criteria, and the middle statements are there to ease the proof.

\begin{proposition}
If $t$ is a representation of $\Lambda$, and $\sigma=\sigma_t$ is the associated representation of the semilattice $E=E(\Lambda)$, then the following are equivalent:

\begin{enumerate}
\item $\sigma$ is a tight representation in the $E$-sense.
\item $\sigma$ is a tight representation in the $E$-sense when $Y= \emptyset$.
\item $\sigma$ is a tight representation in the $E$-sense when $Y=\emptyset$ and $X=\{\mu \Lambda\}$.
\item $t$ is a tight representation in the $H$-sense.
\end{enumerate}

\end{proposition}

\begin{proof}
Certainly $(1) \Rightarrow (2) \Rightarrow (3)$. We will prove four non-obvious directions: $(3) \Rightarrow  (4), (4) \Rightarrow  (3), (3) \Rightarrow  (2)$ and $(2) \Rightarrow  (1)$.

$(3) \Rightarrow  (4)$ Fix some $\mu \in \Lambda$. We wish to show that if $B \subseteq \mu \Lambda$ is exhaustive, then $Z_B= \{\beta \Lambda: \beta \in B\}$ is a finite cover of $E^{\{\mu \Lambda\}, \emptyset}$. In particular, given such an exhaustive $B$, if $w \in E^{\{\mu \Lambda\}, \emptyset}$ is nonzero, then taking $\alpha \in m(w)$, $\alpha \in w \leq \mu \Lambda$, so $\alpha \in \mu \Lambda$, so since $B$ is exhaustive there is a $\beta \in B$ such that $\alpha \Cap \beta$, and by the intersection criterion, $\beta \Lambda \Cap w$. Thus $Z_B$ covers.

Then by (3), we have that $\bigvee_{z \in Z_B} \sigma(z)= \sigma(\mu \Lambda)$ and substituting $\sigma(z)=\sigma(\beta \Lambda)=q_\beta$ and $\sigma(\mu \Lambda)=q_\mu$, we have

\[ \bigvee_{\beta \in B} q_\beta = q_\mu\]

\noindent which is equivalent to H-tightness by Lemma \ref{Join Product Equivalence}.

Given such an exhaustive $B$, if $w \in E^{\{\mu \Lambda\}, \emptyset}$ is nonzero, then taking $\alpha \in m(w)$, $\alpha \in w \leq \mu \Lambda$, so $\alpha \in \mu \Lambda$, so since $B$ is exhaustive there is a $\beta \in B$ such that $\alpha \Cap \beta$, and by the intersection criterion, $\beta \Lambda \Cap w$. Thus $Z_B$ covers.

$(4) \Rightarrow (3)$ Fix some $\mu \in \Lambda$. We wish to show that if $Z \subseteq E^{\{\mu \Lambda\}, \emptyset}$ is a finite cover, then $B_Z= \bigcup_{z \in Z} m(z)$ is a finite exhaustive subset of $\mu \Lambda$. In particular, given $\alpha \in \mu \Lambda$, we have that $\alpha \Lambda \in E^{\{\mu \Lambda\}, \emptyset}$, so there is $z \in Z$ such that $z \Cap \alpha \Lambda$. Then by the intersection property, there is a $\beta \in m(z)$ such that $\beta \Cap \alpha$. Thus $B_Z$ is exhaustive for $\mu \Lambda$, as desired.

Then, by (4), and Lemma \ref{Join Product Equivalence}, we have that

\[  q_\mu=\bigvee_{\beta \in B_Z} q_\beta.\]

Substituting $\sigma(\mu \Lambda)= q_\mu$ and $\sigma(\beta \Lambda)= q_\beta$, then regrouping the terms on the right, we have that

\bea
\sigma(\mu \Lambda)&=&\bigvee_{\beta \in B_Z} \sigma(\beta \Lambda)\\
&=& \bigvee_{z \in Z} \bigvee_{\beta\in m(z)} \sigma(\beta \Lambda) \\
&=& \bigvee_{z \in Z}  \sigma(\bigvee_{\beta\in m(z)}\beta \Lambda) \\
&=& \bigvee_{z \in Z}  \sigma(z) \\
\eea

\noindent as desired.

$(3 \Rightarrow 2)$ Let $X \subseteq E$ and let $Z$ be a finite cover of $E^{X,\emptyset}$. By the remark in \cite{exel2009tight} following Definition 2.5, we may assume without loss of generality that $X$ is a singleton, since otherwise we may replace $X$ with $\{x_{min}\}$ where $x_{min}=\bigwedge_{x \in X} x$.

Let $Z$ be a finite cover of $E^{X, \emptyset}$. By Lemma \ref{Basic Cover}, we may assume without loss of generality that $Z = \{\beta_1 \Lambda , ...,\beta_n \Lambda\}$ for some $\beta_1,...,\beta_n \in  \Lambda$. Let $B=\{\beta_1,...\beta_n\}$ and $A=m(x)$.

Now, for each $\alpha \in A$ and $\beta \in B$, let $Z_{\alpha, \beta} = \{\gamma \Lambda: \gamma \in MCE(\alpha, \beta)\}$. Note that $Z_{\alpha, \beta}$ is finite and $\gamma \Lambda \in Z_{\alpha, \beta}$ is in $\alpha \Lambda \cap \beta \Lambda$.

We wish to show that for fixed $\beta \in B$, $\bigcup_{\alpha \in A} Z_{\alpha, \beta}$ is a cover for $E^{\{\beta \Lambda\}, \emptyset}$ and for fixed $\alpha \in A$, that $\bigcup_{\beta \in B} Z_{\alpha, \beta}$ is a cover for $E^{\{\alpha \Lambda\}, \emptyset}$. Each is a finite union of finite sets, hence finite, and is contained in the appropriate space. It then suffices to show that they are covering.

For the former set, fix some nonzero $w \in E^{\{\beta \Lambda\}, \emptyset}$. Since $w$ is nonzero, there is some $\mu \in m(w)$. Then $\mu \Lambda \leq w \leq \beta \Lambda \leq x$, so $\mu \Lambda \wedge x = \mu \Lambda <\infty$, so $\mu \Lambda \Cap x$, and thus there is $\alpha \in m(x)$ such that $\mu \Cap \alpha$. That is, $\alpha$ and $\mu$ have a common extension $\nu$. Since $\beta$ is a prefix of $\mu$, then $\nu$ is a common extension of $\alpha$ and $\beta$ as well. Thus some prefix of $\nu$ is a $\gamma \in MCE(\alpha, \beta)$, so $\nu$ is a common extension of both $\gamma$ and $\mu$. That is, $\gamma \Lambda \Cap w$, and $\gamma \Lambda \in Z_{\alpha, \beta}$. Thus $\bigcup_{\alpha \in A} Z_{\alpha, \beta}$ is a cover for $E^{\{\beta \Lambda\}, \emptyset}$, as desired.

For the latter set, fix some nonzero $w \in E^{\{\alpha \Lambda\}, \emptyset}$. Then we know that $w \leq \alpha \Lambda \leq x$, so $w \in E^{X, \emptyset}$, and because $Z$ is a cover, there is some $\beta \Lambda \in Z$ with $w \Cap \beta \Lambda$. Then by the intersection criterion, there is a $\mu \in m(w)$ such that $\mu \Cap \beta$. That is, $\mu$ and $\beta$ have a common extension $\nu$. However, $\alpha \leq \mu$ as $w \in E^{\{\alpha \Lambda\}, \emptyset}$, so $\nu$ is a common extension of $\beta$ and $\alpha$, and thus has a prefix which is some $\gamma \in MCE(\beta, \alpha)$. Thus $\nu$ is a common extension of $\mu$ and $\gamma$, so $w \Cap \gamma \Lambda$, and $\gamma \Lambda \in Z_{\alpha, \beta}$. Thus $\bigcup_{\beta \in B} Z_{\alpha, \beta}$ is a cover for $E^{\{\alpha \Lambda\}, \emptyset}$, as desired.

Now, using (3), the fact that these two sets cover gives us that

\bea
\sigma(\beta \Lambda)&=& \bigvee_{\gamma \in \bigcup_{\alpha \in A} Z_{\alpha, \beta}} \sigma(\gamma \Lambda ) \\
&=& \bigvee_{\alpha \in A} \bigvee_{\gamma \in MCE(\alpha, \beta) }\sigma(\gamma \Lambda) 
\eea
and 
\bea
\sigma(\alpha \Lambda)&=& \bigvee_{\gamma \in \bigcup_{\beta \in B} Z_{\alpha, \beta}} \sigma(\gamma \Lambda ) \\
&=& \bigvee_{\beta \in B} \bigvee_{\gamma \in MCE(\alpha, \beta)}  \sigma(\gamma \Lambda ) \\
\eea

Finally, using the fact that $x= \bigcup_{\alpha \in A} \alpha \Lambda$, we have that

\bea
\sigma(x) &=& \bigvee_{\alpha \in A} \sigma(\alpha \Lambda) \\
&=& \bigvee_{\alpha \in A} \left( \bigvee_{\beta \in B} \bigvee_{\gamma \in MCE(\alpha, \beta)}  \sigma(\gamma \Lambda )  \right)\\
&=&  \bigvee_{\beta \in B} \left( \bigvee_{\alpha \in A}  \bigvee_{\gamma \in MCE(\alpha, \beta)}  \sigma(\gamma \Lambda ) \right) \\
&=&  \bigvee_{\beta \in B} \sigma(\beta \Lambda) \\
&=&  \bigvee_{z \in Z} \sigma(z)
\eea

which is the E-tightness condition for $E^{X, \emptyset}$.

$(2 \Rightarrow 1)$ Let $X, Y \subseteq E$ be finite, and let $Z$ be a finite cover of $E^{X,Y}$. By Lemma \ref{Y is a subset}, we may assume that $Y$ is a subset of $E^{X, \emptyset}$. Let $Z'= Z \cup Y$.  We claim that $Z'$ is a finite cover of $E^{X, \emptyset}$. In particular, given $w \in E^{X,\emptyset}$, if $w \perp y$ for all $y\in Y$, then $w \in E^{X,Y}$, so $w \Cap z$ for some $z \in Z \subseteq Z'$. If instead $w \not \perp y$ for some $y$, then $w \Cap y$ for some $y \in Y \subseteq Z'$.

Then by (2), we know that

\[
\bigvee_{z' \in Z'} \sigma(z') = \bigwedge_{x \in X} \sigma(x) 
\]

and meeting with $\bigwedge_{y \in Y} \neg \sigma(y)$, we have that

\[
\left( \bigvee_{z' \in Z'} \sigma(z')\right) \wedge \bigwedge_{y \in Y} \neg \sigma(y)= \bigwedge_{x \in X} \sigma(x) \wedge \bigwedge_{y \in Y} \neg \sigma(y)
\]

Simplifying the lefthand side, we have that

\bea
\left( \bigvee_{z' \in Z'} \sigma(z')\right) \wedge \bigwedge_{y \in Y} \neg \sigma(y) &=&\left( \bigvee_{z \in Z}\sigma(z) \vee \bigvee_{y' \in Y} \sigma(y') \right) \wedge \bigwedge_{y \in Y} \neg \sigma(y)\\
&=&\left( \bigvee_{z \in Z}\sigma(z) \wedge \bigwedge_{y \in Y} \neg \sigma(y) \right) \vee\left( \bigvee_{y' \in Y} \sigma(y') \wedge \bigwedge_{y \in Y} \neg \sigma(y) \right)\\
&=&\left( \bigvee_{z \in Z}\sigma(z) \wedge \bigwedge_{y \in Y} \neg \sigma(y) \right)\\
&=&\bigvee_{z \in Z} \left(\sigma(z) \wedge \bigwedge_{y \in Y} \neg \sigma(y) \right)\\
\eea

Now, for a fixed $z \in Z, y \in Y$, and for any $\alpha \in m(z)$, $\beta \in m(y)$, we have that $\alpha \perp \beta$, so $q_\alpha q_\beta=0$. Then $\sigma(\alpha \Lambda) \wedge \neg \sigma(\beta \Lambda) = q_\alpha (1-q_\beta)=q_\alpha$, so $\sigma(z) \vee \neg \sigma(y)=\sigma(z)$. Thus

\[
\bigvee_{z \in Z} \left(\sigma(z) \wedge \bigwedge_{y \in Y} \neg \sigma(y) \right)= \bigvee_{z \in Z} \sigma(z)
\]

\noindent so by transitivity, 

\[ 
\bigvee_{z \in Z} \sigma(z) = \bigwedge_{x \in X} \sigma(x) \wedge \bigwedge_{y \in Y} \neg \sigma(y)
\]

\noindent as desired. This completes the proof that (2) implies (1).

\end{proof}

\bibliography{References}{}

\printindex

\end{document}